\documentclass[a4paper,11pt]{article}

\usepackage[british,UKenglish,USenglish,american]{babel}
\usepackage{amsmath}
\usepackage{amsfonts}
\usepackage{amsthm}
\usepackage{amssymb}
\usepackage{hyperref}
\usepackage{graphics}
\usepackage{makeidx}
\usepackage{fancyhdr}
\usepackage{bm} 
\usepackage{color}

\usepackage[letterpaper,margin=0.85in]{geometry}

\newcommand{\be}{\begin{equation}}
\newcommand{\ee}{\end{equation}}
\newcommand{\benn}{\begin{equation*}}
\newcommand{\eenn}{\end{equation*}}
\newcommand{\bea}{\begin{eqnarray}}
\newcommand{\eea}{\end{eqnarray}}
\newcommand{\beann}{\begin{eqnarray*}}
\newcommand{\eeann}{\end{eqnarray*}}

\theoremstyle{plain}
\newtheorem{theorem}{Theorem}[section]

\newtheorem{lemma}[theorem]{Lemma}

\newtheorem{definition}[theorem]{Definition}
\newtheorem{example}[theorem]{Example}
\newtheorem{remark}[theorem]{Remark}
\newtheorem{assumptions}[theorem]{Assumptions}

\newcommand{\E}{\noindent{$\mathbb{E}$ \ }}

\def\R{\mathbb{R}}
\def\C{\mathbb{C}}

\def\E{\mathbb{E}}

\def\cF{\mathcal{F}}

\def\cL{\mathcal{L}}
\def\cM{\mathcal{M}}

\def\cO{\mathcal{O}}

\def\cV{\mathcal{V}}

\def\cX{\mathcal{X}}

\def\txtb{{\textnormal{b}}}
\def\txtc{{\textnormal{c}}}
\def\txtd{{\textnormal{d}}}
\def\txte{{\textnormal{e}}}
\def\txti{{\textnormal{i}}}
\def\txts{{\textnormal{s}}}

\def\txtu{{\textnormal{u}}}

\def\txtD{{\textnormal{D}}}

\def\Id{{\textnormal{Id}}}

\def\ra{\rightarrow}
\def\I{\infty}

\title{Rough Center Manifolds}

\author{Christian Kuehn\thanks{Technical University of Munich (TUM), 
Faculty of Mathematics, 85748 Garching bei M\"unchen, Germany}~~and~Alexandra 
Neam\c tu\thanks{Technical University of Munich (TUM), 
Faculty of Mathematics, 85748 Garching bei M\"unchen, Germany}}

\begin{document}

\maketitle

\begin{abstract}
Since the breakthrough in rough paths theory for stochastic ordinary differential
equations (SDEs), there has been a strong interest 
in investigating the rough differential equation (RDE) approach and its numerous 
applications. Rough path techniques can stay closer to deterministic analytical methods 
and have the potential to transfer many pathwise ordinary differential equation (ODE) 
techniques more directly to a stochastic setting.  
However, there are few works that analyze dynamical properties of RDEs and connect 
the rough path / regularity structures, ODE and random dynamical systems approaches. 
Here we contribute to this aspect and analyze invariant manifolds for RDEs. By means 
of a suitably discretized Lyapunov-Perron-type method we prove the existence and 
regularity of local center manifolds for such systems. Our method directly works with
the RDE and we exploit rough paths estimates to obtain the relevant contraction 
properties of the Lyapunov-Perron map. 
\end{abstract}

\section{Introduction}
\label{sec:intro}

Classically, the theory of center manifolds is well-studied for ODEs~\cite{Carr,v} 
of the form
\be
\label{eq:ODEintro}
\frac{\txtd U}{\txtd t} := U'= AU+F(U),\qquad U=U_t\in\R^n,
\ee
where $A\in\R^{n\times n}$ is a matrix, $F(0)=0$, and $F(U)=\cO(\|U\|^2)$ as $\|U\|\ra 0$. 
Suppose $A$ has spectrum on the imaginary axis, 
then the equilibrium $U\equiv 0$ is non-hyperbolic, yet there are usually also eigenvalues with 
positive and negative real parts. Center manifolds help us to reduce the dynamics 
to the dimension of the number of eigenvalues with zero real part~\cite{Carr,v}.
The theory also naturally extends to several classes of infinite-dimensional 
partial differential equations (PDEs)~\cite{vioos}, e.g., thinking of $A$ 
in~\eqref{eq:ODEintro} as a differential operator and viewing~\eqref{eq:ODEintro} 
as an evolution equation on a Banach space~\cite{BatesJones,Henry}.

Our aim here is to develop a theory for center manifolds for SDEs driven by 
multiplicative noise, which goes far beyond the case of a Brownian motion. We 
are going to develop center manifolds within the theory of rough 
paths~\cite{Lyons,Gubinelli,FritzHairer}. Of course, the first step is to establish 
existence of center manifolds. Therefore, this work is entirely devoted to this 
aspect. Our proof shows that rough path theory is ideally suited to carry out
the Lyapunov-Perron method for existence of center manifolds for stochastic
systems. 

Based on the existence theory, and motivated by numerous physical applications, 
we are going to analyze approximations of center 
manifolds~\cite{Boxler1,Boxler2,ChenRobertsDuan,Roberts,SchoenerHaken} and 
bifurcations~\cite{Arnold,KnoblauchW} in future work. Regarding this, a challenging 
question that naturally arises is whether the expansions used in the rough 
paths/regularity structures theory can help us gain dynamical insight.\medskip

There is already considerable interest in analyzing stable/unstable or center 
manifolds for stochastic (partial) differential equations (SDEs/SPDEs), see 
e.g.~\cite{Arnold,BloemkerWang,Boxler1,Boxler2,ChenRobertsDuan,
DuanLuSchmalfuss,DuDuan}. However, in previous works the standard approach to derive
invariant manifold results is a transformation argument, often carried out for SDEs 
for $U=U_t$ of the form
\begin{equation}
\label{eqin1}
\txtd U = (A U + F (U)) ~\txtd t + U \circ ~\txtd \tilde{B}, \qquad U_0=:\xi\in\R^n.
\end{equation}
Here, $\circ$ stands for Stratonovich differential, $\tilde{B}=\tilde{B}_t$ 
is a two-sided real-valued 
Brownian motion, $A\in \mathbb{R}^{n\times n}$ is a matrix, $n\geq 1$, and 
$F:\mathbb{R}^{n}\to \mathbb{R}^{n}$ is Lipschitz continuous with $F(0)=0$. 
As for the ODE/PDE case, some results for~\eqref{eqin1} can be extended to certain 
unbounded operators $A$ generating strongly-continuous semigroups on separable Banach 
spaces, which is relevant for SPDEs. The classical transformation of~\eqref{eqin1} 
relies on the Ornstein-Uhlenbeck (OU) process as follows: Consider the unique 
solution $z=z_t$ of the one-dimensional OU process given by the SDE 
\begin{equation}
\label{ouoned}
\txtd z = -z ~\txtd t + \txtd \tilde{B}.
\end{equation}
Performing the Doss-Sussmann-type 
transformation $\tilde{U}=U \txte^{-z}$, one obtains that~\eqref{eqin1} can be transformed 
to a non-autonomous random differential equation 
\begin{equation}
\label{eqin2}
\tilde{U}' = A \tilde{U} + \tilde{F}(\tilde{U};z),
\end{equation}
where the map $\tilde{F}$ depends upon the time-dependent process $z$ and 
hence on $\tilde{B}$; see~\cite{Arnold}. Now, the random differential equation~\eqref{eqin2} 
can be analyzed using the properties of the OU process \cite[Lem.~2.1]{DuanLuSchmalfuss} 
and suitable assumptions on the coefficients. Of course, the same methods apply if 
one takes in~\eqref{eqin1} linear multiplicative noise, namely $G(U) \circ \txtd \tilde{B}$, where $G$ is a linear operator generating a strongly-continuous group commuting with the 
strongly-continuous semigroup generated by $A$.

Under suitable spectral assumptions on $A$, which imply the existence of 
an exponential trichotomy that entails invariant splittings of the phase space in 
stable/unstable/center subspaces, one can set up a classical Lyapunov-Perron method 
and derive under certain gap conditions invariant manifolds for \eqref{eqin2}. Since 
the random dynamical systems generated by~\eqref{eqin1} and~\eqref{eqin2} are conjugated, 
one can transfer all these results to~\eqref{eqin1}.\medskip

Instead, the goal of this paper is to use a direct pathwise 
approach~\cite{GarridoLuSchmalfuss} to investigate center manifolds for RDEs 
of the form~\eqref{sde1}. More precisely, we consider~\eqref{eqin1} 
driven by a quite general \emph{nonlinear multiplicative noise}. Moreover, the random 
input goes beyond the case of a Brownian motion and it can be a certain 
\emph{Gaussian process}. The only restriction comes from the H\"older regularity 
of its trajectories. We are going to prove our results
using \emph{rough paths}, which have not been employed for \emph{invariant manifolds} 
so far.
Since the breakthrough in the rough paths theory introduced by T.~Lyons \cite{Lyons}, 
there has been a growing interest in analyzing flows driven by rough paths~\cite{B} 
or random dynamical systems for rough differential 
equations~\cite{BailleulRiedelScheutzow}. The main techniques used in this work rely 
on Gubinelli's controlled rough paths as described in~\cite{Gubinelli,FritzHairer}. 
This theory gives a pathwise meaning to~\eqref{eqin1} \emph{without reducing it to an 
equation with random coefficients}. Since the solution of an RDE driven by a 
Stratonovich rough path lift yields a strong solution of a classical Stratonovich 
SDE \cite[Thm.~9.1(ii)]{FritzHairer}, many classical center manifold results for 
\eqref{eqin1} can be recovered as special cases.\\


This work is structured as follows. In Section~\ref{rde} we provide background 
on controlled rough paths and RDEs. Section \ref{rd} is devoted to dynamics of 
rough differential equations~\cite{BailleulRiedelScheutzow}. The existence of 
center manifolds is based on a discrete-time Lyapunov-Perron method. This technically 
challenging step is necessary since we work with pathwise integrals and therefore need 
to control at each step the norms of the random input on a fixed time-interval. After 
deriving suitable estimates of the controlled rough integrals, the existence of local center manifolds for rough differential equations
 (Theorem~\ref{lemma:cm}) follows using a random dynamical systems approach. 
The results obtained for the discrete Lyapunov-Perron map can then be extended to 
the time-continuous one~\cite{LianLu,GarridoLuSchmalfuss}. Finally, we point out in 
Section~\ref{smoothness} the main arguments, which lead to the smoothness of the manifolds 
obtained. This requires technical transfer of several existing 
ideas~\cite{DuanLuSchmalfuss,ChenRobertsDuan,GuoShen} to the RDE context; Appendix~\ref{a} 
summarizes basic methods used to establish the fiber-wise 
invariance~\cite{Arnold,DuanLuSchmalfuss,WaymireDuan} of the random manifolds. 
Appendix~\ref{b} provides additional information on the extension of our proof from
dichotomies to trichotomies. 

We also point out the translation of our 
results to the language of regularity structures should be possible as controlled 
rough paths can be viewed as one particular instance of a regularity structure; 
see~\cite{Brault}. Hence, the techniques we developed here for RDEs/SDEs may have 
the potential to be eventually applicable for quite large classes of singular 
SPDEs~\cite{Hairer,GubinelliPerkowski,HairerShen,HairerWeber,BerglundKuehn1}. 
We leave this aspect as a major direction for future work.\medskip 

\textbf{Notation:} We use $\E$ to denote the expectation, $\otimes$ for the tensor 
product, $\oplus$ for the direct sum, $C^{m}$ for the space of $m$-times differentiable 
maps, $C^m_\txtb$ for the space of $m$-times differentiable maps with bounded 
derivatives, and $C^\alpha$ for $\alpha$-H\"older maps.\medskip

\textbf{Acknowledgments:} CK and AN have been supported by a DFG grant in the
D-A-CH framework (KU 3333/2-1). CK also acknowledges support by a Lichtenberg 
Professorship. CK also thanks the EU for partial support within the TiPES project
funded by the European Union Horizon 2020 research and innovation program under
grant No.~820970. The authors thank the referee for the numerous extremely valuable suggestions.
 
\section{Rough differential equations} 
\label{rde} 

We fix the time interval $[0,1]$ due to technical reasons as discussed further in 
Section~\ref{lpm}. For $t\in[0,1]$ consider the rough differential equation (RDE)
\begin{equation}
\label{sde1}
\txtd U = [A U + F(U)] ~\txtd t + G(U) ~\txtd \textbf{W},\qquad U=U_t,~U_0=\xi\in\R^n,
\end{equation}
where $n\geq 1$, $A\in\R^{n\times n}$ is a matrix, $F:\R^n\to \R^n$ is Lipschitz, 
$G:\R^n \to \R^{n\times d}$ is a $C^{3}_{\txtb}$ matrix-valued map, and the rough path $\textbf{W}$ 
will be defined below; one may think of a vector of independent two-sided Brownian 
motions in $\R^d$ as one key example leading to a construction of $\textbf{W}$, 
yet our results are not limited to this case. The solution of~\eqref{sde1} can 
formally be written as
\begin{equation}
\label{integraleq}
U_{t}= S(t) \xi + \int\limits_{0}^{t} S(t-r)F(U_{r})~\txtd r 
+ \int\limits_{0}^{t} S(t-r) G(U_{r}) ~\txtd \textbf{W}_{r},
\end{equation}
where $S(t):=\txte^{tA}$, and the last integral is given by Gubinelli's controlled 
rough integral to be defined below. Fix $\alpha\in (1/3,1/2]$ and a finite-dimensional
vector space $\cV$ (the reader may think of $\cV=\R^d$ as one concrete example we shall
often use for $\cV$). Then we define 
$\textbf{W}:=(W,\mathbb{W})$ as an $\alpha$-H\"older rough-path, more precisely  
\benn
W\in C^{\alpha}([0,1];\cV) \qquad \text{and}\qquad  
\mathbb{W}\in C^{2\alpha}([0,1]^{2};\cV \otimes \cV),
\eenn 
and the connection between $W$ and $\mathbb{W}$ is given by Chen's relation. 
This means that
\begin{align}
\label{chen}
\mathbb{W}_{s,t} -\mathbb{W}_{s,u} -\mathbb{W}_{u,t} = {W}_{s,u} \otimes W_{u,t}.
\end{align}
where $W_{s,t}:=W_{t} -W_{s}$. For a smooth path $W$, $\mathbb{W}$ could be constructed 
using iterated integrals of $W$. The theory of controlled rough paths gives a meaning to 
(\ref{integraleq}) even if the paths of the driving process are not smooth, but only 
$\alpha$-H\"older regular for $\alpha\in(1/3,1/2]$. This includes Brownian motion and 
fractional Brownian motion with Hurst parameter $H\in(1/3,1/2]$. The second order process 
$\mathbb{W}$ can be thought of as 
\begin{align}\label{omega2}
\int\limits_{s}^{t} W_{s,u} \otimes \txtd  W_{u}.
\end{align}
Note that if $\mathbb{W}$ exists, it is not unique, since for any $\R^d\otimes \R^d$-valued 
$2\alpha$-H\"older function $f$, we have that $\mathbb{W}_{s,t}+ f_{t} - f_{s}$ satisfies 
again all the required properties. For more information about the construction of 
$\mathbb{W}$ for Gaussian processes see \cite[Ch.~10, Thm.~10.4]{FritzHairer}, where 
the regularity of the covariance function allows one to define $\mathbb{W}$ as an 
iterated integral; see also Section~\ref{rd}.

\begin{remark}
Results regarding global existence and uniqueness of a solution of (\ref{sde1}) for 
smooth $F$ and $G$ can be found in \cite[Ch.~8]{FritzHairer}. These are stated without 
using Duhamel's formula and incorporating $\txtd t$ in a space-time rough path. Results 
without incorporating $\txtd t$ in a space-time rough path can be found 
in~\cite{CoutinLejay,FritzVictoir}. Since we are going to investigate center manifolds 
for (\ref{sde1}), we prefer the mild formulation specified in (\ref{integraleq}). In fact,
mild solutions to SPDEs~\cite{DaPratoZabczyk} and their 
variants~\cite{KuehnNeamtu,PronkVeraar}, are used implicitly in regularity 
structures~\cite{Hairer} and form a cornerstone of deterministic PDE 
dynamics~\cite{Henry} so it is also natural for our future goals to use this notion 
of solution.
\end{remark}

Now we consider preliminary results, which are necessary to establish the existence of 
a solution for (\ref{sde1}), cf.~\cite[Sec.~8.5]{FritzHairer}. These results, and the 
techniques by which they are proven, are going to also play a key role for our center 
manifold problem. Throughout this section $|\cdot|$ denotes the usual Euclidean norm, 
and $|\cdot|$ also denotes the induced Euclidean norm of linear operators. We write    
$\|\cdot\|_{\infty}$ for the supremum norm, $\|\cdot\|_{\alpha}$ for the $\alpha$-H\"older 
semi-norm on $[0,1]$ and $|||\cdot|||_{\alpha}$ for the $\alpha$-H\"older norm on $[0,1]$. Furthermore, $\cL(\cV,\cX)$ will be the space of linear maps between 
two finite-dimensional vector spaces $\cV,\cX$. \\

As commonly met in the rough paths theory, compare~\cite[Ch.~4]{FritzVictoir}, we firstly work with H\"older semi-norms and point out the main techniques required in order to solve~\eqref{integraleq}. Afterwards, we provide some preliminary results for the local center manifolds theory developed in Section~\ref{lpm}. These are based on a cut-off argument which will be illustrated in Section~\ref{cut}. 

\begin{definition}
\label{def:crp} 
A path $Y\in C^{\alpha}([0,1];\cX)$ is controlled by $W\in C^{\alpha}([0,1];\cV)$ 
if there exists $Y'\in C^{\alpha}([0,1];\cL(\cV,\cX))$ such that
\begin{align}
\label{def1}
Y_{t} = Y_{s} + Y'_{s} W_{s,t} + R^{Y}_{s,t},
\end{align}
where the remainder $R^{Y}$ has $2\alpha$-H\"older regularity. 
 The space of controlled 
rough paths $(Y,Y')$ on the time-interval $[0,1]$ is denoted by $D^{2\alpha}_{W}([0,1];\cX)=:D^{2\alpha}_{W}$. 
This space is endowed with the semi-norm 
\begin{align}\label{norm}
\|Y,Y'\|_{D^{2\alpha}_W}:= \|Y'\|_{\alpha} + \|R^{Y}\|_{2\alpha}.
\end{align}
\end{definition}
A norm on $D^{2\alpha}_W$ is given by $|Y_{0}| + |Y'_{0}| + \|Y'\|_{\alpha} + \|R^{Y}\|_{2\alpha} $ or $\|Y\|_{\infty} + \|Y'\|_{\infty} + \|Y'\|_{\alpha} + \|R^{Y}\|_{2\alpha}$, where both norms are equivalent, see~\cite[Rem.~4.10]{Brault}. The norm on $D^{2\alpha}_{W}([0,1];\cX)$ will be denoted by $\|\cdot\|_{\textbf{D}^{2\alpha}_{W}([0,1];\cX)}$ (short $\|\cdot\|_{\textbf{D}^{2\alpha}_{W}}$).
Obviously, if $Y_0=0$ and 
$Y'_0=0$, then~\eqref{norm} is a norm on $D^{2\alpha}_{W}$. This fact will be exploited in Section~\ref{sectfp}. The term $Y'$ is referred to as the Gubinelli derivative of $Y$, see~\cite{Gubinelli} 
and \cite[Ch.~4]{FritzHairer}\footnote{For smooth paths $Y$ and $W$, the choice of $Y'$ is not unique. However, one can show that for rough inputs $W$, $Y'$ is uniquely determined by $Y$, see~\cite[Rem.~4.7 and Sec.~6.2]{FritzHairer}.}.

\begin{example}
Consider integrating the RDE~\eqref{sde1} directly over time, then one may
give meaning to the integral $\int G(U)~\txtd \textbf{W}_r$ by using 
\be
\label{eq:concchoice}
\quad Y=G(U), \quad Y'=\txtD G(U)G(U), \quad \cV=\R^d,\quad \cX=\R^{n\times d},
\ee
$\txtD G$ denotes the total derivative of $G$. In this case, we also note that 
\benn
\cL(\cV,\cX)=\cL(\R^d,\R^{n\times d})=\cL(\R^d,\cL(\R^d,\R^n))\simeq
\cL(\R^d\otimes \R^d,\R^n).
\eenn
Yet, we shall see that this example does not suffice for our case as we 
also would like to study solution formulas involving a (semi-)group. 
\end{example}

We proceed with some rough path estimates. Here we emphasize that 
$C$ stands for a universal constant, 
which may vary from line to line. The dependence of this constant 
\benn
C=C[\cdot,\cdot,\ldots]
\eenn
on certain parameters and/or problem input will be explicitly stated in square 
brackets. $C$ can always be uniformly chosen over the unit interval, i.e., for 
$T\in[0,1]$.\medskip

Re-writing~(\ref{def1}) 
entails $Y_{s,t}= Y'_{s} W_{s,t} + R^{Y}_{s,t}$, so one obtains
\begin{align*}
|Y_{s,t}| \leq \|Y'\|_{\infty} \|W\|_{\alpha}(t-s)^{\alpha} 
+ \|R^{Y}\|_{2\alpha}(t-s)^{2\alpha},~\mbox{ for } s,t\in[0,1].
\end{align*}
This immediately entails the following estimates for the path
$Y$:
\begin{align}\label{y:infty}
\|Y\|_{\infty} \leq \|Y'\|_{\infty} \| W\|_{\alpha}  
+ \| R^{Y}\|_{2\alpha} ~~ 
\end{align}
as well as
\begin{equation}
\label{y}
|||Y|||_{\alpha} \leq  |Y_0| + \| Y'\|_{\infty} \|W\|_{\alpha} +  
\|R^{Y}\|_{2\alpha}.
\end{equation}
Since
	$$\|Y \| _{\infty} \leq |Y_{0}| + \|Y\|_{\alpha},$$
we can further estimate the $\alpha$-H\"older semi-norm of $Y$ as
\begin{equation}
\label{esthoeldernormy}
\| Y\|_{\alpha} \leq C (1+ \|W\|_{\alpha}) (|Y'_{0}| + \|Y,Y'\|_{D^{2\alpha}_{W}}). 
\end{equation} 
The second estimate immediately follows from (\ref{y}) combined with the 
definition of the semi-norm on $D^{2\alpha}_{W}$. The next step is to explain 
the concept of the integral we use in (\ref{integraleq}). 

\begin{theorem}(\cite[Prop.~1]{Gubinelli})
\label{gubinelliintegral}
Let $(Y,Y')\in D^{2\alpha}_{W}([0,1];\R^{n\times d})$ and $\bm{W}=(W,\mathbb{W})$ 
be an $\R^d$-valued 
$\alpha$-H\"older rough path for some $\alpha\in(\frac{1}{3},\frac{1}{2}]$. Furthermore, 
$\mathcal{P}$ stands for a partition of $[0,1]$. Then, the integral of $Y$ against 
$\bm{W}$ defined as
\begin{equation}\label{Gintegral}
	\int\limits_{s}^{t} Y_{r} ~\txtd\textbf{W}_{r} 
	:= \lim\limits_{|\mathcal{P}|\to 0} \sum\limits_{[u,v]\in\mathcal{P}} 
	(Y_{u}W_{u,v} + Y'_{u}\mathbb{W}_{u,v} )
\end{equation}
exists for every pair $s,t\in[0,1]$. Moreover, the estimate
\begin{equation}\label{estimate_g_integral}
	\left| \int\limits_{s}^{t} Y_{r}~ \txtd\textbf{W}_{r} - Y_{s}W_{s,t} - 
	Y'_{s}\mathbb{W}_{s,t} \right| \leq C (\|W\|_{\alpha} \|R^{Y}\|_{2\alpha} 
	+ \|\mathbb{W} \|_{2\alpha} \|Y' \|_{\alpha} ) |t-s|^{3\alpha}
\end{equation} 
holds true for all $s,t\in[0,1]$.
The map from $\mathcal{D}^{2\alpha}_{W}([0,1];\R^{d\times n})$ to 
$D^{2\alpha}_{W}([0,1];\R^n)$ given by 
\benn
(Y,Y')\mapsto (P,P'):=
\left(\int\limits_{0}^{\cdot}Y_{r}~\txtd\textbf{W}_{r} , Y_{\cdot} \right), 
\eenn
is linear and continuous. Furthermore, the estimate
\begin{equation}\label{bounddalphaintegral}
	 \| P,P'\|_{D^{2\alpha}_{W}} \leq \| Y\|_{\alpha} + \|Y'\|_{\infty} 
	\|\mathbb{W}\|_{2\alpha} + C (\|W\|_{\alpha} \|R^{Y}\|_{2\alpha} 
	+ \|\mathbb{W} \|_{2\alpha} \|Y' \|_{\alpha} )
\end{equation}
is valid.
\end{theorem}

Next, we have to consider the semigroup $S=S(t)$ and provide estimates for certain constants 
depending on $S$, which are helpful throughout this section, see also Section~\ref{lpm}. 
We recall the following simple result:

\begin{lemma}
\label{estimate:semigroup}
Let $\cX$ be a finite-dimensional vector space and $A\in\cL(\cX,\cX)=:\cL(\cX)$. Then 
there exists a constant $\tilde{M}\geq 1$ such that
$$|S(t)| \leq \tilde{M} \txte^{ t |A|}  
\quad\emph{and}\quad 
|S(t) -\textnormal{Id} |\leq \tilde{M} |A| t \txte^{t |A|},~~  t\geq 0.$$
\end{lemma}

For simplicity, we can assume $\tilde{M}=1$ throughout this section, since we only need to derive 
appropriate time-regularity results. \medskip

In order to construct (\ref{Gintegral}) we need that 
$(Y,Y')\in D^{2\alpha}_{W}([0,1];\mathcal{L}(\cV,\cX))$. Therefore, to justify that
\begin{align}
\label{s:integral}
\int\limits_{0}^{\cdot} S(\cdot - r) Y_{r} ~\txtd \textbf{W}_{r},
\end{align}
can be defined by (\ref{Gintegral}) we need to establish that 
$(S(t-\cdot)Y_{\cdot}, (S(t-\cdot)Y_{\cdot})') \in D^{2\alpha}_{W}([0,1];\mathcal{L}(\cV,\cX))$.
This is contained in the next result, which contains the choice~\eqref{eq:concchoice} as a
special case.

\begin{lemma}
\label{derivativesemigroup} 
Let $(Y,Y')\in D^{2\alpha}_{W}([0,1];\mathcal{L}(\cV,\cX))$. For every $t\in[0,1]$ 
we set $Z^{t}_{\cdot}:= S(t-\cdot{})Y_{\cdot}$. Then we have $(Z^{t}, (Z^{t})' ) 
\in D^{2\alpha}_W([0,1];\mathcal{L}(\cV,\cX))$, where $(Z^{t}_{\cdot})'= 
S(t-\cdot{})Y'_{\cdot}$.
\end{lemma}

\begin{proof} Regarding the definition of $D^{2\alpha}_{W}$ we have to show that 
$Z^{t}\in C^{\alpha}$, $(Z^{t})'\in C^{\alpha}$ and $R^{Z^{t}}\in C^{2\alpha}$.
We fix $0\leq u < v \leq 1$. Then it follows that
\begin{align*}
	 |Z^{t}_{u,v}|& = |  S(t-v) Y_{v} - S(t-u) Y_{u} | \\
	 & \leq | S(t-v) Y_{v} - S(t-v) Y_{u} | + |(S(t-v)- S(t-u)) Y_{u}|\\
	 & \leq |S (t-v)| |Y_{u,v} | + |S(t-v) | |S(v-u)-\mbox{Id} | |Y_{u}|  \\
	 & \leq \txte^{t |A|} \|Y\|_{\alpha} |v-u|^{\alpha} + |A| \txte^{2t |A|}  
	\|Y\|_{\infty} |v-u|\\
	 & \leq  \txte^{t |A|} \|Y\|_{\alpha} |v-u|^{\alpha} + |A|  \txte^{2t |A|} 
	(|Y_{0}| + \|Y\|_{\alpha} ) |v-u|.
\end{align*}
Using the previous computation, we derive
\begin{align}
\label{estz}
\|Z^{t}\|_{\alpha} \leq C |A| \txte^{2t |A| } ( |Y_{0}| + \|Y\|_{\alpha} ).
\end{align}
Replacing $Y$ with $Y'$ in the previous computation, one obtains the same 
estimate for the Gubinelli derivative $(Z^{t}_{\cdot})'$ from which 
the $\alpha$-H\"older regularity immediately follows.	Next, we now focus on 
the remainder $R^{Z^{t}}$ and aim to show that it is $2\alpha$-H\"older continuous. 
We have:
\begin{align}
 	|R^{Z^{t}}_{u,v}| &= |Z^{t}_{u,v} - (Z^{t}_{u})'W_{u,v} | \label{remainder1} \\
 	& =|S(t-v) Y_{v} - S(t-u) Y_{u} - S(t-u) Y'_{u} W_{u,v} |\nonumber\\
 	& = | S(t-v) Y_{v} - S(t-u) Y_{u} - S(t-u) Y_{u,v} + S(t-u) R^{Y}_{u,v} |\nonumber\\
 	& \leq |S(t-v) | |S(v-u) -\mbox{Id}| |Y_{v} | + |S(t-u) | | R^{Y}_{u,v}|\nonumber\\
 	& \leq | A| \txte^{2 t |A| }  \| Y\|_{\infty} |v-u| + 
	\txte^{t|A|} \|R^{Y}\|_{2\alpha}|v-u|^{2\alpha}. \nonumber
\end{align}
From this we infer that $R^{Z^{t}}\in C^{2\alpha}$ and
$$\|R^{Z^{t}}\|_{2\alpha} \leq C |A| \txte^{2 t |A|} 
(|Y_{0}|+ \|Y\|_{\alpha} + \|R^{Y}\|_{2\alpha} ) .$$
This completes the proof.
\end{proof}

Up to now, we have shown that we can define the 
integral (\ref{s:integral}) by (\ref{Gintegral}). Now, we compute its Gubinelli 
derivative and prove that
\begin{align*}
\left(\int\limits_{0}^{\cdot} S(\cdot - r) Y_{r} ~\txtd\textbf{W}_{r} , 
\left(\int\limits_{0}^{\cdot} S(\cdot - r) Y_{r} ~\txtd\textbf{W}_{r} \right)'\right)
\end{align*}
forms a controlled rough path.

\begin{lemma}
\label{integralgub}
Let $(Y,Y')\in D^{2\alpha}_{W}([0,1];\mathcal{L}(\cV,\cX))$. Then 
\begin{align}
\label{integral}
\left(  \int\limits_{0}^{\cdot{}}  S(\cdot{}-r)Y_{r} ~\txtd\textbf{W}_{r}, 
Y_{\cdot} \right ) \in D^{2\alpha}_{W}([0,1];\cX).
\end{align}
\end{lemma}

\begin{proof}
Due to Lemma \ref{derivativesemigroup} we can define the integral 
\benn
I_{t}:=\int\limits_{0}^{t} S(t-r) Y_{r} ~\txtd\textbf{W}_{r}.
\eenn
To prove (\ref{integral}) we only show that $R^{I}$ is $2\alpha$-H\"older, 
since the other properties are obvious, i.e., a Gubinelli derivative 
of $I$ is $Y$, which is $\alpha$-H\"older continuous by definition. 
Lemma~\ref{estimate:semigroup} together with (\ref{estimate_g_integral}) 
justify the $\alpha$-H\"older continuity of $I$ itself, since
\begin{align*}
I_{s,t} = (S(t-s) -\mbox{Id})\int\limits_{0}^{\txts}  S(s-r) Y_{r} ~\txtd \textbf{W}
_{r}
+ \int\limits_{s}^{t} S(t-r) Y_{r} ~\txtd\textbf{W}_{r}.
\end{align*}
Consequently, the first term gives us 
\begin{align*}
\Bigg| (S(t-s) -\mbox{Id})\int\limits_{0}^{\txts}  S(s-r) Y_{r} ~\txtd \textbf{W}
_{r} \Bigg| \leq |A|~ |t-s| ~ \txte ^{(t-s)|A|} \left| 
\int\limits_{0}^{\txts} S(s-r)Y_{r} ~\txtd \textbf{W}_{r} \right|
\end{align*}
and the second one
\begin{align*}
\Bigg| \int\limits_{s}^{t} S(t-r)Y_{r} ~\txtd \textbf{W}_{r} \Bigg| 
& \leq |S(t-s)| \|Y\|_{\infty}\|W\|_{\alpha }(t-s)^{\alpha} \\
&+ |S(t-s)| \|Y'\|_{\infty} \|\mathbb{W}\|_{2\alpha} (t-s)^{2\alpha} 
+ C (t-s)^{3\alpha}.
\end{align*}
We now prove the $2\alpha$-H\"older regularity of the remainder. To this aim 
 we compute for $0\leq s \leq t \leq 1$
\begin{align}\label{remainderint}
		|R^{I}_{s,t} |& = | I_{s,t} - I'_{s} W_{s,t} | = 
		\left|  \int\limits_{0}^{t} S(t-r) Y_{r} ~\txtd \textbf{W}_{
			r} - \int\limits_{0}^{\txts} S(s-r) Y_{r} ~\txtd \textbf{W}_{r} 
			- Y_{s} W_{s,t} \right| \\
	& \leq \left|  \int\limits_{0}^{\txts} (S(t-r) - S(s-r)) Y_{r} ~\txtd \textbf{W}
		_{r} \right| + \left| \int\limits_{s}^{t} S(t-r) Y_{r} ~\txtd\textbf{W}_{r} 
		-Y_{s}W_{s,t} \right|.\nonumber
\end{align}	
We start estimating the first term in the previous inequality as
\begin{align*}
\left|  \int\limits_{0}^{\txts} (S(t-r) - S(s-r)) Y_{r} ~\txtd \textbf{W}_{r} 
\right| &\leq | S(t-s) -\mbox{Id}| \left| 
\int\limits_{0}^{\txts} S(s-r)Y_{r} ~\txtd \textbf{W}_{r} \right|\\
& \leq|A| |t-s| e ^{(t-s)|A|} \left| \int\limits_{0}^{\txts} S(s-r)Y_{r} 
~\txtd \textbf{W}_{r} \right|.
\end{align*}	
To estimate the rough integral we apply (\ref{estimate_g_integral}) and obtain 
\begin{align*}
&	\left|  \int\limits_{0}^{\txts} (S(t-r) - S(s-r)) Y_{r} ~\txtd 
\textbf{W}_{r} \right| \\
& \leq |A| ~|t-s|~ e ^{(t-s)|A|}~ \left(  |S(s)Y_{0}W_{0,s}| + |S(s)Y'_{0}\mathbb{W}_{0,s}| 
+ C (\|W \|_{\alpha} \| R^{Z^{\txts}} \|_{2\alpha} + \|\mathbb{W}\|_{2\alpha}
 \|(Z^{\txts})'\|_{\alpha} ) s^{3\alpha} \right)\\
& \leq C[|A|] ~ |t-s|~  \left( |Y_{0}|~ \|W\|_{\alpha} s^{\alpha} 
+|Y'_{0}| ~\|W\|_{2\alpha}s^{2\alpha} + C (\|W \|_{\alpha}~ \| R^{Z^{\txts}} 
\|_{2\alpha} 
+ \|\mathbb{W}\|_{2\alpha} ~\|(Z^{\txts})'\|_{\alpha} ~s^{3\alpha})  \right).
\end{align*}
Plugging in the estimates obtained for $\|R^{Z^{\txts}}\|_{2\alpha}$ and 
$\| (Z^{\txts})' \|_{\alpha}$ in Lemma \ref{derivativesemigroup} we conclude 
that we can find a constant $C=C[|A|]$ such that 
\begin{align*}
&	\left|  \int\limits_{0}^{\txts} (S(t-r) - S(s-r)) Y_{r} ~\txtd \textbf{W}_{r} \right| \\
& \leq C [|A|]~ (|Y_{0}| + |Y'_{0}| +\|Y\|_{\alpha} + 
\|Y'\|_{\alpha}+ \|R^{Y}\|_{2\alpha} )~ (\|W\|_{\alpha} + 
\|\mathbb{W}\|_{2\alpha} )~ |t-s|.
\end{align*}
We now estimate the second term as 
follows. Using again (\ref{estimate_g_integral}) we infer that
\begin{align*}
&\left| \int\limits_{s}^{t} S(t-r) Y_{r} ~\txtd \textbf{W}_{r} - Y_{s}W_{s,t} \right| \\
& \leq |S(t-s)Y_{s}W_{s,t} -Y_{s}W_{s,t} | + |S(t-s)Y'_{s}\mathbb{W}_{s,t}| 
+ C ( \|W\|_{\alpha} \|R^{Z^{t}}\|_{2\alpha} + \|\mathbb{W}\|_{2\alpha} \| 
(Z^{t})'\|_{\alpha}  ) |t-s|^{3\alpha}\\
& \leq C[|A|] ~\Big( \|Y\|_{\infty} ~\|W\|_{\alpha}~ |t-s|^{1+\alpha} 
+  \|Y'\|_{\infty} ~\|\mathbb{W}\|_{2\alpha}~|t-s|^{2\alpha} \Big) \\
& + C ~( \|W\|_{\alpha} ~\|R^{Z^{t}}\|_{2\alpha} + \|\mathbb{W}\|_{2\alpha}  
~\|(Z^{t})'\|_{\alpha}  ) ~|t-s|^{3\alpha}.
\end{align*}
Consequently, we have 
\begin{align*}
&\left| \int\limits_{s}^{t} S(t-r) Y_{r} ~\txtd \textbf{W}_{r} - Y_{s}W_{s,t} \right| \\
&\leq C[|A|]~ (|Y_{0}| + |Y'_{0}| + \|Y\|_{\alpha} + \|Y'\|_{\alpha} 
+ \|R^Y\|_{2\alpha} ) (\|W\|_{\alpha} + \|\mathbb{W} \|_{2\alpha}  ) |t-s|^{2\alpha},
\end{align*}
which proves the required regularity of the remainder of the controlled rough integral. 
Putting all these estimates together and recalling that
\begin{align*}
\|Y\|_{\alpha} \leq C (1+\|W\|_{\alpha}) (|Y'_{0}| + \|Y,Y'\|_{D^{2\alpha}_{W}} ),		
\end{align*}
we finally infer that
\begin{align}
&	\left\|\int\limits_{0}^{\cdot} S(\cdot - r)Y_{r}~\txtd \textbf{W}_{r}, Y 
\right\|_{D^{2\alpha}_{W}}\nonumber\\
& \leq \|Y\|_{\alpha} + C[|A|] (|Y_{0}| + |Y'_{0}| + \|Y,Y'\|_{D^{2\alpha}_{W}} ) 
(1+\|W\|_{\alpha})(\|W\|_{\alpha} + \|\mathbb{W}\|_{2\alpha} ).\label{if}
\end{align}
This finishes the proof.
\end{proof}

Keeping the previous computation in mind, one can easily derive that 
for $(Y,Y')\in D^{2\alpha}_{W}([0,1];\cX)$ and $G\in 
C^{3}_{\txtb}(\cX;\mathcal{L}(\cV,\cX))$ the following estimate holds true
\begin{align}
\label{estg}
& \left\|\int\limits_{0}^{\cdot} S(\cdot - r)G(Y_{r})~\txtd \textbf{W}_{r}, 
G(Y) \right\|_{D^{2\alpha}_{W}} \\
& \leq \| G (Y)\|_{\alpha} + C [|A|] (|G(Y_{0})| + |(G(Y))'_{0}| + 
\|G(Y),(G(Y))' \|_{D^{2\alpha}_{W}} ) (1+ \|W\|_{\alpha}) (\|W\|_{\alpha} 
+ \|\mathbb{W}\|_{2\alpha}). \nonumber
\end{align}
Note that
\begin{align*}
\|G(Y) \|_{\alpha}  \leq \|\txtD G\|_{\infty} \| Y\|_{\alpha}
\leq \|G\|_{C^{1}_{\txtb}} \|Y\|_{\alpha}
\end{align*}
and
\begin{align*}
|(G(Y))'_{0}|= |\txtD G(Y_{0})Y'_{0}|\leq \|G\|_{C^{2}_{\txtb}}.
\end{align*}
Now, one uses \cite[Lem.~7.3]{FritzHairer} which gives a meaning of the 
operation of composition of a controlled rough path with a smooth function 
together with the estimate 
\begin{align}
\label{comp:smooth:funct}
\|G(Y),(G(Y))' \|_{D^{2\alpha}_{W}}&= \|G(Y), DG(Y)Y' \|_{D^{2\alpha}_{W}}\nonumber\\
& \leq C \|G\|_{C^{2}_{\txtb}} M \left(|Y'_{0}| + \|Y,Y'\|_{D^{2\alpha}_{W}}\right ) 
(1+ \|W\|_{\alpha})^{2},
\end{align}
for $(Y,Y')\in D^{2\alpha}_{W}([0,1];\cX)$ with $|Y'_{0}| 
+ \|Y,Y'\|_{D^{2\alpha}_{W}}\leq M$, for a positive constant $M$.\medskip

The following result (\cite[Thm.~7.5]{FritzHairer}) provides explicit estimates for the difference of two controlled rough 
paths. For the sake of completeness we sketch the main ideas of the proof, since similar arguments will be required in Section~\ref{cut}.
\begin{lemma}
\label{differ} Let $Y, \widetilde{Y}\in D^{2\alpha}_{W}([0,1];\cX)$ 
with $|Y_{0}'| + \|Y,Y'\|_{D^{2\alpha}_{W}}\leq M$ and $|\widetilde{Y}'_{0}| 
+\|\widetilde{Y},\widetilde{Y}'\|_{D^{2\alpha}_{W}}\leq M$. Then, there 
exists a constant $C=C[\alpha,\|W\|_{\alpha}]$, such that the estimate
	\begin{align*}
&	\|G(Y)- G(\widetilde{Y}), (G(Y))' -(G(\widetilde{Y}))' \|_{D^{2\alpha}_{W}}
\\
&\leq C M^{2} \|G\|_{C^{3}_{\txtb}}  (|Y_{0}-\widetilde{Y}_{0}| 
+ |Y'_{0} -\widetilde{Y}'_{0}| + 
\| Y-\widetilde{Y}, Y'-\widetilde{Y}'\|_{D^{2\alpha}_{W}} )
	\end{align*}
is valid.
\end{lemma}

\begin{proof}
Fix $x,y\in \cX$. Then it holds
\begin{align*}
	G(x) - G(y)= (x-y) \int\limits_{0}^{1} \txtD G (rx+(1-r)y)~\txtd r.
\end{align*}
Therefore, we define for $Y,\widetilde{Y}\in D^{2\alpha}_{W}$ and $s\in[0,1]$
\benn
\frak{K}_{s} :=g(Y_{s},\widetilde{Y}_{s}) \quad \text{and}\quad \frak{H}_{s} 
:= Y_{s} -\widetilde{Y}_{s},
\eenn
where we set $$g(x,y):=\int\limits_{0}^{1} \txtD G (rx+(1-r)y)~\txtd r.$$ Consequently, we 
have that $G(Y)-G(\widetilde{Y})=\frak{K}\frak{H}$. First of all, we infer 
from \cite[Lem.~7.3]{FritzHairer} that $(\frak{K},\frak{K}')\in D^{2\alpha}_{W}$, where
\begin{align*}
\frak{K}'= \txtD_{x} g(Y,\widetilde{Y})Y' + \txtD_{y} g (Y,\widetilde{Y})\widetilde{Y}'.
\end{align*}
Again the result \cite[Lem.~7.3]{FritzHairer} entails
\begin{align}
	\|\frak{K},\frak{K}'\|_{D^{2\alpha}_{W}} & \leq C \| g\|_{C^{2}_{\txtb}} M (|Y'_{0}| 
	+ |\widetilde{Y}'_{0}| + \|Y,Y' \|_{D^{2\alpha}_{W}} + \|\widetilde{Y},
	\widetilde{Y}'\|_{D^{2\alpha}_{W}} ) (1+\|W\|_{\alpha})^{2} \label{g}\\
	& \leq C \|G\|_{C^{3}_{\txtb}} M (|Y'_{0}| + |\widetilde{Y}'_{0}| + 
	\|Y,Y' \|_{D^{2\alpha}_{W}} + 
	\|\widetilde{Y},\widetilde{Y}'\|_{D^{2\alpha}_{W}} ) 
	(1+\|W\|_{\alpha})^{2}\label{one:g}.
\end{align}
Furthermore, \cite[Cor.~7.4]{FritzHairer} gives us that 
$(\frak{K}\frak{H}, (\frak{K}\frak{H})')\in D^{2\alpha}_{W}$ and $(\frak{K}\frak{H})'
= \frak{K}'\frak{H} + \frak{K} \frak{H}'$. Moreover, we obtain
\be
\label{eq:KHest}
\|\frak{K}\frak{H}, (\frak{K}\frak{H})' \|_{D^{2\alpha}_{W}} \leq C (|\frak{K}_{0}| 
+ |\frak{K}'_{0}| + 
\|\frak{K},\frak{K}'\|_{D^{2\alpha}_{W}} )
 (|\frak{H}_{0}| + |\frak{H}'_{0}| + \|\frak{H},\frak{H}'\|_{D^{2\alpha}_{W}} ).
\ee
Using~\eqref{one:g} in~\eqref{eq:KHest} we derive
\begin{align*}
&	\| G(Y) - G(\widetilde{Y}), (G(Y) - G(\widetilde{Y}))'\|_{D^{2\alpha}_{W}}\\ 
&\leq C\|G\|_{C^{3}_{\txtb}}  M^{2} (|Y_{0}-\widetilde{Y}_{0}| 
+ |Y'_{0} -\widetilde{Y}'_{0}| +  
\| Y- \widetilde{Y}, Y' - \widetilde{Y}'\|_{D^{2\alpha}_{W}} ),
\end{align*}
where again $C=C[\alpha,\|W\|_{\alpha}]$.
\end{proof}

Putting~\eqref{estg} and Lemma~\ref{differ} together, we get the following result:

\begin{lemma}
\label{gdiffusion}
Let $Y, \widetilde{Y}\in D^{2\alpha}_{W}([0,1];\cX)$, with $Y_0=\tilde{Y}_0$, $Y'_0=\tilde{Y}'_0$, $|Y_{0}| 
+ \|Y,Y'\|_{D^{2\alpha}_{W}}\leq M$ and $|\widetilde{Y}_{0}| 
+ \|\widetilde{Y},\widetilde{Y}'\|_{D^{2\alpha}_{W}}\leq M$. Further let
$G\in C^{3}_{\txtb}(\cX;\mathcal{L}(\cV,\cX))$. Then the following estimate holds true
\begin{align*}
&	\left\| \int\limits_{0}^{\cdot} S(\cdot - r) (G(Y_{r}) -G(\widetilde{Y}_{r} )) 
~\txtd\textbf{W}_{r}, G (Y) - G (\widetilde{Y}) \right\|_{D^{2\alpha}_{W}}\\ 
&\leq \|G(Y_{\cdot}) - G (\widetilde{Y}_{\cdot}) \|_{\alpha} + C[|A|]~
\|G\|_{C^{3}_{\txtb}}M^{2} \|Y-\widetilde{Y}, Y'-\widetilde{Y}'\|_{D^{2\alpha}_{W}} 
(1+\|W\|_{\alpha})(\|W\|_{\alpha}+\|\mathbb{W}\|_{2\alpha}).
\end{align*}
\end{lemma}
Recall that according to (\ref{esthoeldernormy}) one has 
\begin{align*}
\| G(Y)- G(\widetilde{Y})\|_{\alpha} \leq 
C(1+\|W\|_{\alpha}) \| G(Y)-G(\widetilde{Y}), 
(G(Y)- G(\widetilde{Y}))'\|_{D^{2\alpha}_{W}}.
\end{align*}
From this one can further infer that
\begin{align}
\label{finalestg}
&	\left\| \int\limits_{0}^{\cdot} S(\cdot - r) 
(G(Y_{r}) -G(\widetilde{Y}_{r} )) ~\txtd \textbf{W}_{r}, 
G (Y) - G (\widetilde{Y}) \right\|_{D^{2\alpha}_{W}}\nonumber\\ 
&\leq  C[|A|] \|G\|_{C^{3}_{\txtb}} M^{2} (1+\|W\|_{\alpha})(\|W\|_{\alpha}
+\|\mathbb{W}\|_{2\alpha}) \|Y-\widetilde{Y}, Y'-\widetilde{Y}'\|_{D^{2\alpha}_{W}}.
\end{align}
In addition to the last estimate, we also need another result to be employed 
later on in order to estimate the terms containing the initial condition and 
the drift of a rough differential equation in $D^{2\alpha}_{W}$. In this case, 
the computation simplifies since the Gubinelli derivative no longer plays a role.
We have:

\begin{lemma}
\label{drift} Let $\xi \in \cX$ and let $F:\cX \to \cX$ be Lipschitz continuous. 
Then for $(Y,Y')\in D^{2\alpha}_{W}([0,1];\cX)$, a Gubinelli derivative is given by
\begin{align}\label{gderivative:drfift}
\left( S(\cdot) \xi + \int\limits_{0}^{\cdot} S(\cdot-r)F(Y_{r})~\txtd r  \right)'= 0.
\end{align}
Moreover, we have the estimate
\begin{align}
\label{fdrift}
\left\|  S(\cdot)\xi + \int\limits_{0}^{\cdot} S(\cdot-r)F(Y_{r})~\txtd r, 
0 \right\|_{D^{2\alpha}_{W}} \leq  C[|A| ](|\xi|+|F(Y_{0})| + L_{F} \|Y\|_{\alpha} ),
\end{align}
where $L_{F}$ denotes the Lipschitz constant of $F$. Furthermore, for two controlled 
rough paths $(Y, Y')$ and $(\widetilde{Y}, \widetilde{Y}')$ with $Y_{0}=\xi$ and $\widetilde{Y}_{0}=\widetilde{\xi}$, one obtains
\begin{align}\label{diff:drift}
\left\| S(\cdot) (\xi-\widetilde{\xi}) + \int\limits_{0}^{\cdot}S(\cdot - r) 
(F(Y_{r})- F(\widetilde{Y}_{r} )) ~\txtd r  \right\|_{2\alpha} &\leq C [| A |]  
(|\xi-\widetilde{\xi}| + L_{F} \|Y-\widetilde{Y}\|_{\infty} ).
\end{align}
\end{lemma}

\begin{proof}
	In order to prove the statement we show that the term in~\eqref{gderivative:drfift} is $2\alpha$-H\"older regular. Letting $0\leq s < t \leq 1$ we first verify the regularity of the term containing the initial condition. Since we are working on finite-dimensional spaces, we immediately infer that
\begin{align*}
|(S(t)-S(s))\xi |\leq |S(s) |~ |S(t-s)-\mbox{Id}|~ |\xi| \leq e ^{s|A|}~ |A|  
~\txte^{(t-s)|A|} ~(t-s)~|\xi|,
\end{align*}
which gives us
\begin{align}\label{2alpha:abed}
\|S(\cdot{})\xi\|_{2\alpha} \leq C[|A|]~ |\xi| .
\end{align}
For the convolution with the nonlinear drift term, we obtain
\begin{align*}
& \left|\int\limits_{0}^{t} S(t-r)F(Y_{r}) ~\txtd r - \int\limits_{0}^{\txts}
S(s-r)F(Y_{r})~\txtd r\right|\\
& \leq \left|\int\limits_{0}^{\txts} (S(t-r) - S(s-r ))F(Y_{r}) ~\txtd r \right| 
+ \left|\int\limits_{s}^{t} S(t-r)F(Y_{r})~\txtd r \right|.
\end{align*}
This leads to 
\begin{align*}
  \left|\int\limits_{0}^{\txts} (S(t-r) - S(s-r ))F(Y_{r}) ~\txtd r \right| &\leq 
	|S(t-s)-\mbox{Id}|~ \int\limits_{0}^{\txts} |S(s-r) |~ |F(Y_{r})|~\txtd r\\
  & \leq |A|~ (t-s) ~\txte^{(t-s)|A|} ~\|F(Y)\|_{\infty}~ \int\limits_{0}^{\txts}
	\txte^{(s-r)|A|}~ \txtd r\\
  & \leq  C[|A|]~ (|F(Y_{0})| + L_{F}\|Y\|_{\alpha})(t-s).
\end{align*}
It is then straightforward to observe that
\begin{align*}
\left| \int\limits_{s}^{t} S(t-r) F(Y_{r})~\txtd r \right| 
&\leq C[|A|] (|F(Y_{0})| + L_{F}\|Y\|_{\alpha} ) (t-s),
\end{align*}
which implies
\begin{align*}
\left\| S(\cdot{})\xi + \int\limits_{0}^{\cdot} S(\cdot- r)F(Y_{r})~\txtd r
\right\|_{2\alpha} \leq  C[|A| ](|\xi|+|F(Y_{0})| + L_{F} \|Y\|_{\alpha} ).  
\end{align*}
Using that 
\begin{align*}
\left\| S(\cdot{})\xi + \int\limits_{0}^{\cdot} S(\cdot- r)F(Y_{r})~\txtd r
\right\|_{2\alpha} \leq  C[|A| ] (|\xi| + \|F(Y)\|_{\infty} )
\end{align*}
taking two controlled 
rough paths $(Y, Y')$ and $(\widetilde{Y}, \widetilde{Y}')$ with $Y_{0}=\xi$ and $\widetilde{Y}_{0}=\widetilde{\xi}$, one obtains~\eqref{diff:drift}.
\end{proof}

Regarding all these preliminary results one can then 
prove easily, in analogy to~\cite[Thm.~8.4]{FritzHairer}, 
existence and uniqueness results for~\eqref{sde1}. A local solution is derived 
by a fixed-point argument, which can be extended to a global one by a standard 
concatenation argument~\cite[Ch.~8]{FritzHairer}.

\begin{theorem}\label{globalexistence}
	There exists a unique element $(U,U')\in D^{2\alpha}_W$ such that $U'=G(U)$ and for $t\in[0,1]$
\begin{equation}\label{solution}
U_t= S(\cdot)\xi+ \int\limits_{0}^{\cdot} 
S(\cdot - r) F(U_{r}) ~\txtd r + \int\limits_{0}^{\cdot} 
S(\cdot - r) G(U_{r}) ~\txtd\textbf{W}_{r}.
\end{equation}
\end{theorem}

\begin{remark}
	\begin{itemize}
		\item [1)] The results stated here for $\alpha$-H\"older continuous 
		paths, carry over to $p$-variation paths, see \cite[Sec.~5.3]{FritzVictoir}, 
		\cite[Ch.~12]{FritzVictoir} or~\cite{CoutinLejay}.
		\item [2)] We stress the fact that the solution of~\eqref{sde1} is global-in-time. The previous results were stated only for $t\in[0,1]$ in order to derive uniform explicit estimates of the constants depending on the group $S$, the coefficients $F$ and $G$ and the random input $\textbf{W}$. These will be required in Section~\ref{lpm}. 
	\end{itemize}
\end{remark}

All the previous results can be formulated also in the context of regularity structures~\cite{Hairer}. A {\em controlled rough path} corresponds in the language of regularity structures to a {\em modelled distribution}~\cite[Sec.~13.2.2]{FritzHairer} and~\cite{Brault}.
For the sake of brevity we chose to work within the rough path framework.

\subsection{A cut-off technique}
\label{cut}
We conclude this subsection pointing out some results which are required 
for the computation of invariant manifolds for~\eqref{sde1}. In order to derive existence of invariant manifolds by means of fixed-point arguments, certain conditions involving the growth of the nonlinear terms are required (see e.g~\cite{BloemkerWang, 
	CaraballoDuanLuSchmalfuss, GarridoLuSchmalfuss}). Since these 
may be too restrictive and require a certain smallness of the nonlinearities, the first step is to introduce an appropriate 
cut-off technique in order to truncate these nonlinearities outside a ball around the origin. In contrast to the classical cut-off techniques, where one performs the truncation argument at the level of vector fields, we work here at the level of paths. More precisely, we employ an argument that truncates the controlled rough path norm of a pair $(Y,Y')$.\\

We additionally impose the following restrictions on the drift and diffusion coefficients:
\begin{description}
	\item [(\textbf{F})\label{f}] $F:\cX\to \cX\in C^{1}_{\txtb}$ is Lipschitz continuous with $F(0)=\txtD F(0)=0$;
	\item [(\textbf{G})\label{gi}] $G:\cX\to\mathcal{L}(\cV,\cX)\in C^{3}_{\txtb}$ is Lipschitz continuous with $G(0)=\txtD G(0)=\txtD^2 G(0)=0$.
\end{description}
We recall that the notation $\|\cdot\|_{\textbf{D}^{2\alpha}_{W}([0,1];\cX)}$ (short $\|\cdot\|_{\textbf{D}^{2\alpha}_{W}}$) stands for
 the norm on $D^{2\alpha}_{W}([0,1];\cX)$. Furthermore, we write $|||\cdot|||_{\alpha}$ for the $\alpha$-H\"older norm of an arbitrary path. For an element $(Y,Y')\in D^{2\alpha}_{W}([0,1];\cX)$, our aim is to truncate its $\|\cdot\|_{\textbf{D}^{2\alpha}_{W}}$ norm using a smooth (three-times differentiable with bounded derivatives) Lipschitz cut-off function.
This will be done as follows. \\
Let $\chi:D^{2\alpha}_{W}([0,1];\cX)\to D^{2\alpha}_{W}([0,1];\cX)$ be defined as 
\begin{align*}
\chi(Y):= \begin{cases} Y, & \|Y,Y'\|_{\textbf{D}^{2\alpha}_{W}}\leq 1/2\\
0, & \|Y,Y'\|_{\textbf{D}^{2\alpha}_{W}} \geq 1.
\end{cases}
\end{align*}
Such a function can be constructed as
\begin{align*}
\chi(Y)= Y f \Big(\|Y,Y'\|_{\textbf{D}^{2\alpha}_{W}}\Big), 
\end{align*}	
where $f:\mathbb{R}_{+}\to[0,1]$ is a $C^{3}_{\txtb}$ Lipschitz cut-off function. For instance $f$ can be obtained by applying a smooth convolution to
\begin{align*}
\tilde{f}(x):=\begin{cases}
1, &x\leq 1/2\\
2- 2x, &x\in(1/2,1)\\
0, &x\geq 1.
\end{cases}
\end{align*}
Another example in this sense is given by
\begin{align*}
\tilde{f}(x):=\begin{cases}
e^{-1/x}, &x>0\\
0, &x\leq 0
\end{cases}
\end{align*}
and
\begin{align*}
f(x):=\frac{\tilde{f}(1-x)}{\tilde{f}(x-1/2) + \tilde{f}(1-x) }.
\end{align*}
Using that $\chi(Y)=Y f (\|Y,Y'\|_{\textbf{D}^{2\alpha}_{W}})$ 
we obtain
\begin{align*}
(\chi(Y))' = Y'  f (\|Y,Y'\|_{\textbf{D}^{2\alpha}_{W}}),
\end{align*}
since $f(\|Y,Y'\|_{\textbf{D}^{2\alpha}_{W}})$ is a {\em constant with respect to time}, therefore we can choose its Gubinelli derivative to be zero.
This construction indicates that
\begin{align*}
(\chi(Y), (\chi(Y))') = \begin{cases}
(Y,Y') &: \|Y,Y'\|_{\textbf{D}^{2\alpha}_{W}} \leq 1/2;\\
0 & : \|Y,Y'\|_{\textbf{D}^{2\alpha}_{W}} \geq 1.
\end{cases}
\end{align*}
For a positive number $R>0$ we now introduce
\begin{align*}
\chi_{R}(Y):= R \chi(Y/R).
\end{align*}
This means that
\begin{align*}
\chi_R(Y) =\begin{cases}
Y &: \|Y,Y'\|_{\textbf{D}^{2\alpha}_{W}} \leq R/2;\\
0 &: \|Y,Y'\|_{\textbf{D}^{2\alpha}_{W}} \geq R.
\end{cases}
\end{align*}
Since $\chi_{R}(Y)= Y f( \|Y,Y'\|_{\textbf{D}^{2\alpha}_{W}}/R)$ we have that
$$(\chi_{R}(Y))'= Y' f\Bigg( \|Y,Y'\|_{\textbf{D}^{2\alpha}_{W}}/R\Bigg),$$
therefore
\begin{align*}
(\chi_{R}(Y), (\chi_R(Y))' ) =\begin{cases}
(Y,Y') &: \|Y,Y'\|_{\textbf{D}^{2\alpha}_{W}} \leq R/2;\\
0 &: \|Y,Y'\|_{\textbf{D}^{2\alpha}_{W}} \geq R.
\end{cases}
\end{align*}

\begin{remark}
	In order to avoid this abstract setting, one could also modify the $\textbf{D}^{2\alpha}_{W}$-norm in several steps. For the convenience of the reader we shortly indicate this procedure as well. The main idea is to modify separately the $\alpha$-H\"older norm of $Y'$, the $2\alpha$-H\"older norm of the remainder $R^{Y}$ and $|Y_{0}|$. More precisely, for a given path $(Y,Y')\in D^{2\alpha}_{W}$, one can perform a truncation argument to obtain $|||Y'|||_{\alpha} + |||R^{Y}|||_{2\alpha}  \leq 1/2$ as follows.\\
	Let $\chi_1:C^{\alpha}([0,1];\cX)\to C^{\alpha}([0,1];\cX)$ be defined as
	\begin{align*}
	\chi_{1}(u):=\begin{cases} u&: |||u|||_{\alpha} \leq 1/4;\\
	0&: |||u|||_{\alpha}\geq 1.
	\end{cases}
	\end{align*}
	Such a function can be analogously constructed as
	\begin{align*}
	\chi_{1}(u)= u f_{1}(|||u|||_{\alpha}),
	\end{align*}
	where $f_{1}:\mathbb{R}_{+}\to[0,1]$ is as before a $C^{3}_{\txtb}$ cut-off function, e.g.~$f_{1}$ can be obtained by applying a smooth convolution to
	\begin{align*}
	\tilde{f}_1(x):=\begin{cases}
	1, &x\leq 1/4\\
	\frac{4}{3}-\frac{4}{3}x, &x\in(1/4,1)\\
	0, &x\geq 1.
	\end{cases}
	\end{align*}
	Similarly, we define $\chi_{2}:C^{2\alpha}([0,1];\cL(\cV,\cX))\to C^{2\alpha}([0,1];\cL(\cV,\cX))$ as
	\begin{align*}
	\chi_{2}(u)= u f_{2}(|||u|||_{2\alpha}),
	\end{align*}
	where again
	$f_2:\mathbb{R}_{+}\to [0,1]$ is a $C^{3}_{b}$ cut-off function analogously constructed as $f_1$. Keeping this in mind we define for $(Y,Y')\in D^{2\alpha}_{W}([0,1];\cX)$ 
	\begin{align*}
	\chi(Y):= Y f_{1}(|||Y'|||_{\alpha}) f_{2}(|||R^{Y}|||_{2\alpha}).
	\end{align*}
    with Gubinelli derivative
	\begin{align*}                                                           
	(\chi(Y))'= Y' f_{1}(|||Y'|||_{\alpha}) f_{2}(|||R^{Y}|||_{2\alpha}).
	\end{align*}	
This represents another possibility to modify the $\|\cdot\|_{\textbf{D}^{2\alpha}_{W}}$-norm.
\end{remark}

Our next aim is to consider a modified version of~\eqref{sde1}, using the cut-off function $\chi_{R}$ defined above. To this aim, we firstly introduce the following notations. We recall that $(Y,Y')\in D^{2\alpha}_{W}([0,1];\cX)$ and drop for simplicity the index from the notation $\|\cdot\|_{\textbf{D}^{\alpha}_{W}([0,1];\cX)}$.
Furthermore $F$ and $G$ satisfy $\textbf{(F)}$ respectively $\textbf{(G)}$. We set for $t\in [0,1]$:
$$\overline{F}(Y)(t): = F(Y_t) \mbox { and } \overline{G}(Y)(t):=G(Y_t).$$
With this notation, the first component of the mild solution of~\eqref{sde1} equivalently rewrites as
\begin{align}\label{equiv:rde}
U_t = S(t)\xi + \int\limits_{0}^{t}S(t-r)\overline{F}(U)(r)~\txtd r + \int\limits_{0}^{t}S(t-r)\overline{G}(U)(r)~\txtd \textbf{W}_r. 
\end{align}
For a controlled rough path $(Y,Y')\in D^{2\alpha}_{W}([0,1];\cX)$,  we now introduce the truncation as
$$ F_R(Y):= \overline{F}\circ \chi_{R}(Y) \mbox{ respectively } G_R:=\overline{G}\circ \chi_{R}(Y).$$
This means that we have $$F_{R}(Y)(t) = \overline{F}(\chi_{R}(Y))(t)= F(\chi_{R}(Y)_t )= F(Y_t f (\|Y,Y'\|/R))$$
and analogously, $$G_{R}(Y)(t) = \overline{G}(\chi_{R}(Y))(t)= G(\chi_{R}(Y)_t  ) = G(Y_tf(\|Y,Y'\|/R) ),$$
where we write for simplicity $f(\|Y,Y'\|/R)=f(\|Y,Y'\|_{\textbf{D}^{2\alpha}_{W}}/R)$. 
The Gubinelli derivative of $G_R$ can be computed according to the chain rule~\cite[Lem.~7.3]{FritzHairer} as
\begin{align*}
(G_{R}(Y))' = \txtD G (\chi_{R}(Y)) (\chi_{R}(Y))' = \txtD G (Yf (\|Y,Y'\|/R)) Y' f(\|Y,Y'\|/R),
\end{align*}
since $f$ is a constant with respect to time.\\
Evaluating $(G_R(Y))'$ at a time $t\in[0,1]$, we obtain
\begin{align*}
(G_{R}(Y)(t))' = \txtD G (\chi_{R}(Y)_t) (\chi_{R}(Y)_t)' = \txtD G (Y_tf (\|Y,Y'\|/R)) Y'_t f(\|Y,Y'\|/R).
\end{align*}
By the definition of $\chi_{R}$ we have that $F_{R}(Y)=\overline{F}(Y)$ and $G_{R}(Y)=\overline{G}(Y)$ if $\|(Y,Y')\|_{\textbf{D}^{2\alpha}_{W}([0,1];\cX)}\leq R/2$.\\

We now argue in several steps, why one can replace $\overline{F}$ and $\overline{G}$ by $F_R$ and $G_R$ in the RDE~\eqref{sde1} (equivalently in~\eqref{equiv:rde}). We emphasize the fact that $F:\cX\to \cX$ and $G:\cX \to \mathcal{L}(\cV, \cX) $ whereas $F_R$ and $G_R$ as defined above, are path-dependent. To this aim, we firstly show that $F_R$ and $G_R$ are Lipschitz continuous.\\
Before we proceed with the next steps, we point out the following inequality that will frequently be used in this section, under the additional assumptions on the derivatives of $G$, compare Lemma~\ref{differ}. Let $H:\cX \to \cX$ be a twice differentiable function with Lipschitz continuous derivatives and $H(0)=DH(0)=0$. 
For $x,y\in\cX$, one has
\begin{align*}
H(x) - H(y)= (x-y) \int\limits_{0}^{1} \txtD H (rx+(1-r)y)~\txtd r.
\end{align*}
Now, since $\txtD H$ is Lipschitz and $\txtD H(0)=0$, we further obtain that
\begin{align}
|H(x) - H(y) | &\leq \int\limits_{0}^1|\txtD H(rx+(1-r)y )|~\txtd r ~|x-y|\nonumber\\
& \leq C_{H} \max\{|x|,|y|\} |x-y|, \label{lip:g}
\end{align}
where the constant $C_{H}>0$  depends only on the Lipschitz constants of $H$, $\txtD H$ and $\txtD^2 H$.
Note that the additional assumption $\txtD H(0)=0$ yields~\eqref{lip:g}, instead of only
\begin{align*}
|H(x) - H(y) | \leq \|\txtD H\|_{\infty} |x-y|\leq \|H\|_{C^1_{\txtb}} |x-y|,
\end{align*}
as in the proof of Lemma~\ref{differ}. Naturally, if additionally $\txtD^2H(0)=0$, one can estimate the difference $\txtD H(x)-\txtD H(y)$ by the same arguments using the Lipschitz continuity of the second derivative instead of directly using the $C^2_{\txtb}$-norm of $H$. Namely,
\begin{align*}
\txtD H(x) -\txtD H(y) = (x-y) \int\limits_{0}^1\txtD^2 H(rx+(1-r)y)~\txtd r,
\end{align*}
one has
\begin{align}\label{der}
|\txtD H(x)- \txtD H(y)| \leq C_{H}\max\{ |x|,|y|\}|x-y|.
\end{align}

We use the previous deliberations for $F_R$ and $G_R$. As a next step, we show that $F_R$ is Lipschitz continuous.

\begin{lemma}\label{lip:fr}
	Let $(Y,Y'), (\widetilde{Y},\widetilde{Y}')\in D^{2\alpha}_W([0,1];\cX)$. Then there exists a constant $C=C[F,\chi]$ such that
	\begin{align}\label{fr}
	\|F_{R}(Y) - F_{R}(\widetilde{Y})\|_{\infty}  &\leq C R \|Y-\widetilde{Y}, Y'-\widetilde{Y}'\|_{\textbf{D}^{2\alpha}_W}.
	\end{align}
\end{lemma}
\begin{proof}
	We have
	\begin{align*}
	\sup\limits_{t\in[0,1]} |F_R(Y)(t) - F_{R}(\widetilde{Y})(t)| &=\sup\limits_{t\in[0,1]} |F(\chi_{R}(Y)_t) - F(\chi_{R}(\widetilde{Y})_t) |.
	\end{align*}
	Furthermore,
		\begin{align}
	|F(\chi_{R}(Y)_t) - F(\chi_{R}(\widetilde{Y})_t) |&	\leq \int\limits_{0}^{1} |\txtD F(r\chi_{R}(Y)_t + (1-r)\chi_R(\widetilde{Y})_t)|~\txtd r~ |\chi_{R}(Y)_t -\chi_{R}(\widetilde{Y})_t|\nonumber\\
	&\leq C\max \{|\chi_{R}(Y)_t|, |\chi_{R}(\widetilde{Y})_t|\} |\chi_R(Y)_t - \chi_R(\widetilde{Y})_t|\nonumber\\
	& \leq C R |\chi_R(Y)_t - \chi_R(\widetilde{Y})_t|.\label{here}
	\end{align}
	Here we use that $ |\chi_{R}(Y)_t| \leq \|\chi_{R}(Y)\|_{\infty} =\|Y f(\|Y,Y'\|/R)\|_{\infty}=\|Y\|_{\infty} f(\|Y,Y'\|/R)$. 
	Furthermore
	\begin{align*}
	|\chi_R(Y)_t - \chi_R(\tilde{Y})_t| & = |Y_t f(\|Y,Y'\|/R) - \tilde{Y}_t f(\|\tilde{Y},\tilde{Y}'\|/R)|\\
	& \leq |Y_t -\tilde{Y}_t| f(\|Y,Y'\|/R) +|\tilde{Y}_t|~ |f(\|Y,Y'\|/R) - f(\|\tilde{Y},\tilde{Y}'\|/R) |\\
	& \leq \|Y-\tilde{Y}\|_{\infty} + R \|\txtD f\|_{\infty} (\|Y,Y'\|/R - \|\tilde{Y},\tilde{Y}'\|/R)\\
	& \leq (1+\|\txtD f\|_{\infty})  \|Y -\widetilde{Y}, Y'-\widetilde{Y}'\|_{\textbf{D}^{2\alpha}_{W}}.
	\end{align*}
	Consequently, this further yields
	\begin{align*}
	\sup\limits_{t\in[0,1]} |F_R(Y)(t) - F_{R}(\widetilde{Y})(t) |&\leq C R \|\chi_{R}(Y) -\chi_{R}(\widetilde{Y})\|_{\infty} 
	 \leq C [F,\chi] ~R~ \|Y -\widetilde{Y}, Y'-\widetilde{Y}'\|_{\textbf{D}^{2\alpha}_{W}},
	\end{align*}
	which proves the statement.
\end{proof}
We now focus on the pair $(G_R, G'_R)$.
We let $C_{G}$ stand for a universal constant which depends only on the Lipschitz constants of $G$, $\txtD G$ and $\txtD^2 G$. The aim is eventually to make $R$ as small as possible, therefore ensure that such a choice decreases the Lipschitz constants of $\overline{F}\circ\chi_R$ and $\overline{G}\circ\chi_R$. For $\overline{F}\circ\chi_R$ this holds true due to~\eqref{fr}. We aim to prove a similar assertion for $\overline{G}\circ\chi_R$.
Therefore, we write $C(R)$ to point out the dependence on the cut-off parameter $R$.  Due to the observation above, it is crucial that $C(R)\to 0 $ as $R\to 0$.
\begin{lemma}\label{composition:cutoff}
	Let $(Y,Y'), (\widetilde{Y},\widetilde{Y}')\in D^{2\alpha}_{W}([0,1];\cX)$. Then there exists a constant $C(R)=C[R,\|W\|_{\alpha},G,\chi]$ such that $C(R)\to 0$ as $R\to 0$ and that
	\begin{align}\label{composition:r}
	\|G_{R}(Y) - G_R(\widetilde{Y}), (G_{R}(Y)- G_R(\widetilde{Y}))'\|_{\textbf{D}^{2\alpha}_{W}([0,1];\cX)} \leq C (R) \|Y-\widetilde{Y}, Y'-\widetilde{Y}'\|_{\textbf{D}^{2\alpha}_{W}([0,1];\cX)} .
	\end{align}
\end{lemma}
\begin{proof}

We let $s,t\in[0,1]$ and estimate $\|G_R(Y) -G_R(\widetilde{Y})\|_{\alpha}=\|\overline{G}(\chi_R(Y)) -\overline{G}(\chi_{R}(\widetilde{Y}))\|_{\alpha}$.
We have
\begin{align*}
|G_R(Y)(t) - G_R(\widetilde{Y})(t) - [G_{R}(Y)(s) - G_{R}(\widetilde{Y})(s) ]| =| G(\chi_{R}(Y)_t)  - G(\chi_{R}(\widetilde{Y})_t) - [G(\chi_{R}(Y)_s)  - G(\chi_{R}(\widetilde{Y})_s)] |.
\end{align*}
Applying~\eqref{lip:g} twice and using that $\txtD G(0)=0$ leads to
\begin{align}
\begin{split}\label{subg1}
&| G(\chi_{R}(Y)_t)  - G(\chi_{R}(\widetilde{Y})_t) - [G(\chi_{R}(Y)_s)  - G(\chi_{R}(\widetilde{Y})_s)] |\\
& \leq C_G \max\{ |\chi_{R}(Y)_t, \chi_{R}(\widetilde{Y}_t)| \}~ |\chi_{R}(Y)_t - \chi_{R}(\widetilde{Y})_t -  [ \chi_{R}(Y)_s -\chi_{R}(\widetilde{Y})_s ]|  \\
&\hspace*{10 mm}+ C_G ~|\chi_{R}(Y)_t -\chi_{R}(\widetilde{Y}_t)|~ [|\chi_{R}(Y)_t -\chi_{R}(Y)_s| + |\chi_{R}(\widetilde{Y})_t -\chi_{R}(\widetilde{Y})_s| ]\\
\end{split}\\
&\leq C_G R \|\chi_R (Y) -\chi_{R}(\widetilde{Y})\|_{\alpha} + C_{G} \|\chi_{R}(Y) - \chi_{R}(\widetilde{Y})\|_{\infty} ( \|\chi_{R}(Y)\|_{\alpha} +\|\chi_{R}(\widetilde{Y})\|_{\alpha} )\nonumber\\
& \leq C_G R \|Y-\widetilde{Y}, Y'-\widetilde{Y}'\|_{\textbf{D}^{2\alpha}_W}. \nonumber
\end{align}
Here we used again that $ |\chi_{R}(Y)_t| \leq \|\chi_{R}(Y)\|_{\infty} =\|Y f(\|Y,Y'\|/R)\|_{\infty}=\|Y\|_{\infty} f(\|Y,Y'\|/R)$.\\
We further recall  that \begin{align*}
(G_{R}(Y))' = \txtD G (\chi_{R}(Y)) (\chi_{R}(Y))' = \txtD G (Yf (\|Y,Y'\|/R)) Y' f(\|Y,Y'\|/R).
\end{align*} 
We now investigate $\|(G_{R}(Y) -G_R(\widetilde{Y}))'\|_{\alpha}$ and have to estimate
\begin{align*}
&[G_R (Y)(t) -G_R(\widetilde{Y})(t) - (G_{R}(Y)(s) -G_R(\widetilde{Y})(s)) ]'\\
& =[\txtD G(\chi_{R}(Y)_t)(\chi_{R}(Y)_t)' - \txtD G(\chi_R(\widetilde{Y})_t)(\chi_R(\widetilde{Y})_t)' - (\txtD G(\chi_R(Y)_s ) (\chi_R(Y)_s)' - \txtD G(\chi_R(\widetilde{Y})_s)(\chi_R(\widetilde{Y})_s)' )] \\
& = \txtD G(Y_t f(\|Y,Y'\|/R) )Y'_t f(\|Y,Y'\|/R) - \txtD G(\widetilde{Y}_t f (\|\widetilde{Y},\widetilde{Y}'\|/R) ) \widetilde{Y}'_t  f (\|\widetilde{Y},\widetilde{Y}'\|/R) \\
&   - (\txtD G(Y_s f(\|Y,Y'\|/R) )Y'_s f(\|Y,Y'\|/R) - \txtD G(\widetilde{Y}_s f (\|Y,Y'\|/R) ) \widetilde{Y}'_s  f (\|\widetilde{Y},\widetilde{Y}'\|/R)).
\end{align*}
Therefore we estimate
\begin{align*}
&|[\txtD G(\chi_{R}(Y)_t)(\chi_{R}(Y)_t)' - \txtD G(\chi_R(\widetilde{Y})_t)(\chi_R(\widetilde{Y})_t)' - (\txtD G(\chi_R(Y)_s ) (\chi_R(Y)_s)' - \txtD G(\chi_R(\widetilde{Y})_s)(\chi_R(\widetilde{Y})_s)' )]|\\
& \leq | \txtD G(\chi_R(Y)_t) - \txtD G(\chi_R(Y)_s) - \txtD G(\chi_R(\widetilde{Y})_t) + \txtD G (\chi_R(\widetilde{Y})_s) | ~|(\chi_R(Y)_t)' + (\chi_R(Y)_s)' + (\chi_R(\widetilde{Y}_t))' + (\chi_R(\widetilde{Y})_s)' |\\
& \leq | \txtD G(\chi_R(Y)_t) - \txtD G(\chi_R(Y)_s) + \txtD G(\chi_R(\widetilde{Y})_t) - \txtD G (\chi_R(\widetilde{Y})_s) | ~|(\chi_R(Y)_t)' + (\chi_R(Y)_s)' - (\chi_R(\widetilde{Y}_t))' - (\chi_R(\widetilde{Y})_s)' |\\
& \leq | \txtD G(\chi_R(Y)_t) + \txtD G(\chi_R(Y)_s) -\txtD G(\chi_R(\widetilde{Y})_t) - \txtD G (\chi_R(\widetilde{Y})_s) | ~|(\chi_R(Y)_t)' - (\chi_R(Y)_s)' + (\chi_R(\widetilde{Y}_t))' - (\chi_R(\widetilde{Y})_s)' |\\
& \leq | \txtD G(\chi_R(Y)_t) + \txtD G(\chi_R(Y)_s) + \txtD G(\chi_R(\widetilde{Y})_t) + \txtD G (\chi_R(\widetilde{Y})_s) | ~|(\chi_R(Y)_t)' - (\chi_R(Y)_s)' - (\chi_R(\widetilde{Y}_t))' + (\chi_R(\widetilde{Y})_s)' |\\
& := I + II + III +IV.
\end{align*}
We analyze each of the four terms above separately. For the first one we obtain using~\eqref{subg1} for $\txtD G$ instead of $G$ and~\eqref{esthoeldernormy} that
\begin{align*}
|I|&\leq C_G R \|Y-\widetilde{Y}, Y'-\widetilde{Y}'\|_{\textbf{D}^{2\alpha}_W} [\|\chi_R(Y)\|_{\infty}+ \|(\chi_R(Y))'\|_{\infty} ]\\
& \leq  C R \|Y-\widetilde{Y}, Y'-\widetilde{Y}'\|_{\textbf{D}^{2\alpha}_W}  [|||\chi_R(Y)|||_{\alpha}+ |||(\chi_R(Y))'|||_{\alpha}]\\
& \leq  C R^2 \|Y-\widetilde{Y}, Y'-\widetilde{Y}'\|_{\textbf{D}^{2\alpha}_W}.
\end{align*}
The second one results in
\begin{align*}
|II|&\leq  C \|\chi_R(Y) -\chi_R(\widetilde{Y})\|_{\alpha} [\|\chi_R(Y)\|_{\infty} +\|\chi_R(\widetilde{Y})\|_{\infty} ]\\
& \leq C R \|Y - \widetilde{Y}, Y'-\widetilde{Y}'\|_{\textbf{D}^{2\alpha}_W}.
\end{align*} 
Analogously, for the third one we have
\begin{align*}
|III|& \leq C \|\chi_R (Y)-\chi_R(\widetilde{Y})\|_{\infty} [\|\chi_R(Y)\|_{\alpha} +\|\chi_R(\widetilde{Y})\|_{\alpha}]\\
& \leq C R \|Y - \widetilde{Y}, Y'-\widetilde{Y}'\|_{\textbf{D}^{2\alpha}_W}. 
\end{align*}
Finally, we easily estimate the fourth term regarding that $\txtD^2 G(0)=0$ as
\begin{align*}
|IV| &\leq C \sup\limits_{t\in[0,1]}\max\{|\txtD G (\chi_R(Y)_t)|,|\txtD G(\chi_R(\widetilde{Y})_t)| \}  \|(\chi_R(Y))' - (\chi_R(\widetilde{Y}))'\|_{\alpha}\\
& \leq C R \|Y - \widetilde{Y}, Y'-\widetilde{Y}'\|_{\textbf{D}^{2\alpha}_W}.
\end{align*}
We now focus on the remainder of $G_{R}(Y) - G_{R}(\widetilde{Y})$.
The remainder $G_R(Y)$ satisfies the relation
\begin{align*}
(R^{G_R(Y)})_{s,t}= G(\chi_{R}(Y)_t) - G(\chi_R(Y)_s) - \txtD G(\chi_{R}(Y)_s)(\chi_{R}(Y)_s)'W_{s,t}.
\end{align*}
Since $(Y,Y')\in D^{2\alpha}_W$, we know that
$$Y_t = Y_s + Y'_s W_{s,t} + R^{Y}_{s,t}.$$ Therefore,
$$Y_t f(\|Y,Y'\|/R) = Y_s f(\|Y,Y'\|/R) + Y'_s f(\|Y,Y'\|/R) W_{s,t} + R^{Y}_{s,t}f(\|Y,Y'\|/R),$$
which gives us $R^{\chi_{R}(Y)} = R^{Y} f(\|Y,Y'\|/R)$.
Consequently,
\begin{align}\label{rem:cutoff}
(R^{G_R(Y)})_{s,t}= G(\chi_{R}(Y)_t) - G(\chi_R(Y)_s) - 
\txtD G(\chi_{R}(Y)_s)(\chi_{R}(Y))_{s,t} + \txtD G (\chi_{R}(Y)_s)R^{Y}_{s,t} f(\|Y,Y'\|/R) .
\end{align}
This means that we have to estimate the $2\alpha$-H\"older norm of 
\begin{align}
R ^{G_{R}(Y)}_{s,t} - R^{G_{R}(\widetilde{Y})}_{s,t} &= G(\chi_{R}(Y)_t) - G(\chi_R(Y)_s) - 
\txtD G(\chi_{R}(Y)_s)(\chi_{R}(Y))'_{s,t}\nonumber \\&- [G(\chi_{R}(\widetilde{Y})_t) - G(\chi_R(\widetilde{Y})_s) - 
\txtD G(\chi_{R}(\widetilde{Y})_s)(\chi_{R}(\widetilde{Y}))'_{s,t}]\nonumber\\
& + \txtD G (\chi_{R}(Y)_s)R^{Y}_{s,t} f(\|Y,Y'\|/R)  - \txtD G (\chi_{R}(\widetilde{Y})_s)R^{\widetilde{Y}}_{s,t} f(\|\widetilde{Y},\widetilde{Y}'\|/R). \label{last:rem}
\end{align}
The terms in~\eqref{last:rem} easily result in
\begin{align*}
&|\txtD G (\chi_R(Y)_s) R^{\chi_R(Y)}_{s,t} -\txtD G (\chi_R(\widetilde{Y})_s) R^{\chi_R(\widetilde{Y})}_{s,t}| \\
& \leq |\txtD G (\chi_R(Y)_s) [R^{\chi_R}(Y)_{s,t} - R^{\chi_R(\widetilde{Y})}_{s,t} ] |+ |[\txtD G (\chi_R(Y)_s) -\txtD G(\chi_R(\widetilde{Y})_s)] R^{\chi_R(\widetilde{Y})}_{s,t}|\\
& \leq C R \| R^{\chi_R}(Y) - R^{\chi_R(\widetilde{Y})}\|_{2\alpha} + C_{G} \|R^{\chi_R(\widetilde{Y})}\|_{2\alpha} \|\chi_R(Y)_s -\chi_R(\widetilde{Y})_s\|\\
& \leq C R \|R^{Y} -R^{\widetilde{Y}}\|_{2\alpha} + C_G R \|Y - \widetilde{Y}\|_{\infty}.
\end{align*}

To estimate the $2\alpha$-H\"older norm of the expressions~\eqref{last:rem} appearing in the difference of two remainders, we firstly recall the identity
\begin{align}
G(Y_t) - G(Y_s) - \txtD G(Y_s)Y_{s,t}& =\int\limits_{0}^{1} \txtD G (r Y_t +(1-r) Y_s )~\txtd r~ Y_{s,t} - \txtD G(Y_s)Y_{s,t} \label{dg} \\
& = \int\limits_{0}^{1} [\txtD G (r Y_t +(1-r) Y_s ) - \txtD G(Y_s)] Y_{s,t}~\txtd r\nonumber\\
& = \int\limits_{0}^{1} \int\limits_{0}^1 \txtD^2 G [\tau (r Y_t +(1-r) Y_s)+(1-\tau)Y_s ] (rY_t+(1-r)Y_s-Y_s) Y_{s,t}~\txtd\tau~ \txtd r\nonumber\\
&=\int\limits_{0}^1\int\limits_{0}^1 r \txtD^2 G [\tau (r Y_t +(1-r) Y_s)+(1-\tau)Y_s ]~\txtd\tau~\txtd r ~(Y_{s,t})^2.\nonumber
 \end{align}
 Applying~\eqref{dg} twice, one obtains the following inequality
 \begin{align}
 &|G(\chi_{R}(Y)_t) - G(\chi_R(Y)_s) - DG(\chi_{R}(Y)_s)(\chi_{R}(Y))_{s,t} - (G(\chi_{R}(\widetilde{Y})_t) - G(\chi_R(\widetilde{Y})_s) - DG(\chi_{R}(\widetilde{Y})_s)(\chi_{R}(\widetilde{Y}))_{s,t}) | \nonumber\\
 \begin{split}
 & \leq C_{G}|\chi_{R}(Y)_t| ~[|\chi_{R}(Y)_t -\chi_{R}(Y)_s| +|\chi_{R}(\widetilde{Y})_t -\chi_{R}(\widetilde{Y})_s | ] ~|\chi_{R}(Y)_t -\chi_{R}(Y)_s - (\chi_{R}(\widetilde{Y})_t -\chi_{R}(\widetilde{Y})_s)|\label{2}\\
 &+ C_G |\chi_{R}(\widetilde{Y})_t -\chi_{R}(\widetilde{Y})_s|^2 ~[|\chi_{R}(Y)_t -\chi_{R}(\widetilde{Y}_t) | + |\chi_{R}(Y)_s -\chi_{R}(\widetilde{Y})_s |].
 \end{split}
 \end{align}
This further leads to 
\begin{align*}
&|G(\chi_{R}(Y)_t) - G(\chi_R(Y)_s) - DG(\chi_{R}(Y)_s)(\chi_{R}(Y))_{s,t} - (G(\chi_{R}(\widetilde{Y})_t) - G(\chi_R(\widetilde{Y})_s) - DG(\chi_{R}(\widetilde{Y})_s)(\chi_{R}(\widetilde{Y}))_{s,t}) |\\
& \leq C R  [\|Y\|_{\alpha} + \|\widetilde{Y}\|_{\alpha}]~ \|Y-\widetilde{Y}\|_{\alpha} + C\|\widetilde{Y}\|^{2}_{\alpha} \|Y-\widetilde{Y}\|_{\infty} \\
&\leq  C R^2 \|Y-\widetilde{Y}, Y'- \widetilde{Y}'\|_{\textbf{D}^{2\alpha}_W}.
\end{align*} 
 \end{proof}
We now justify that the modified equation~\eqref{sde1} obtained by replacing $F$ and $G$ with $F_R$ and $G_R$ has a unique solution in the appropriate space of controlled rough paths. For notational simplicity we introduce
 for $(Y,Y')\in D^{2\alpha}_{W}([0,1];\cX)$ and $t\in[0,1]$
\begin{align}\label{tr}
T_{R}({W}, Y,Y')[t]:&=  \int\limits_{0}^{t} S(t-r) 
\overline{F}(\chi_R(Y))(r) ~\txtd r + \int\limits_{0}^{t} S(t-r) \overline{G}(\chi_R(Y))(r) 
~\txtd \bm{{W}}_{r}\\
& = \int\limits_{0}^{t} S(t-r) 
F_R(Y)(r) ~\txtd r + \int\limits_{0}^{t} S(t-r) G_R(Y)(r) 
~\txtd \bm{{W}}_{r} \nonumber,
\end{align}
with Gubinelli derivative $(T_{R}(W,Y,Y'))'=G_{R}(Y)$.
Collecting the estimates derived in the previous Lemmas leads to the following result.

\begin{theorem}\label{fpr}
There exists a unique element $(Y,Y')\in D^{2\alpha}_W([0,1];\cX)$ with $Y'=G_R(Y)$ such that for $t\in[0,1]$
	\begin{align*}
Y_t = \int\limits_{0}^{t} S(t-r) 
F_R(Y)(r) ~\txtd r + \int\limits_{0}^{t} S(t-r) G_R(Y)(r) 
~\txtd \bm{{W}}_{r}.
\end{align*} 
\end{theorem}
\begin{proof}
	The proof relies on Banach's fixed-point theorem and is similar to~\cite[Thm.~8.4]{FritzHairer}. In our situation $G_R$ depends on the path itself. However, as previously shown, the cut-off procedure does not modify the time-regularity of the coefficients. Moreover, the Lipschitz continuity showed in Lemmas~\ref{fr} and~\ref{composition:cutoff} allows us to perform a fixed-point argument and prove the existence of a solution for the modified equation~\eqref{sde1} (equivalently~\eqref{equiv:rde}), in which one replaces $\overline{F}$ and $\overline{G}$ by $F_R$ and $G_R$.\\
	Let $(Y,Y'), (\widetilde{Y},\widetilde{Y}')\in D^{2\alpha}_{W}([0,1];\cX)$ with $Y_0=\widetilde{Y}_0$ and $Y'_0=\tilde{Y}'_0$. 
	Using~\eqref{fr}, we obtain as in Lemma~\ref{drift} that
	\begin{align*}
\Bigg\| \int\limits_{0}^{\cdot} S(\cdot-r) [F(\chi_R(Y))(r) - F(\chi_R(\widetilde{Y}))(r)]~\txtd r\Bigg\|_{2\alpha}
&\leq  C[A]~ \|F(\chi_R(Y)) - F(\chi_R(\widetilde{Y}))\|_{\infty} \\
&\leq C[A]~ R ~\|Y - \widetilde{Y}, Y' -\widetilde{Y}'\|_{\textbf{D}^{2\alpha}_{W}}.
	\end{align*}
	Using Lemma~\ref{composition:cutoff} and~\eqref{if} applied to $(G_R(Y)-G_R(\widetilde{Y}), (G_R(Y) -G_{R}(\widetilde{Y}) )')\in D^{2\alpha}_W$ we obtain
	\begin{align*}
	&\Bigg\|\int\limits_{0}^{\cdot} S(\cdot -r) [G_R(Y)(r)-G_R(\widetilde{Y})(r)]~\txtd\textbf{W}_r, G_R(Y)-G_R(\widetilde{Y})  \Big\|_{\textbf{D}^{2\alpha}_{W}}\\
	&\leq C [|A|] \| G_R(Y)-G_R(\widetilde{Y}), (G_R(Y) -G_{R}(\widetilde{Y}) )'\|_{\textbf{D}^{2\alpha}_{W}}  (1+\|W\|_{\alpha}) (\|W\|_{\alpha} +\|\mathbb{W}\|_{2\alpha} )\\
	& \leq C[|A|] ~C(R) (1+\|W\|_{\alpha}) (\|W\|_{\alpha} +\|\mathbb{W}\|_{2\alpha} )\|Y - \widetilde{Y}, Y' -\widetilde{Y}'\|_{\textbf{D}^{2\alpha}_{W}}.
	\end{align*}
Putting all these estimates together we infer
	\begin{align}\label{ef}
	&	\left\|\int\limits_{0}^{\cdot} S(\cdot-r) (F_R(Y)(r) - F_R(\widetilde{Y})(r))~\txtd r + \int\limits_{0}^{\cdot} S(\cdot - r) 
	(G_R(Y)(r) -G_R(\widetilde{Y})(r)) ~\txtd \textbf{W}_{r}, 
	G_R(Y) - G_R(\widetilde{Y}) \right\|_{\textbf{D}^{2\alpha}_{W}}\nonumber\\ 
	&\leq  \Big(C[|A|,F] R + C[|A|,G]~ C(R) (1+\|W\|_{\alpha})(\|W\|_{\alpha} +\|\mathbb{W}\|_{2\alpha}) \Big) \|Y-\widetilde{Y}, Y'-\widetilde{Y}'\|_{\textbf{D}^{2\alpha}_{W}}.
	\end{align}
	Setting $\widetilde{Y}\equiv 0$ and using that $F_R(0)=G_R(0)=0$,  we see from the previous deliberations that $(Y,Y')\mapsto (T_R(W,Y,Y'), G_R(Y))$ maps $D^{2\alpha}_{W}([0,1];\cX)$ into itself. Furthermore, by choosing $R$ small enough and regarding that $C(R)\to 0$ as $R\to 0$, we obtain that the mapping $(Y,Y')\mapsto (T_R(W,Y,Y'), G_R(Y))$ is a contraction. Therefore, Banach's fixed-point theorem proves the statement.
\end{proof}

\begin{remark}
 For precise details and computations regarding the existence of solutions of RDEs (without the truncation) see~\cite[Thm.~8.4]{FritzHairer}. In this case, taking $W\in C^{\beta}$ for $1/3<\alpha<\beta\leq 1/2$ and using that $\|W\|_{\alpha}\leq T^{\beta-\alpha}\|W\|_{\beta}$, one shows that the fixed-point mapping,
   maps a ball in $D^{2\alpha}_{W}$ with $\|Y,Y'\|_{\textbf{D}^{2\alpha}_{W}}\leq 1$ into itself and is a contraction for a small time step $T$. This argument entails the existence of a unique solution for an RDE on a small interval $[0,T_0]$, which can be extended by a standard concatenation procedure on the whole time interval $[0,1]$.  
\end{remark}
Going back to our setting, the next aim is to characterize the parameter $R$ required in order to
 decrease the Lipschitz constants of $\overline{F}$ and $\overline{G}$ using $\chi_R$. This fact is also required in Section~\ref{lpm}. As already seen we have to choose $R$ as small as possible.\\
Since in our deliberations, it is always required that $R\leq 1$ and $C(R)\to 0$ as $R\to 0$,
 the estimate provided in~\eqref{ef} reads as
 \begin{align}
 &	\left\|\int\limits_{0}^{\cdot} S(\cdot-r) (F_R(Y)(r) - F_R(\widetilde{Y})(r) )~\txtd r+ \int\limits_{0}^{\cdot} S(\cdot - r) 
 (G_R(Y)(r) -G_R(\widetilde{Y})(r )) ~\txtd \textbf{W}_{r}, 
 G_R(Y) - G_R(\widetilde{Y}) \right\|_{\textbf{D}^{2\alpha}_{W}}\nonumber\\ 
 &\leq  \Big(C[|A|,F]~ R + C[|A|,G]~  R (1+\|W\|_{\alpha})(\|W\|_{\alpha} +\|\mathbb{W}\|_{2\alpha}) \Big) \|Y-\widetilde{Y}, Y'-\widetilde{Y}'\|_{\textbf{D}^{2\alpha}_{W}}.\label{e:finall}
 \end{align}
As commonly met in the theory of random dynamical systems~\cite{LuSchmalfuss, GarridoLuSchmalfuss, BloemkerWang}, since all the estimates depend on the random input, it is meaningful to employ a cut-off technique for a {\em random variable}, i.e.~$R=R(W)$. Such an argument will also be used here as follows.\\

We fix $K>0$ and regarding~\eqref{e:finall}, we let 
$\widetilde{R}({W})$ be the unique solution of 
\begin{equation}\label{k}
C[|A|,F]~\widetilde{R}({W}) + C[|A|,G]~\widetilde{R}({W}) (1+\|{W}\|_{\alpha})(\|{W}\|_{\alpha} 
+ \|\mathbb{W}\|_{2\alpha}) =K
\end{equation}
and set 
 \begin{align}\label{r}
 R({W}):=\min\{\widetilde{R}({W}),1\}.
 \end{align}
 This means that if $R(W)=1$, we apply the cut-off procedure for $\|Y,Y'\|_{\textbf{D}^{2\alpha}_{W}}\leq 1/2$ or else (if $R(W)<1$) for $\|Y,Y'\|_{\textbf{D}^{2\alpha}_{W}}\leq R(W)/2$. \\
 In conclusion, we work in the next sections with a modified version of~\eqref{sde1} (equivalently~\eqref{equiv:rde}), where the drift and diffusion coefficients $\overline{F}$ and $\overline{G}$ are replaced by $F_{R(W)}$ and $G_{R(W)}$. For notational simplicity, the $W$-dependence of $R$ will be dropped whenever there is no confusion.\\

Due to~\eqref{k} we conclude.
\begin{lemma}  
	Let $(Y,Y'), (\widetilde{Y},\widetilde{Y}')\in D^{2\alpha}_{{W}}([0,1];\cX)$. We have
	\begin{equation}
	\label{wanttohave}
	\|T_{R}({W},Y,Y') - T_{R}({W},\widetilde{Y}, \widetilde{Y}'), (T_{R}({W},Y,Y') - T_{R}({W},\widetilde{Y}, \widetilde{Y}'))'
	\|_{\textbf{D}^{2\alpha}_{{W}}} \leq K \|  Y- \widetilde{Y}, 
	Y'-\widetilde{Y}'\|_{\textbf{D}^{2\alpha}_{{W}}}.
	\end{equation}
\end{lemma}

\begin{remark}
	We emphasize that one needs to control the derivatives of the diffusion coefficient $G$ in order to make the constants that depend on $G$ small after the cut-off procedure. Such restrictions are often met in the context of invariant manifolds for S(P)DEs with nonlinear multiplicative noise, see e.g.~\cite{FuBloemker,GarridoLuSchmalfuss}.
\end{remark}
\begin{remark}\label{fbm}
	If the random input is smoother, i.e.~$\alpha\in(1/2,1)$, it is enough to assume only $DG(0)=0$. In this case, the stochastic convolution used in~\eqref{integraleq} is defined as a Young integral, compare~\cite{Young} and~\cite[Sec.~8.2]{FritzHairer}. Here, it suffices to work only with paths $Y\in C^{\alpha}$ ($Y'$ and $R^Y$ are no longer required). This essentially simplifies the computations in Sections~\ref{rde} and~\ref{cut}. In the framework of Section~\ref{cut}, one has to modify only the $\alpha$-H\"older norm of a path $Y$ and to control the difference $\|G_R(Y)-G_R(\widetilde{Y})\|_{\alpha}$ for two paths $Y,\widetilde{Y}\in C^{\alpha}$, see~\cite[Sec.~8.3]{FritzHairer}. 
\end{remark}

\section{Random Dynamics}
\label{rd}

The main techniques and results established in the previous section using 
controlled rough paths are necessary in order to provide pathwise 
estimates for the solutions of~\eqref{sde1}. In this section, we provide some 
concepts from the random dynamical systems theory~\cite{Arnold}, which allow us
to give a definition of an invariant manifold for~\eqref{sde1}; for even further 
information regarding random dynamical systems generated by RDEs, 
see~\cite[Sec.~3]{BailleulRiedelScheutzow}.\medskip
	
The next concept is fundamental in the theory of random dynamical systems, 
since it describes a model of the driving noise.

\begin{definition} 
Let $(\Omega,\mathcal{F},\mathbb{P})$ stand for a probability space and 
$\theta:\mathbb{R}\times\Omega\rightarrow\Omega$ be a family of 
$\mathbb{P}$-preserving transformations (i.e.,~$\theta_{t}\mathbb{P}=
\mathbb{P}$ for $t\in\mathbb{R}$) having following properties:
\begin{description}
		\item[(i)] the mapping $(t,\omega)\mapsto\theta_{t}\omega$ is 
		$(\mathcal{B}(\mathbb{R})\otimes\mathcal{F},\mathcal{F})$-measurable, where 
		$\mathcal{B}(\cdot)$ denotes the Borel sigma-algebra;
		\item[(ii)] $\theta_{0}=\textnormal{Id}_{\Omega}$;
		\item[(iii)] $\theta_{t+s}=\theta_{t}\circ\theta_{s}$ for all 
		$t,s,\in\mathbb{R}$.
\end{description}
Then the quadrupel $(\Omega,\mathcal{F},\mathbb{P},(\theta_{t})_{t\in\mathbb{R}})$ 
is called a metric dynamical system.
\end{definition}

In our context, constructing a metric dynamical system is going to rely on
constructing $\theta$ as a shift map. We start to check that shifts act quite 
naturally on rough paths. For an $\alpha$-H\"older rough path $\textbf{W}=(W,\mathbb{W})$ 
and $\tau\in\mathbb{R}$ let us define the time-shift $\Theta_{\tau}\textbf{W}
:=(\Theta_{\tau}W,\Theta_{\tau}\mathbb{W})$ by
\begin{align*}
& \Theta_{\tau} W_t : =W_{t+\tau} - W_{\tau}\\
& \Theta_{\tau}\mathbb{W}_{s,t}:=\mathbb{W}_{s+\tau,t+\tau}
\end{align*}
Note that 
$\Theta_{\tau} W_{s,t}=W_{t+\tau} - W_{s+\tau}$. Furthermore, the shift leaves 
the path space invariant:

\begin{lemma}
\label{shift} 
Let $T_{1}, T_{2},\tau\in\mathbb{R}$, and $\bm{W}=(W,\mathbb{W})$ 
be an $\alpha$-H\"older rough path on $[T_{1},T_{2}]$ for $\alpha\in(1/3,1/2)$. 
Then the time-shift $\Theta_{\tau}\bm{W}= (\Theta_{\tau}W, 
\Theta_{\tau}\mathbb{W})$ is also an 
$\alpha$-H\"older rough path on $[T_{1}-\tau, T_{2}-\tau]$.
\end{lemma}

\begin{proof}
Let $s,u,t\in [T_{1}-\tau, T_{2}-\tau]$. The $\alpha$-H\"older-continuity 
of $\theta_{\tau}W$ and the $2\alpha$-H\"older continuity of $\theta_{\tau}\mathbb{W}$ 
are obvious. We only prove that Chen's relation~\eqref{chen} holds true. We have
\begin{align}
\Theta_{\tau} \mathbb{W}_{s,t} - \Theta_{\tau} \mathbb{W}_{s,u} 
-\Theta_{\tau}\mathbb{W}_{u,t} &= \mathbb{W}_{s+\tau,t+\tau} 
- \mathbb{W}_{s+\tau,u+\tau} - \mathbb{W}_{u+\tau,t+\tau} \nonumber \\
	& = W_{s+\tau,u+\tau} \otimes W_{u+\tau,t+\tau} \label{chena},\\
	&=(W_{u+\tau} -W_{\tau} - W_{s+\tau} +W_{\tau} ) \otimes (W_{t+\tau} 
	- W_{\tau} -W_{u+\tau} + W_{\tau} ) \nonumber \\
	& = \Theta_{\tau} W_{s,u} \otimes \Theta_{\tau} W_{u,t}. \nonumber
\end{align}
where in~\ref{chena} we use Chen's relation~\eqref{chen}.
\end{proof}

Based upon~\cite{BailleulRiedelScheutzow} we consider the following concept:

\begin{definition}
\label{rpc}
Let $(\Omega,\mathcal{F},\mathbb{P},(\theta_{t})_{t\in\mathbb{R}})$ be a 
metric dynamical system. We call $\textbf{W}=(W,\mathbb{W})$ a rough path 
cocycle if the identity
\begin{align*}
\textbf{W}_{s,s+t}(\omega)=\textbf{W}_{0,t}(\theta_{s}\omega)
\end{align*}
holds true for every $\omega\in\Omega$, $s\in\mathbb{R}$ and $t\geq 0$.
\end{definition}

The previous definitions hint already at the fact that one may be 
able to use as a probability space $\Omega$ a space of paths. We specify this construction for the case of a fractional Brownian motion, see~\cite{BailleulRiedelScheutzow} and~\cite[Sec.~6]{HesseNeamtu2}. Before we provide such a construction, we discuss the following {\em distance concept} between two $\alpha$-H\"older rough paths, see~\cite[Ch.~2]{FritzHairer}. To this aim let $J\subset\mathbb{R}$ be a compact interval and set
\begin{align*}
\Delta_{J}:=\{(s,t)\in J\times J : t\geq s\}.
\end{align*}
\begin{definition}
	Let $\textbf{W}^1=(W^1,\mathbb{W}^1)$ and $\textbf{W}^2=(W^2,\mathbb{W}^2)$ be two $\alpha$-H\"older rough paths. We introduce the $\alpha$-H\"older rough path (inhomogenuous) metric
	\begin{align}\label{rpmetric}
	d_{a,J}(\textbf{W}^1,\textbf{W}^2) = \sup\limits_{(s,t)\in\Delta_{J}} \frac{|W^1_t-W^1_s-W^2_t+W^2_s|}{|t-s|^{\alpha}} + \sup\limits_{(s,t)\in\Delta_{J}} \frac{|\mathbb{W}^1_{st}- \mathbb{W}^2_{st}| }{|t-s|^{2\alpha}}.
	\end{align}
	
\end{definition}
We stress that we only work with $\alpha$-H\"older processes $W$ such that $W(0)=0$ and therefore~\eqref{rpmetric} is a metric.
\begin{example}
\label{eq:Bmain}
As a concrete example for $W$ we consider fractional Brownian motion 
$B^{H}$ restricted to any compact interval $[-L,L]$ with $L\geq 1$ and for $H\in(1/3,1/2]$. 
This includes Brownian motion for $H=1/2$. In order to define a metric dynamical system, we introduce the canonical probability
space $(C_{0}(\mathbb{R};\mathbb{R}^{d}),\mathcal{B}(C_{0}(\mathbb{R};
\mathbb{R}^{d})),\mathbb{P} )$, where $C_{0}(\mathbb{R};\mathbb{R}^{d})$ 
denotes the space of all $\mathbb{R}^{d}$-valued continuous functions, which are $0$ 
in $0$, endowed with the compact open topology and $\mathbb{P}$ is the fractional Gau\ss{} measure. The shift on this 
space is given by
\begin{align}
(\Theta_{\tau} f) (\cdot) := f(\tau+\cdot) - f(\tau),  
~ \tau\in\mathbb{R}, \quad f\in C_{0}(\R,\mathbb{R}^{d}).
\end{align}
Furthermore, we restrict it to the set $\Omega:=C^{\alpha'}_{0}(\mathbb{R},\mathbb{R}^d)$, of all $\alpha'$-H\"older continuous paths on any compact interval $J\subset\mathbb{R}$, where $\frac{1}{3}<\alpha<\alpha'< H <\frac{1}{2}$. Note that we consider $\alpha'$-H\"older functions in order to lift $(W,\mathbb{W})$ to an $\alpha$-H\"older rough path.  We further equip this set with the trace $\sigma$-algebra $\mathcal{F}:=\Omega\cap \mathcal{B}(C_{0}(\mathbb{R};\mathbb{R}^d))$ and take the restriction of $\mathbb{P}$ as well. Then the set $\Omega\subset C_{0}(\mathbb{R};\mathbb{R}^d)$ has full measure and is $\Theta$-invariant. The new quadrupel $(\Omega,\mathcal{F},\mathbb{P},\Theta)$ builds again a metric dynamical system, which will be further restricted as follows. Namely, one can show that there exists a $\Theta$-invariant subset $\Omega_{B}\subset \Omega$ of full measure such that for any $W\in\Omega_{B}$ and any compact interval $J\subset \mathbb{R}$ there exists a L\'evy-area $\mathbb{W}\in C^{2\alpha}(\Delta_{J};\mathbb{R}^d\otimes\mathbb{R}^d)$ such that $\textbf{W}=(W,\mathbb{W})$ is an $\alpha$-H\"older rough path. This can further be approximated by a sequence $W^n:=((W^n,\mathbb{W}^n))_{n\in\mathbb{N}}$ in the $d_{\alpha,J}$ metric. Here $(W^{n})_{n\in\mathbb{N}}$ are piecewise dyadic linear functions. In this case, the iterated integral $\mathbb{W}^n$ exists in the classical sense, i.e.
\begin{align}\label{approx}
\mathbb{W}^n_{s,t}=\int\limits_{s}^{t} (W^n_r- W^n_s)~ \txtd W^{n}_r,~~\mbox{ for }\in\Delta_J.
\end{align} 
Assuming that $J\subseteq [-L,L]$ we have according to~\cite[Thm.~2]{CoutinQian} (see also~\cite[Cor.~10]{FritzHairer}) that the iterated integral
\begin{align*}
\mathbb{W}_{s,t}:=\int\limits_{s}^{t} (W_r - W_s)~\txtd W_r,~~\mbox{ for } -L\leq s\leq t \leq L.
\end{align*}
exists almost surely. Moreover, one also has $\mathbb{W}^n\to \mathbb{W}$ in $C^{2\alpha}(\Delta_{[-L,L]};\mathbb{R}\otimes\mathbb{R}^d)$ almost surely. Since $J\subseteq[-L,L]$ one immediately obtains that $\textbf{W}^n$ converges to $\textbf{W}$ with respect to the $d_{\alpha,J}$ metric. In conclusion, one obtains the existence of the L\'evy area $\mathbb{W}$ which can be approximated by~\eqref{approx} on a compact interval $[-L,L]$. Glueing these lifts together, one obtains a process $\mathbb{W}$ on the whole real line.
Since $\textbf{W}^n$ converges to $\textbf{W}$ with respect to the $d_{\alpha,J}$-metric one can conclude that $\Omega_B$ has full measure and is $\Theta$-invariant.
Consequently, taking $(\Omega_{B},\mathcal{F}_{B}, \mathbb{P}_{B},\Theta)$ as introduced above, one has that the fractional Brownian motion
$\textbf{B}^{H}=(B^{H},\mathbb{B}^{H})$ represents a rough 
path cocycle as introduced in Definition~\ref{rpc}. 
\end{example}
We observe that piecewise linear approximations are used to construct the L\'evy-area~\eqref{omega2}. This is a fundamental property of {\em geometric} $\alpha$-H\"older rough paths, see~\cite[Prop.~2.5]{FritzHairer}.
Of course, the same construction of an appropriate path-space $(\Omega_W,\cF_W,\mathbb{P})$,
can be carried out for more general geometric $\alpha$-H\"older rough paths $\textbf{W}=
(W,\mathbb{W})$ constructed from a Gaussian stochastic process with stationary increments $W_t$, not just fractional 
Brownian motion, where the definition of a shift map is still as above, i.e., 
\benn
(\Theta_{\tau} W )(t) : =W_{t+\tau} - W_{\tau}.
\eenn
For further details on such statements, see~\cite[Cor.~9]{BailleulRiedelScheutzow}.
We have the abstract definition of, as well as concrete examples for,
metric dynamical systems for our problem modelling the underlying rough driving
process (or ``noise'').

Now we have to also define the dynamical system structure 
of the solution operator of our RDE. As a first step we recall the 
classical definition of a random dynamical system~\cite{Arnold}.

\begin{definition}
\label{rds} 
A random dynamical system on $\cX$ over a metric dynamical 
system $(\Omega,\mathcal{F},\mathbb{P},(\theta_{t})_{t\in\mathbb{R}})$ 
is a mapping $$\varphi:[0,\I)\times\Omega\times \cX\to \cX,
\mbox{  } (t,\omega,x)\mapsto \varphi(t,\omega,x), $$
which is $(\mathcal{B}([0,\I))\times\mathcal{F}\times
\mathcal{B}(\cX),\mathcal{B}(\cX))$-measurable and satisfies:
	\begin{description}
		\item[(i)] $\varphi(0,\omega,\cdot{})=\textnormal{Id}_{\cX}$ 
		for all $\omega\in\Omega$;
		\item[(ii)]$ \varphi(t+\tau,\omega,x)=
		\varphi(t,\theta_{\tau}\omega,\varphi(\tau,\omega,x)), 
		\mbox{ for all } x\in \cX, ~t,\tau\in[0,\I),~\omega\in\Omega;$
		\item[(iii)] $\varphi(t,\omega,\cdot{}):\cX\to \cX$ is 
		continuous for all $t\in[0,\I)$ and all $\omega\in\Omega$.
	\end{description}
\end{definition}

The second property in Definition~\ref{rds} is referred to as the 
cocycle property. One can now expect that 
the solution operator of~\eqref{sde1} generates a random dynamical 
system. Indeed, working with a pathwise interpretation of the 
stochastic integral as given in~\eqref{Gintegral}, no exceptional 
sets can occur. For completeness, we indicate a proof of this
fact, see also~\cite{BailleulRiedelScheutzow}.
 
\begin{lemma}
\label{cocycle} 
Let $\textbf{W}$ be a rough path cocycle. 
Then the solution operator
\begin{align*}
t\mapsto \varphi(t,W,\xi)=U_{t}= S(t)\xi + \int\limits_{0}^{t}S(t-r)F(U_{r})~\txtd r 
+ \int\limits_{0}^{t} S(t-r)G(U_{r})~\txtd \bm{{W}}_{r}, 
\end{align*}
for any $t\in[0,\I)$ of the RDE~\eqref{sde1} generates a random dynamical system
over the metric dynamical system $(\Omega_W,\cF_W,\mathbb{P},(\Theta_t)_{t\in\R})$.
\end{lemma}

\begin{proof}
The relevant properties to define the metric dynamical system we need
have been discussed in Example~\eqref{eq:Bmain}. The only difficulty
is checking cocycle property for the solution operator. We calculate
\begin{align*}
U_{t+\tau} &=S(t+\tau)\xi + \int\limits_{0}^{t+\tau} S(t+\tau-r) 
F(U_{r}) ~\txtd r + \int\limits_{0}^{t+\tau} S(t+\tau-r) 
G(U_{r}) ~\txtd \bm{{W}}_{r}\\
	& = S(t) S(\tau)\xi + \int\limits_{0}^{\tau}  S(t+\tau- r) 
	F(U_{r}) ~\txtd r + \int\limits_{\tau}^{t+\tau} S(t+\tau-r) 
	F(U_{r}) ~\txtd r\\
	& + \int\limits_{0}^{\tau}S(t+\tau-r) G(U_{r}) ~\txtd 
	\bm{{W}}_{r} + \int\limits_{\tau}^{t+\tau} S(t+\tau-r) G(U_{r}) 
	~\txtd \bm{{W}}_{r}\\
	&=S(t) \left( S(\tau)\xi + \int\limits_{0}^{\tau} S(\tau-r) 
	F(U_{r}) ~\txtd r + \int\limits_{0}^{\tau} S(\tau-r)G(U_{r}) ~\txtd 
	\bm{{W}}_{r}  \right)\\
	& + \int\limits_{0}^{t}  S(t-r) F(U_{r+\tau}) ~\txtd r + 
	\int\limits_{0}^{t} S(t-r) G(U_{r+\tau}) ~\txtd \Theta_{\tau}
	\bm{{W}}_{r}\\
	& = S(t) U_{\tau} + \int\limits_{0}^{t} S(t-r) F(U_{r+\tau}) 
	~\txtd r + \int\limits_{0}^{t}S(t-r)G(U_{r+\tau}) ~\txtd 
	\Theta_{\tau}\bm{{W}}_{r}.
\end{align*}
The above computation are rigorously justified, since one can 
check that if $(U,U')\in D^{2\alpha}_{{W}}([T_{1}+\tau,T_{2}+\tau];
\cX)$ then $(U_{\cdot +\tau}, U'_{\cdot+\tau} ) \in 
D^{2\alpha}_{\Theta_{\tau}W}([T_{1},T_{2}];\cX)$. 
Here $T_{1}, T_{2}\in\mathbb{R}$ with $T_{1}<T_{2}$. 
The $\alpha$-H\"older continuity of $U_{\cdot+\tau}$ and 
$U'_{\cdot+\tau}$ is obvious. For the remainder we have
\begin{align*}
	|R^{U_{\cdot+\tau}}_{s,t} - U'_{s+\tau} \Theta_{\tau}{W}_{s,t}| 
	= |U_{s+\tau,t+\tau} - U'_{s+\tau} {W}_{s+\tau,t+\tau} | 
	= |R^{U}_{s+\tau,t+\tau}|\leq \|R^{U}\|_{2\alpha}(t-s)^{2\alpha}.
\end{align*}
Furthermore, to show the shift property of the rough integral, we take 
a partition $\mathcal{P}$ of $[T_{1},T_{2}]$ and have
\begin{align}\label{shiftintegral}
	\int\limits_{T_{1}}^{T_{2}} U_{r+\tau}~\txtd\Theta_{\tau} \bm{{W}}_{r} 
	&= \lim\limits_{|\mathcal{P}|\to 0 }	\sum\limits_{[s,t]\in\mathcal{P}} 
	(U_{s+\tau} \Theta_{\tau} {W}_{s,t} + U'_{s+\tau}\Theta_{\tau}
	\mathbb{W}_{s,t} )\nonumber\\ 
	&= \lim\limits_{|\mathcal{P}|\to 0 }	
	\sum\limits_{[s,t]\in\mathcal{P}} U_{s+\tau} {W}_{s+\tau,t+\tau} 
	+ U'_{s+\tau} \mathbb{W}_{s+\tau,t+\tau}\nonumber\\
	& = \lim\limits_{|\mathcal{\widetilde{P}}|\to 0 } 
	\sum\limits_{[\widetilde{s},\widetilde{t}]\in\mathcal{\widetilde{P}}}  
	U_{\widetilde{s}} {W}_{\widetilde{s},\widetilde{t}} + U'_{\widetilde{s}} 
	\mathbb{W}_{\widetilde{s}, \widetilde{t}} = \int\limits_{
	T_{1}+\tau}^{T_{2}+\tau} U_{r} ~\txtd \bm{{W}}_{r}.
\end{align}
Here $\mathcal{\widetilde{P}}$ is a partition of $[T_{1}+
\tau,T_{2}+\tau]$ given by $\mathcal{\widetilde{P}} : = 
\{ [s+\tau,t+\tau] \mbox{ : }[s,t]\in\mathcal{P}\}$.\\

The $(\mathcal{B}([0,\infty))\times\mathcal{F}_{W}\times\mathcal{B}(\cX), \mathcal{B}(\cX))$-measurability of $\varphi$ follows be well-known arguments. One considers a sequence of (classical) solutions $(U^{n},(U^{n})')_{n\in\mathbb{N}}$ of~\eqref{sde1} corresponding to smooth approximations $(W^{n},\mathbb{W}^{n})_{n\in\mathbb{N}}$ of $(W,\mathbb{W})$. Obviously, the mapping $(t,W,\xi)\mapsto U^{n}_{t}$ is $(\mathcal{B}([0,T])\times\mathcal{F}_{W}\times\mathcal{B}(\cX), \mathcal{B}(\cX))$-measurable for any $T>0$.
Since $U$ continuously depends on the rough input $W$, according to Thm.~8.5~in~\cite{FritzHairer}, one immediately concludes that $\lim\limits_{n\to\infty}U^{n}_t=U_{t}$.
This gives the measurability of $U$ with respect to $\mathcal{F}_{W}\times \mathcal{B}(\cX)$. Due to the time-continuity of $U$, we obtain by Ch.~3~in~\cite{CastaingValadier} the $(\mathcal{B}([0,T])\times\mathcal{F}_{W}\times\mathcal{B}(\cX), \mathcal{B}(\cX))$-measurability of the mapping $(t,\omega,\xi)\mapsto U_{t}$ for any $T>0$.
\end{proof}
	
Note that the role of the random elements in $\Omega_W$ is played by
the paths $W$, recall Example~\ref{eq:Bmain} for the construction of such a space for the case of the fractional Brownian motion. As a convention, we directly denote these elements by $W\in \Omega_{W}$ (as in Example~\ref{eq:Bmain}) and do not employ the 
identification $W_t(\omega):=\omega(t)$. 
The random dynamical system $\varphi:\mathbb{R}^{+}\times \Omega\times \R^n 
\to \R^n$ obviously
depends upon the $t,\xi,W$, and $\mathbb{W}$ although we do not directly
display the dependence upon $\mathbb{W}$ in the notation.\medskip

To construct local random invariant manifolds, which can be characterized 
by the graph of a smooth function in a ball with a certain 
radius~\cite{CaraballoDuanLuSchmalfuss,DuanLuSchmalfuss,GarridoLuSchmalfuss} 
one requires the concept of tempered random variables~\cite[Chapter 4]{Arnold},
which we recall next:

\begin{definition}\label{def:tempered}
A random variable $\widetilde{R}:\Omega\to (0,\infty)$ is called tempered from above, with 
respect to a metric dynamical system $(\Omega,\mathcal{F},\mathbb{P},
(\theta_{t})_{t\in\mathbb{R}})$, if
\begin{equation}
\label{tempered}
\limsup\limits_{t\to\pm\infty}\frac{\ln^{+} \widetilde{R} (\theta_{t}\omega)}{t}=0, 
\quad \mbox{ for all } \omega\in\Omega,
\end{equation}	
where $\ln^+a:=\max\left\{\ln a,0\right\}$.
A random variable is called tempered from below if $1/\widetilde{R}$ is tempered from above.
A random variable is tempered if and only if is tempered from above and from below.
\end{definition}

We emphasize that temperedness is equivalent to {\em subexponential growth}. This concept is crucial for computations, since one has to control the growth of a random variable along the orbits $(\theta_t)$.
Note that the set of all $\omega\in\Omega$ satisfying~\eqref{tempered} 
is invariant with respect to any shift map $(\theta_{t})_{t\in\mathbb{R}}$,
which is an observation applicable to our case when $\theta_t=\Theta_t$. A 
sufficient condition for temperedness from above is according to~\cite[Prop.~4.1.3]{Arnold} 
that
\begin{equation}\label{c:temp}
\mathbb{E} \sup\limits_{t\in[0,1]}  \widetilde{R}(\theta_{t}\omega)<\infty.
\end{equation}
Moreover, if the random variable $\widetilde{R}$ is tempered from below with 
$t\mapsto \widetilde{R}(\theta_{t}\omega)$ continuous for all $\omega\in\Omega$, 
then for every $\varepsilon>0$ there exists a constant $C[\varepsilon,\omega]>0$ 
such that
\begin{equation}
\label{temperedmanradius}
\widetilde{R}(\theta_{t}\omega) \geq C[\varepsilon,\omega] \txte^{-\varepsilon|t|},
\end{equation}
for any $\omega\in\Omega$. Again, for our concrete example $\Omega=\Omega_B$
we have.

\begin{lemma}
\label{ltempered} 
Let $\textbf{B}^H=(B^H,\mathbb{B}^H)$ be the rough path cocycle associated to a fractional 
Brownian motion $B^H$ with Hurst parameter $H\in(1/3,1/2]$. Then the random variables 
\benn
R_1(B^H)=\|B^H\|_{\alpha}\quad \text{ and }\quad R_2(\mathbb{B}^H)=\|\mathbb{B}^H \|_{2\alpha}  
\eenn
are tempered from above.
\end{lemma}

\begin{proof}
The first assertion is valid due to the fact that 
$\mathbb{E}\|B^{H}\|^{m}_{\alpha}<\infty$ and the second one follows regarding that
$\mathbb{E}\|\mathbb{B}^{H}\|^{m}_{2\alpha}<\infty$, for $m\in\mathbb{N}$ 
as contained in~\cite[Thm.~10.4]{FritzHairer}. This implies that $\mathbb{E}\sup\limits_{t\in[0,1]}\|\Theta_tB^{H}\|_{\alpha}<\infty$ and $\mathbb{E}\sup\limits_{t\in[0,1]}\|\Theta_t \mathbb{B}^{H}\|_{2\alpha}<\infty$. Consequently, condition~\eqref{c:temp} is valid and shows the temperedness from above
of both random variables.
\end{proof}

From now, we shall simply assume that 
$\textbf{W}=(W,\mathbb{W})$ is a  rough path cocycle such that the
random variables
\benn
R_1(W)=\|W\|_{\alpha}\quad \text{ and }\quad R_2(\mathbb{W})=\|\mathbb{W} \|_{2\alpha}  
\eenn
are tempered from above. This will be necessary in the existence proof 
of a local center manifold. One wants to ensure~\cite{LuSchmalfuss,LianLu,
CaraballoDuanLuSchmalfuss} that for initial conditions 
belonging to a ball with a sufficiently small tempered from below radius, 
the corresponding trajectories remain within such a ball (see the proof of 
Lemma~\ref{localcman} below). To this aim we emphasize.
\begin{lemma}
	The random variable $R(W)$ in~\eqref{k} is tempered from below.
\end{lemma}
\begin{proof}
	This follows by Lemma~\ref{ltempered} regarding the structure given in~\eqref{k}. The random variables $\|W\|_{\alpha}$ and $\|\mathbb{W}\|_{2\alpha}$ are tempered from above and therefore $(1+\|W\|_{\alpha})(\|W\|_{\alpha} + \|\mathbb{W}\|_{2\alpha})$ is again tempered from above. More generally, a polynomial containing $\|W\|_{\alpha}$ and $\|\mathbb{W}\|_{2\alpha}$ is tempered from above. According to Def~\ref{def:tempered} the inverse of a tempered from above random variable is tempered from below. In conclusion, a random variable having the structure $\frac{K}{C_1 + C_2 \widetilde{C}(W)}$, where $K,C_1,C_2$ are positive constants and the random variable $\widetilde{C}$ is tempered from above, is tempered from below.
\end{proof}

\section{Local center manifolds for RDEs}
\label{lpm}

In this section we prove the existence of a local center manifold for~\eqref{sde1}. 
The approach is similar to the one employed in~\cite{GarridoLuSchmalfuss} in order 
to compute unstable manifolds for SPDEs driven by a fractional Brownian motion with 
Hurst parameter $H>1/2$ using elements from fractional calculus~\cite{MaslowskiNualart}. 
However, here we want to connect the theory of random invariant manifolds for S(P)DEs 
as in~\cite{DuanLuSchmalfuss, LianLu,  GarridoLuSchmalfuss} to rough paths theory. This 
allows us to consider SDEs driven by general $\alpha$-H\"older continuous processes, 
as described in Section~\ref{rd}.

\begin{assumptions}
\label{ass:linearpart} We assume that we are in a center-stable situation, namely there 
are eigenvalues $\{\lambda^\txtc_j\}_{j=1}^{n_\txtc}$ of the linear operator $A$ on 
the imaginary axis $\txti\R$ as well as eigenvalues $\{\lambda^\txts_j\}_{j=1}^{n_\txts}$ 
in the left-half plane $\{z\in\C:\textnormal{Re}(z)<0\}$. Upon counting multiplicies we
have $n_\txtc+n_\txts=n$. Hence, there 
exists a decomposition of the phase space $\R^n=\cX=\cX^{\txtc}\oplus \cX^{\txts}$, where 
the linear spaces $\cX^{\txtc}$ and $\cX^{\txts}$ are spanned by the (generalized) 
eigenvectors with eigenvalues $\lambda^\txtc_j$ and $\lambda^\txts_j$ respectively. 
We denote the restrictions of $A$ on $\cX^{\txtc}$ and $\cX^{\txts}$ by $A_{\txtc}:=A|_{\cX^{\txtc}}$ and $A_{\txts}:=A|_{\cX^{\txts}}$. Then, $S^{\txtc}(t):=e^{tA_{\txtc}}$ and $S^{\txts}(t):=e^{tA_{\txts}}$ are groups of linear operators on $\cX^{\txtc}$ respectively $\cX^{\txts}$. Moreover, there exist two bounded projections $P^{\txtc}$ and $P^{\txts}$ associated 
to this splitting such that
	\begin{itemize} 
	\item [1)] $\Id= P^{\txts}+ P^{\txtc}$;
	\item [2)] $P^{\txtc}S(t)=S(t)P^{\txtc}$ and $P^{\txts}S(t)=S(t)P^{\txts}$ 
	for $t\geq 0$.
\end{itemize}
Additionally, we assume that there exist two exponents $\gamma$ and $\beta$ 
with $-\beta<0\leq \gamma<\beta$ and constants $M_{c},M_{s}\geq 1$, such that
	\begin{align}
	&|S^{\txtc}(t)  x| \leq M_{c} \txte^{\gamma t} |x|, 
	~~\mbox{  for } t\leq 0 \mbox{ and } x\in \cX;\label{gamma}\\
	& |S^{\txts}(t)x| \leq M_{s} \txte^{-\beta t} |x|, 
	~\mbox{for } t\geq 0 \mbox{ and } x\in \cX.\label{beta}
	\end{align}
\end{assumptions}

For further details and similar assumptions, see~\cite[Sec.~6.1.1]{DuanWang} 
and~\cite[Sec.~7.1.2]{SellYou}. According to our restrictions we have 
$\gamma\geq 0$ and $-\beta<0$ which gives us the spectral gap $\gamma+\beta>0$.
We also use the notation $\xi^{\txtc}:=P^{\txtc}\xi$ and refer to $\cX^{\txtc}$ 
and $\cX^{\txts}$ as center, respectively stable, subspace.\medskip

\begin{remark}
One can easily extend the techniques presented below if one additionally has an 
unstable subspace, namely if there exist eigenvalues of $A$ with 
real part greater than zero. In this case the classical exponential 
trichotomy condition is satisfied, see for instance~\cite{ChenRobertsDuan, SellYou} 
and Appendix~\ref{b}. For simplicity and from the point of view of applications 
we assume that we are in a center-stable situation and work with 
Assumptions~\ref{ass:linearpart} similar to~\cite[Sec.~6]{DuanWang}.
\end{remark}
 
\begin{definition}
\label{def:cm} We call a random set $\cM^{\txtc}({W})$, which is invariant 
with respect to $\varphi$ (i.e. $\varphi(t, {W},\cM^{\txtc}({W}))\subset 
\cM^{\txtc}(\Theta_{t}{W})$ for $t\in\mathbb{R}$ and ${W}\in\Omega_{W}$), a 
center manifold if this can be represented as 
\begin{align}\label{graph}
 \cM^{\txtc}({W})=\{\xi + h^{\txtc}(\xi,{W})\mbox{ : }\xi\in \cX^{\txtc} \},
\end{align}
where $h^{\txtc}(\cdot,{W}):\cX^{\txtc}\to \cX^{\txts}$ is Lipschitz continuous and differentiable in zero. Moreover, 
$h^{\txtc}(0,{W})=0$ and $\cM^{\txtc}({W})$ is tangent to $\cX^{\txtc}$ at 
the origin, meaning that the tangency condition $\txtD h^{\txtc}(0,{W})=0$ 
is satisfied.
\end{definition}

We prove the existence of a local center manifold 
$\cM^{\txtc}_{\textnormal{loc}}({W})$ for~\eqref{sde1}, namely~\eqref{graph} 
holds true when $\xi$ belongs to a random ball of $\cX^{\txtc}$ with a 
tempered radius. The Lipschitz continuity of $h^{\txtc}$ with respect to 
$\xi$ will also be justified.

\begin{remark} 
For a better comprehension Appendix~\ref{a} summarizes basic methods used to 
establish the invariance of random manifolds. In the theory of random dynamical 
systems the suitable concept for invariance of a random set 
(see~\cite{Arnold,DuanLuSchmalfuss}) is that each orbit starting inside this 
random set, evolves and remains there omega-wise modulo the changes that occur 
due to the noise. These changes can be characterized by a suitable shift of the 
fiber of the noise, as argued in the proof of Lemma~\ref{localcman}.
\end{remark}

One of the proof technique for the existence of (local) center manifolds for 
deterministic and stochastic ODEs/PDEs is based on the Lyapunov-Perron method. 
We employ here the Lyapunov-Perron method in conjunction with rough path estimates. 
Note that the continuous-time Lyapunov-Perron map for~\eqref{sde1} is constituted 
in this case by (compare~\cite[Sec.~6.2]{DuanWang} or~\cite{WanngDuan})
\begin{align}
\label{lp}
J({W}, U,\xi)[\tau] & := S^{\txtc}(\tau) \xi^{\txtc} + \int\limits_{0}^{\tau} 
S^{\txtc} (\tau-r)P^{\txtc }F(U_{r}) ~\txtd r + \int\limits_{0}^{\tau} S^{\txtc} 
(\tau-r)P^{\txtc } G(U_{r}) ~\txtd \textbf{W}_r\\
& +\int\limits_{-\infty}^{\tau} S^{\txts} (\tau-r)P^{\txts } F(U_{r}) ~\txtd r 
+ \int\limits_{-\infty}^{\tau} S^{\txts} (\tau-r) P^{\txts }G (U_{r}) 
~\txtd \textbf{W}_{r}, \nonumber
\end{align}
for $\tau\leq 0$. For more details on the Lyapunov-Perron transform in the context of center manifolds, see Section~\ref{sectfp} and the references specified therein.
Due to the presence of the rough stochastic integrals we cannot directly work 
with~\eqref{lp}, since we have to keep track of $\|W\|_{\alpha}$ and 
$\|\mathbb{W}\|_{2\alpha}$ occurring in~\eqref{gdiffusion} on a finite-time 
horizon. Similar to~\cite{GarridoLuSchmalfuss} we derive an appropriate 
discretized version of the Lyapunov-Perron map and show that this possesses 
a fixed-point in a suitable function space. We 
provide the explicit derivation of the discrete Lyapunov-Perron map 
in Subsection~\ref{dlp}. \\
For our aims we also have to considered the modified version of~\eqref{lp} and indicate this as follows, recalling the notations introduced in Section~\ref{cut}.
	\begin{align*}
	J_R({W}, U,\xi)[\tau] & := S^{\txtc}(\tau) \xi^{\txtc} + \int\limits_{0}^{\tau} 
	S^{\txtc} (\tau-r)P^{\txtc }\overline{F}(\chi_R(U))(r) ~\txtd r + \int\limits_{0}^{\tau} S^{\txtc} 
	(\tau-r)P^{\txtc }\overline{G}(\chi_{R}(U))(r) ~\txtd \textbf{W}_r\\
	& +\int\limits_{-\infty}^{\tau} S^{\txts} (\tau-r) P^{\txts }\overline{F}(\chi_R(U))(r) ~\txtd r 
	+ \int\limits_{-\infty}^{\tau} S^{\txts} (\tau-r)P^{\txts } \overline{G}(\chi_R(U))(r) 
	~\txtd \textbf{W}_{r}\\
	&=  S^{\txtc}(\tau) \xi^{\txtc} + \int\limits_{0}^{\tau} 
	S^{\txtc} (\tau-r)P^{\txtc }F_R(U)(r) ~\txtd r + \int\limits_{0}^{\tau} S^{\txtc} 
	(\tau-r)P^{\txtc }G_R(U)(r) ~\txtd \textbf{W}_r\\
	& +\int\limits_{-\infty}^{\tau} S^{\txts} (\tau-r)P^{\txts } F_R(U)(r) ~\txtd r 
	+ \int\limits_{-\infty}^{\tau} S^{\txts} (\tau-r)P^{\txts } G_R(U)(r) 
	~\txtd \textbf{W}_{r}.
	\end{align*}
In the following sequel we will consider the solution of~\eqref{sde1} at 
discrete times and obtain a sequence of mild solutions. The local center 
manifold theory will be developed for the discrete-time random dynamical 
system and will be shown to hold true for the original continuous-time one, as 
in~\cite{GarridoLuSchmalfuss, LianLu}. 

\subsection{Derivation of a discrete Lyapunov-Perron transform}
\label{dlp}
 
The strategy is to rewrite~\eqref{lp} such that we only have to deal with 
stochastic integrals on the time-interval $[0,1]$, as considered in 
Section~\ref{rde}. To this aim we let $W\in\Omega_{W}$, $t\in[0,1]$ and 
$i\in\mathbb{Z}^{-}$. Replacing $\tau$ by $t+i-1$ in~\eqref{lp} we obtain
\begin{align}
&J({W}, U,\xi)[t+i-1]=S^{\txtc}(t+i-1) \xi^{\txtc} \nonumber \\
& + \int\limits_{0}^{t+i-1} S^{\txtc}(t+i-1-r) P^{\txtc }F(U_{r}) ~\txtd r 
+ \int\limits_{0}^{t+i-1} S^{\txtc}(t+i-1-r)P^{\txtc } G(U_{r}) 
~\txtd \bm{{W}}_{r}\label{one}\\
&+ \int\limits_{-\infty}^{t+i-1} S^{\txts} (t+i-1-r)P^{\txts } F(U_{r}) ~\txtd r 
+\int\limits_{-\infty}^{t+i-1} S^{\txts} (t+i -1-r) P^{\txts }G(U_{r}) 
~\txtd \bm{{W}}_{r} \label{two}\\
& = S^{\txtc}(t+i-1)\xi^{\txtc} + \int\limits_{0}^{i} S^{\txtc} (t+i-1-r) P^{\txtc }
F(U_{r}) ~\txtd r + \int\limits_{0}^{i} S^{\txtc} (t+i-1-r) P^{\txtc }
G(U_{r})~\txtd \bm{{W}}_{r} \nonumber\\
&+ \int\limits_{i}^{i-1+t} S^{\txtc} (t+i-1-r) P^{\txtc }F(U_{r}) ~\txtd r 
+ \int\limits_{i}^{i-1+t} S^{\txtc}(t+i-1-r) P^{\txtc }G(U_{r})~\txtd \bm{{W}}_{r} 
\nonumber\\
& + \int\limits_{-\infty}^{i-1} S^{\txts} (t+i-1-r)P^{\txts } F(U_{r})~\txtd r 
+ \int\limits_{-\infty}^{i-1} S^{\txts}(t+i-1-r)P^{\txts } G(U_{r}) 
~\txtd \bm{{W}}_{r}\nonumber\\
& + \int\limits_{i-1}^{t+i-1} S^{\txts} (t+i-1-r)P^{\txts } F(U_{r})~\txtd r 
+ \int\limits_{i-1}^{t+i-1} S^{\txts} (t+i-1-r) P^{\txts }G(U_{r}) 
~\txtd \bm{{W}}_{r}\nonumber\\
& = S^{\txtc}(t+i-1)\xi^{\txtc} + \sum\limits_{k=0}^{i+1} 
\int\limits_{k}^{k-1}S^{\txtc} (t+i-1-r) P^{\txtc }F(U_{r})~\txtd r 
+ \sum\limits_{k=0}^{i+1} \int\limits_{k}^{k-1}S^{\txtc} 
(t+i-1-r)P^{\txtc } F(U_{r})~\txtd \bm{{W}}_{r} \label{subst1} \\
& + \int\limits_{i}^{i-1+t} S^{\txtc} (t+i-1-r)P^{\txtc } F(U_{r}) 
~\txtd r + \int\limits_{i}^{i-1+t}S^{\txtc} (t+i-1-r) P^{\txtc }
G(U_{r})~\txtd \bm{{W}}_{r} \label{subst2} \\
& + \sum\limits_{k=-\infty}^{i-1} \int\limits_{k-1}^{k} 
S^{\txts} (t+i-1-r)P^{\txts } F(U_{r}) ~\txtd r + \sum\limits_{k=-\infty}^{i-1} 
\int\limits_{k-1}^{k} S^{\txts} (t+i-1-r)P^{\txts } G(U_{r}) ~\txtd \bm{{W}}_{r} 
\label{subst3}\\
&  + \int\limits_{i-1}^{t+i-1} S^{\txts} (t+i-1-r)P^{\txts } F(U_{r})~\txtd r 
+ \int\limits_{i-1}^{t+i-1} S^{\txts} (t+i-1-r) P^{\txts }G(U_{r}) 
~\txtd \bm{{W}}_{r}\label{subst4}.
\end{align}
In order to simplify the expressions above we perform the following substitutions 
and use~\eqref{shiftintegral}. More precisely, replacing $r$ by $r-k+1$, the sum 
in~\eqref{subst1} yields
\begin{align*}
&-\sum\limits_{k=0}^{i+1} S^{\txtc}(t+i-1-k) \int\limits_{0}^{1} 
S^{\txtc}(1-r)P^{\txtc } F(U_{r+k-1})~\txtd r \\
& - \sum\limits_{k=0}^{i+1} S^{\txtc} (t+i-1-k) P^{\txtc } \int\limits_{0}^{1} 
S^{\txtc}(1-r) P^{\txtc}G(U_{r+k-1})~\txtd\Theta_{k-1} \bm{{W}}_{r}.
\end{align*} 
Substituting $r$ with $r-i+1$, we may re-write~\eqref{subst2} as
\begin{align*}
 -\int\limits_{t}^{1} S^{\txts} (t-r)P^{\txts } F(U_{r+i-1}) ~\txtd r 
- \int\limits_{t}^{1} S^{\txts}(t-r) P^{\txts }G(U_{r+i-1})~\txtd\Theta_{i-1} 
\bm{{W}}_{r}.
\end{align*}
Using again the substitution $r\to r-k+1$, we  re-formulate~\eqref{subst3} as
\begin{align*}
&\sum\limits_{k=-\infty}^{i-1} S^{\txts}(t+i-1-k) \int\limits_{0}^{1} 
S^{\txts}(1-r)P^{\txts } F(U_{r+k-1}) ~\txtd r \\
&+ \sum\limits_{k=-\infty}^{i-1} S^{\txts}(t+i-1-k) \int\limits_{0}^{1} 
S^{\txts}(1-r)P^{\txts } G (U_{r+k-1})~\txtd \Theta_{k-1} \bm{{W}}_{r}.
\end{align*}
Finally, replacing $r$ by $ r - i +1$ in~\eqref{subst4} entails  
\begin{align*}
  \int\limits_{0}^{t} S^{\txts}(t-r)P^{\txts } F(U_{r+i-1})~\txtd r 
	+ \int\limits_{0}^{t} S^{\txts}(t-r) P^{\txts }G(U_{r+i-1})~\txtd 
	\Theta_{i-1} \bm{{W}}_{r}.
 \end{align*}
Summarizing, we have for ${W}\in\Omega_{W}$, $t\in[0,1]$ and $i\in\mathbb{Z}_{-}$ 
that
\begin{align}
&  J({W},U,\xi)[t+i-1]=S^{\txtc}(t+i-1) \xi^{\txtc} \nonumber \\
& - \sum\limits_{k=0}^{i+1} S^{\txtc} (t+i-1-k)  \left(  
\int\limits_{0}^{1} S^{\txtc}(1-r) P^{\txtc }F(U_{r+i-1}) ~\txtd r 
+ \int\limits_{0}^{1} S^{\txtc}(1-r)P^{\txtc } G(U_{r}) ~\txtd 
\Theta_{i-1}\bm{{W}}_{r} \right ) \label{t1}\\
& -\int\limits_{t}^{1} S^{\txtc} (1-r) P^{\txtc }F (U_{r+i-1}) ~\txtd r 
-  \int\limits_{t}^{1} S^{\txtc} (1-r) P^{\txtc }G(U_{r+i-1}) 
~\txtd \Theta_{i-1}\bm{{W}}_{r}\label{hatt}\\
& +\sum\limits_{k=-\infty}^{i-1}  S^{\txts} (t+i-1-k)
\left( \int\limits_{0}^{1} S^{\txts}(1-r) P^{\txts }F(U_{r+k-1}) ~\txtd r 
+  \int\limits_{0}^{1} S^{\txts}(1-r) P^{\txts }G(U_{r+k-1})
 ~\txtd \Theta_{k-1}\bm{{W}}_{r} \right) \label{t2}\\
& +\int\limits_{0}^{t}S^{\txts}(t-r) P^{\txts }F(U_{r+i-1}) ~\txtd r 
+ \int\limits_{0}^{t} S^{\txts}(t-r)P^{\txts }G(U_{r+i-1}) 
~\txtd \Theta_{i-1}\bm{{W}}_{r} \label{t3},
\end{align}
This will lead us to the structure of the \emph{discrete} Lyapunov-Perron map, 
as defined below~\eqref{j}. To simplify the notation, motivated 
by~\eqref{t1},~\eqref{t2} and~\eqref{t3}, 
we write for $(Y,Y')\in D^{2\alpha}_{{W}}([0,1];\cX)$
\begin{align}\label{tt}
T^{\txts/\txtc}({W},Y, Y')[\cdot]:&= \int\limits_{0}^{\cdot} S^{\txts/\txtc} 
(\cdot-r)P^{\txts/\txtc }\overline{F}(Y)(r) ~\txtd r + \int\limits_{0}^{\cdot} S^{\txts/\txtc}(\cdot-r)P^{\txts /\txtc }\overline{G}(Y)(r) ~\txtd  
\bm{{W}}_{r} \\
&=\int\limits_{0}^{\cdot} S^{\txts/\txtc} 
(\cdot-r)P^{\txts/\txtc  } F(Y_r) ~\txtd r + \int\limits_{0}^{\cdot} S^{\txts/\txtc}(\cdot-r)P^{\txts/\txtc }G(Y_{r}) ~\txtd  
\bm{{W}}_{r} ,
\end{align}
and 
\begin{align}
\label{hattt}
\hat{T}^{\txtc}({W}, Y, Y')[\cdot] :&=  \int\limits_{\cdot}^{1} 
S^{\txtc}(\cdot- r ) P^{\txtc }\overline{F}(Y)(r) ~\txtd r + \int\limits_{\cdot} ^{1} 
S^{\txtc} (\cdot -r ) P^{\txts }\overline{G}(Y)(r) ~\txtd \bm{{W}}_{r}\\
&= \int\limits_{\cdot}^{1} 
S^{\txtc}(\cdot- r ) P^{\txtc }F(Y_{r}) ~\txtd r + \int\limits_{\cdot} ^{1} 
S^{\txtc} (\cdot -r )P^{\txtc } G(Y_{r}) ~\txtd \bm{{W}}_{r}.
\end{align}
Of course, a Gubinelli derivative of $T^{\txts/\txtc}({W},Y, Y')$ respectively $\hat{T}^{\txtc}({W}, Y, Y')$ is given by $G(Y)$.
\begin{remark}
When we  work with $F_{R}$ and $G_{R}$ instead of 
$\overline{F}$ and $\overline{G}$, we indicate this fact using the notation $T^{\txts/\txtc}_{R}$, 
respectively $\hat{T}^{\txtc}_{R}$, see also~\eqref{tr}. 
\end{remark}

The main goal now is to find an appropriate framework, in which we can 
formulate a meaningful fixed-point problem for $J$. 

\subsection{The fixed-point argument}
\label{sectfp}

Before we proceed with the existence proof of local center 
manifolds for~\eqref{sde1} we point out the main differences 
between our approach and a known approach for random center manifold 
theory for SDEs driven by linear multiplicative Stratonovich noise, e.g.
\begin{equation}
\label{eqstrat}
\begin{cases}
\txtd u = (A u + f (u)) ~\txtd t + u \circ ~\txtd \tilde{B}_{t}\\
u(0)=x.
\end{cases}
\end{equation}
Here $\tilde{B}$ stands for a two-sided real-valued Brownian motion.
In this case, using the transformation $\widetilde{u}:=u \txte^{-z(\tilde{B})}$, 
where $(t,\tilde{B})\mapsto z(\theta_{t}{\tilde{B}})$ is the Ornstein-Uhlenbeck 
process (recall~\eqref{ouoned}), one obtains the non-autonomous random 
differential equation
\begin{equation}\label{ou}
\txtd u =( A u + z(\theta_{t}{\tilde{B}}) u + g (\theta_{t}{\tilde{B}},u) )~\txtd t,
\end{equation}
where $g(\tilde{B},u):=\txte^{-z(\tilde{B})}f(\txte^{z(\tilde{B})}u)$.
Regarding Assumptions \ref{ass:linearpart} one immediately infers that 
the continuous-time Lyapunov-Perron transform for~\eqref{ou} is given for $t\leq 0$ by
\begin{align}
\label{lpeinfach}
J({\tilde{B}},u,x)[t] & := S^{\txtc}(t) \txte^{\int\limits_{0}^{t} 
z(\theta_{\tau}{\tilde{B}}) ~\txtd \tau} P^{\txtc}x   
+ \int\limits_{0}^{t} S^{\txtc}(t-r)  \txte^{\int\limits_{r}^{t} 
z(\theta_{\tau}{\tilde{B}}) ~\txtd \tau} P^{\txtc} 
g(\theta_{r}{\tilde{B}}, u(r))~\txtd r\nonumber\\
& +\int\limits_{-\infty}^{t} S^{\txts} (t-r ) \txte^{\int\limits_{r}^{t} 
z(\theta_{\tau}{\tilde{B}}) ~\txtd \tau} P^{\txts} g(\theta_{r}{\tilde{B}}, 
u(r))~\txtd r.
\end{align}
Further details on the derivation/setting of this operator can be found 
in~\cite{WanngDuan},~\cite[Sec.~6.2.2]{DuanWang},~\cite[Ch.4]{ChekrounLiuWang} 
and the references specified therein. The next natural step is to show 
that~\eqref{lpeinfach} possesses a fixed-point in a certain function space. 
One possible choice turns out to be $BC^{\eta,z}(\mathbb{R}^{-};\cX)$, 
see~\cite[p.~156]{DuanWang}. This space is defined as
\begin{align}
\label{eq:spaceDW}
BC^{\eta,z}(\mathbb{R}^{-};\cX):=\left\{u:\mathbb{R}^{-}\to \cX, 
~u ~\mbox{is continuous and } \sup\limits_{t\leq 0}\txte^{-\eta t 
-\int\limits_{0}^{t}z(\theta_{\tau}{\tilde{B}})~\txtd \tau } |u(t)|<\infty\right\}
\end{align}
and is endowed with the norm
\begin{align*}
||u||_{BC^{\eta,z}} := \sup\limits_{t\leq 0}~\txte^{-\eta t 
-\int\limits_{0}^{t}z(\theta_{\tau}{\tilde{B}})~\txtd \tau }|u(t)|.
\end{align*}
Here $\eta$ is determined from~\eqref{gamma} and~\eqref{beta}, namely one 
has $-\beta<\eta<0$. Note that the previous expressions are well-defined  
since 
\benn
\lim\limits_{t\to \pm\infty}\frac{|z(\theta_t{\tilde{B}})|}{|t|}=0, 
\eenn
according to~\cite[Lem.~2.1]{DuanLuSchmalfuss} and the references specified 
therein. Consequently, in this case the Lyapunov-Perron map~\eqref{lpeinfach} works
with an implicitly transformed equation and not directly with the original problem.\medskip

In our context, we directly work with solution of the RDE~\eqref{sde1}. However, since our 
Lyapunov-Perron transform~\eqref{lp} contains stochastic integrals, the entire 
machinery applicable to~\eqref{lpeinfach} breaks down. Therefore, we have to 
find an appropriate setting for the fixed-point argument. To this aim we introduce 
now a function space which helps us incorporate the discretized version 
of~\eqref{lp} derived in the previous subsection. Namely, we work with the 
space of sequences $BC^{\eta}(D^{2\alpha}_{{W}})$, for $-\beta<\eta<0$, whose elements are constituted 
by $\cX$-valued controlled rough paths on $[0,1]$. 

\begin{definition}
We say that a sequence of controlled rough paths $\mathbb{U}:=((U^{i-1}, 
(U^{i-1})'))_{i\in\mathbb{Z}^{-}}$ with $U^{i-1}_{0}=U^{i-2}_{1}$ belongs to the space 
$BC^{\eta}(D^{2\alpha}_{{W}}) $ if
\begin{equation}\label{bcnorm}
\|\mathbb{U}\|_{BC^{\eta}\left(D^{2\alpha}_{{W}}\right)}:=
	\sup\limits_{i\in\mathbb{Z}^{-}} \txte^{-\eta (i-1)} \|U^{i-1},
	(U^{i-1})'\|_{\textbf{D}^{2\alpha}_{{W}}([0,1];\cX)}<\infty.
\end{equation} 
\end{definition}

For the computation of the local center manifolds we firstly 
modify~\eqref{sde1} with the cut-off function given in Section~\ref{cut}
i.e., we replace $\overline{F}$ and $\overline{G}$ by $F_{R}$ respectively $G_{R}$. 

Motivated by Subsection~\ref{dlp} we are justified to introduce the 
discrete Lyapunov-Perron transform $J_{d}({W}, \mathbb{U},\xi)$ for a 
sequence of controlled rough paths $\mathbb{U}\in BC^{\eta}(D^{2\alpha}_{{W}}) $ 
and $\xi\in \cX$ as the pair $J_{d}({W},\mathbb{U},\xi):=(J^{1}_{d}({W},
\mathbb{U},\xi), J^{2}_{d}({W},\mathbb{U},\xi))$, where the precise structure 
is given below. The dependence of $J_d$ on the cut-off parameter $R$ is indicated by the subscript $R$. For $t\in[0,1], ~{W}\in\Omega_{W} \mbox{ and } i\in\mathbb{Z}^{-}$ 
we define

\begin{align}\label{j}
&J^{1}_{R,d}({W}, \mathbb{U},\xi)[i-1,t] : = S^{\txtc}(t+i-1) 
\xi^{\txtc} \\
&  -\sum\limits_{k=0}^{i+1} S^{\txtc} (t+i-1-k) 
\left(\int\limits_{0}^{1} S^{\txtc} (1-r)P^{\txtc } F_{R}(U^{k-1})(r) ~\txtd r
 + \int\limits_{0}^{1} S^{\txtc}(1-r)P^{\txtc }G_{R}(U^{k-1})(r) ~\txtd \Theta_{k-1}
 \bm{{W}}_{r}  \right)\nonumber\\
& - \int\limits_{t}^{1} S^{\txtc} (t-r)P^{\txtc }F_{R}(U^{i-1})(r) ~\txtd r 
- \int\limits_{t}^{1} S^{\txtc} (t-r) P^{\txtc }G_{R}(U^{i-1})(r) ~\txtd 
\Theta_{i-1} \bm{{W}}_{r}\nonumber\\
& + \sum\limits_{k=-\infty}^{i-1} S^{\txts} (t+i-1-k) \left(
\int\limits_{0}^{1} S^{\txts} (1-r) P^{\txts }F_{R}(U^{k-1})(r) ~\txtd r 
+ \int\limits_{0}^{1} S^{\txts}(1-r)P^{\txts }G_{R}(U^{k-1})(r) ~\txtd \Theta_{k-1} 
\bm{{W}}_{r}  \right)\nonumber\\
& + \int\limits_{0}^{t} S^{\txts} (t-r)P^{\txts }F_{R}(U^{i-1})(r) ~\txtd r
 + \int\limits_{0}^{t} S^{\txts} (t-r) P^{\txts }G_{R}(U^{i-1})(r) ~\txtd 
\Theta_{i-1} \bm{{W}}_{r}. ~~ \nonumber
\end{align}
Furthermore, $J^{2}_{R,d}({W},\mathbb{U},\xi)$ stands for the Gubinelli 
derivative of $J^{1}_{R,d}({W},\mathbb{U},\xi)$, i.e.~$J^{2}_{R,d}({W},
\mathbb{U},\xi)[i-1,\cdot]:=(J^{1}_{R,d}({W},\mathbb{U},\xi)[i-1,\cdot] )'$. 
Note that $\xi^{\txtc}$ can be recovered by applying the projection to $J^1_{R,d}$ on the center subspace and setting $i=0$ and $t=1$, i.e., $P^{\txtc}J^{1}_{R,d}({W},
\mathbb{U},\xi)[-1,1]=\xi^{\txtc}$.\medskip

We emphasize that for a sequence $\mathbb{U}\in BC^{\eta}(D^{2\alpha}_{{W}})$ 
the first index $i\in\mathbb{Z}_{-}$ in the definition of $J_{R,d}({W},
\mathbb{U},\xi)[\cdot,\cdot]$ gives the position within the sequence and 
the second one refers to the time variable $t\in[0,1]$.\\

We are going to show that~\eqref{j} maps $BC^{\eta}(D^{2\alpha}_{{W}})$ 
into itself and is a contraction if the constant $K$ specified in~\eqref{k} is 
chosen small enough, as justified by the following computation. We let $C_{S}$ stand for a 
constant which exclusively depends on the group $S$ and state our first main result.

\begin{theorem}
\label{contraction} Let Assumptions~\ref{ass:linearpart},~\nameref{f} and~\nameref{gi} 
hold true and let $K$ satisfy the gap condition
\begin{align}\label{gap:k}
	K  \left( \frac{ \txte^{\beta+\eta}(C_{S}M_{s}
	\txte^{-\eta}+1)}{1-\txte^{-(\beta+\eta)}} + 
	\frac{\txte^{\gamma-\eta} (C_{S}M_{c}\txte^{-\eta}+1)}{1-
	\txte^{-(\gamma-\eta)}} \right) <\frac{1}{4}.
\end{align}
Then, the map $J_{d}:\Omega\times BC^{\eta}(D^{2\alpha}_{{W}})\to 
BC^{\eta}(D^{2\alpha}_{{W}}) $ possesses a unique fixed-point 
$\Gamma\in BC^{\eta}(D^{2\alpha}_{{W}})$.
\end{theorem}

\begin{remark}
Note that~\eqref{gap:k} can be obtained for instance by choosing 
the constant appearing in~\eqref{k} as 
\begin{align}\label{gk}
		K^{-1}:= 4 \txte ^{(\beta+\gamma)/2}\left(  
		\frac{\txte^{(\beta-\gamma)/2} C_{S}(M_{s}+M_{c}) 
		+1} {1-\txte^{-(\beta+\gamma)/2}} \right),
\end{align}
which follows by setting $\eta:=\frac{-\beta+\gamma}{2}<0$.
\end{remark}
	
\begin{proof}
Let two sequences $\mathbb{U}=((U^{i-1}, (U^{i-1})'))_{i\in\mathbb{Z}^{-}}$ 
and $\mathbb{\widetilde{U}}=((\widetilde{U}^{i-1}, 
(\widetilde{U}^{i-1})'))_{i\in\mathbb{Z}^{-}}$ that belong to 
$BC^{\eta}(D^{2\alpha}_{{W}}) $ and satisfy $P^{\txtc}U^{-1}_{1}=
P^{\txtc}\widetilde{U}^{-1}_{1}=\xi^{\txtc}$. We verify the contraction 
property. The fact that $J_{d}(\cdot)$ maps $BC^{\eta}(D^{2\alpha}_{{W}})$ 
into itself can be derived by setting $\widetilde{\mathbb{U}}=0$ in the 
next computation and using that $F_{R}(0)=G_{R}(0)=0$.
According to~\eqref{2alpha:abed} we have 
\begin{align*}
\|S^{\txtc} (t+i-1)\xi^{\txtc},0 \|_{BC^{\eta}(D^{2\alpha}_{{W}})}& 
= ||| S^{\txtc} (\cdot+i-1)\xi^{\txtc}|||_{2\alpha} \txte^{-\eta (i-1)}\\
&=|S^{\txtc}(i-1) \xi^{\txtc}|~ |||S(\cdot)|||_{2\alpha} \txte^{-\eta(i-1)} \\
&\leq C_{S}M_{c} \txte^{(\gamma-\eta)(i-1)}|\xi^{\txtc}|.
\end{align*}
The previous expression remains bounded for $i\in\mathbb{Z}^{-}$ since 
we assumed that $-\beta<\eta<0\leq\gamma<\beta$. Next, we are going to estimate 
the difference
\begin{align*}
 	||J_{d}({W},\mathbb{U},\xi) - J_{d}({W},\widetilde{\mathbb{U}},\xi)
||_{BC^{\eta}(D^{2\alpha}_{{W}})}
 	\end{align*}
in several intermediate steps. Verifying the contraction property on the 
stable part of~\eqref{j}, one has to compute two terms. First of all, 
due to~\eqref{wanttohave}
\begin{align*}
& \sum\limits_{k=-\infty}^{i-1} \txte^{-\eta(i-1)}
|||S^{\txts}(\cdot+i-1-k)|||_{2\alpha} \|T^{\txts}_{R}(\Theta_{k-1}{W}, 
U^{k-1}, (U^{k-1})') [1]- T^{\txts}_{R}(\Theta_{k-1}{W}, 
\widetilde{U}^{k-1}, (\widetilde{U}^{k-1})')[1] \|_{\textbf{D}^{2\alpha}_{{W}}}\\
& \leq \sum\limits_{k=-\infty}^{i-1}C_{S} M_{s}\txte^{-\eta(i-1)} 
\txte^{-\beta(i-1-k)}  K \|U^{k-1}-\widetilde{U}^{k-1}, 
(U^{k-1}-\widetilde{U}^{k-1})'  \|_{\textbf{D}^{2\alpha}_{{W}}}\\
& \leq \sum\limits_{k=-\infty}^{i-1} C_{S}M_{s} \txte^{-\eta(i-1)} 
\txte^{-\beta(i-1-k)} \txte^{\eta(k-1)} K \txte^{-\eta(k-1)} 
\| U^{k-1}-\widetilde{U}^{k-1},(U^{k-1} -\widetilde{U}^{k-1})'
\|_{\textbf{D}^{2\alpha}_{{W}}}\\
& \leq \sum\limits_{k=-\infty}^{i-1}   \txte^{-(\eta+\beta) 
(i-1-k) }C_{S} M_{s} \txte^{-\eta} K \txte^{-\eta(k-1)} \| 
U^{k-1}-\widetilde{U}^{k-1},(U^{k-1} 
-\widetilde{U}^{k-1})'\|_{\textbf{D}^{2\alpha}_{{W}}}.
\end{align*}
Combining this with the last term of (\ref{j})  entails the final 
estimate on the stable part
 \begin{align*}
 & \sum\limits_{k=-\infty}^{i-1} \txte^{-\eta(i-1)}|||
S^{\txts}(\cdot+i-1-k)|||_{2\alpha} \|T^{\txts}_{R}(\Theta_{k-1}{W}, 
U^{k-1}, (U^{k-1})') [1]- T^{\txts}_{R}(\Theta_{k-1}{W}, 
\widetilde{U}^{k-1}, (\widetilde{U}^{k-1})')[1] \|_{\textbf{D}^{2\alpha}_{{W}}}\\
 & + \txte^{-\eta(i-1)} \|T^{\txts}_{R}(\Theta_{i-1}{W}, U^{i-1}, 
(U^{i-1})') [\cdot]- T^{\txts}_{R}(\Theta_{i-1}{W}, \widetilde{U}^{i-1}, 
(\widetilde{U}^{i-1})')[\cdot] \|_{\textbf{D}^{2\alpha}_{{W}}}\\
 & \leq \sum\limits_{k=-\infty}^{i} \txte^{-(\eta+\beta)(i-k-1)} 
K (C_{S}M_{s}\txte^{-\eta} + 1)  e ^{-\eta(k-1)}  \| U^{k-1}-
\widetilde{U}^{k-1},(U^{k-1} -\widetilde{U}^{k-1})'\|_{\textbf{D}^{2\alpha}_{{W}}}\\
 & \leq 	K   \frac{ \txte^{\beta+\eta}(C_{S}M_{s}\txte^{-\eta}+1)}{1
-\txte^{-(\beta+\eta)}}  \|\mathbb{U}-\widetilde{\mathbb{U}}, \mathbb{U}'-\widetilde{\mathbb{U}}'
\|_{BC^{\eta}(D^{2\alpha}_{{W}})}.
 \end{align*}
We focus now on the center part. Here we obtain
\begin{align*}
& \sum\limits_{k=0}^{i+1} \txte^{-\eta(i-1)} |||S^{\txtc} 
(\cdot + i -1-k ) |||_{2\alpha} \|T^{\txtc}_{R}(\Theta_{k-1}{W}, 
U^{k-1}, (U^{k-1})') [1]- T^{\txtc}_{R}(\Theta_{k-1}{W}, 
\widetilde{U}^{k-1}, (\widetilde{U}^{k-1})')[1] \|_{\textbf{D}^{2\alpha}_{{W}}}\\
 & \leq \sum\limits_{k=0}^{i+1}C_{S} M_{c} e ^{-\eta (i-1)} 
\txte^{\gamma(i-1-k)} K \| U^{k-1}-\widetilde{U}^{k-1},(U^{k-1} 
-\widetilde{U}^{k-1})'\|_{\textbf{D}^{2\alpha}_{{W}}}\\
 & \leq \sum\limits_{k=0}^{i+1} C_{S}M_{c} \txte^{-\eta(i-1)} 
\txte^{\gamma(i-1-k)} \txte^{\eta(k-1)} \txte^{-\eta(k-1)} 
K \| U^{k-1}-\widetilde{U}^{k-1},(U^{k-1} -\widetilde{U}^{k-1})'
\|_{\textbf{D}^{2\alpha}_{{W}}}\\
 & \leq \sum\limits_{k=0}^{i+1} C_{S}M_{c} \txte^{(\gamma-\eta)(i-1-k)} 
\txte^{-\eta}  K \txte^{-\eta(k-1)}  \| U^{k-1}-\widetilde{U}^{k-1},
(U^{k-1} -\widetilde{U}^{k-1})'\|_{\textbf{D}^{2\alpha}_{{W}}}.
\end{align*}
Combining this and estimating the third summand in~\eqref{j} yields
\begin{align*}
 & \sum\limits_{k=0}^{i+1} \txte^{-\eta(i-1)} |||S^{\txtc} (\cdot + 
i -1-k ) |||_{2\alpha} \|T^{\txtc}_{R}(\Theta_{k-1}{W}, U^{k-1}, 
(U^{k-1})') [1]- T^{\txtc}_{R}(\Theta_{k-1}{W}, \widetilde{U}^{k-1},
 (\widetilde{U}^{k-1})')[1] \|_{\textbf{D}^{2\alpha}_{{W}}}\\
 & + \txte^{-\eta(i-1)} ||\hat{T}^{\txtc}_{R}(\Theta_{i-1}{W}, U^{i-1}, 
(U^{i-1})') [\cdot]- \hat{T}^{\txtc}_{R}(\Theta_{i-1}{W}, 
\widetilde{U}^{i-1}, (\widetilde{U}^{i-1})')[\cdot] ||_{\textbf{D}^{2\alpha}_{{W}}}\\
 & \leq \sum\limits_{k=0}^{i} \txte^{(\gamma-\eta)(i-1-k)} K (C_{S}M_{c} 
\txte^{-\eta} +1) e^{-\eta(k-1)} \| U^{k-1}-\widetilde{U}^{k-1},(U^{k-1} 
-\widetilde{U}^{k-1})'\|_{\textbf{D}^{2\alpha}_{{W}}}\\
 & \leq K 	 \frac{\txte^{\gamma-\eta} (C_{S}M_{c}\txte^{-\eta}+1)}{1
-\txte^{-(\gamma-\eta)}} ||\mathbb{U}-\widetilde{\mathbb{U}}, \mathbb{U}'-\widetilde{\mathbb{U}}'
||_{BC^{\eta}(D^{2\alpha}_{{W}})}.
\end{align*}
Due to~\eqref{gap:k} we have that 
\begin{align*}
\|J_{R,d}({W},\mathbb{U},\xi) - J_{R,d}({W},\mathbb{\widetilde{U}},
\xi)\|_{BC^{\eta}(D^{2\alpha}_{{W}})} \leq \frac{1}{4} \|\mathbb{U} - 
\widetilde{\mathbb{U}}, \mathbb{U}'-\widetilde{\mathbb{U}}'\|_{BC^{\eta}(D^{2\alpha}_{{W}})}.
\end{align*}
 Applying Banach's fixed-point theorem, we infer that $J_{R,d}({W},
\mathbb{U},\xi^{\txtc})$ possesses a unique fixed-point $\Gamma(\xi^{\txtc},
{W})\in BC^{\eta}(D^{2\alpha}_{{W}})$ for each fixed $\xi^{\txtc}\in \cX^{\txtc}$. 
\end{proof}
 	 
The fixed-point will further help us to characterize the local center manifold.

\begin{lemma}\label{lipfp}
	The mapping $\xi^{\txtc} \mapsto \Gamma(\xi^{\txtc},{W})\in 
	BC^{\eta}(D^{2\alpha}_{{W}})$ is Lipschitz continuous.
\end{lemma}
\begin{proof}
One easily obtains for $\xi^{\txtc}_{1}$, $\xi^{\txtc}_{2}\in \cX^{\txtc}$ that
\begin{align*}
\| \Gamma(\xi^{\txtc}_{1},{W}) - \Gamma(\xi^{\txtc}_{2},{W})\|_{BC^{\eta}
(D^{2\alpha}_{{W}})} &= \|J_{R,d}({W},\Gamma(\xi^{\txtc}_{1},{W}),\xi^{\txtc}_{1}) 
- J_{R,d}({W},\Gamma(\xi_{2}^{\txtc},{W}),\xi^{\txtc}_{2}) \|_{BC^{\eta}
(D^{2\alpha}_{{W}})}\\
& \leq \|J_{R,d}({W},\Gamma(\xi^{\txtc}_{1},{W}),\xi^{\txtc}_{1}) - J_{R,d}
({W},\Gamma(\xi_{1}^{\txtc},{W}),\xi^{\txtc}_{2}) \|_{BC^{\eta}(D^{2\alpha}_{{W}})} \\
&+ \|J_{R,d}({W},\Gamma(\xi^{\txtc}_{1},{W}),\xi^{\txtc}_{2}) - J_{R,d}({W},
\Gamma(\xi_{2}^{\txtc},{W}),\xi^{\txtc}_{2}) \|_{BC^{\eta}(D^{2\alpha}_{{W}})}\\
& \leq \|S^{\txtc}(\cdot+i-1) (\xi^{\txtc}_{1} -\xi^{\txtc}_{2}),0 
\|_{BC^{\eta}(D^{2\alpha}_{{W}})} + \frac{1}{4} \|  \Gamma(\xi^{\txtc}_{1},{W}) 
- \Gamma(\xi^{\txtc}_{2},{W}) \|_{BC^{\eta}(D^{2\alpha}_{{W}})}\\
& \leq C_{S} M_{c}e ^{\gamma} |\xi^{\txtc}_{1}- \xi^{\txtc}_{2}| + \frac{1}{4} \|
  \Gamma(\xi^{\txtc}_{1},{W}) - \Gamma(\xi^{\txtc}_{2},{W}) 
	\|_{BC^{\eta}(D^{2\alpha}_{{W}})},
\end{align*}
which proves the statement. 
\end{proof}

Before stating the existence result of local center manifolds for~\eqref{sde1} 
we must fix further notations. In the following we write $U_{\cdot}(\xi)$ to 
emphasize the dependence of the path $U_{\cdot}$ on the initial condition $\xi$ 
of the RDE~\eqref{sde1}. Furthermore, consider $\Gamma(\xi^{\txtc},{W})$, which 
is the fixed-point of $J_{R,d}({W},\mathbb{U},\xi^{\txtc})$ belonging to 
$BC^{\eta}(D^{2\alpha}_{{W}})$. Again, the first index in the notation 
$\Gamma(\xi^{\txtc},{W})[\cdot,\cdot]$ signifies the position within the 
sequence and the second one refers to the time-variable. Consequently, for a 
fixed $k$, $\Gamma(\xi^{\txtc},{W})[k,\cdot]\in D^{2\alpha}_{{W}}([0,1];\cX)$.
In order to justify the invariance of the center manifold as required in 
Definition \ref{def:cm}, we must show that if $\xi^{\txtc}\in \cM^{\txtc}_{loc}({W})$ 
the solution of the RDE~\eqref{sde1} at a time-point $\tau$ having $\xi^{\txtc}$ 
as initial condition belongs to $\cM^{\txtc}_{loc}(\Theta_{\tau}{W})$. 
Consequently, this means that we have to consider an appropriate shift of ${W}$, 
so in our setting we have to analyze for instance expressions like 
$\Gamma(\xi^{\txtc},\Theta_{\tau}{W})[\cdot,\cdot]$.

\begin{remark}
\label{obvious}
Note that the controlled rough path $(V_{\cdot}, V'_{\cdot})
:=(\Gamma(\xi^{\txtc},{W})[-1,\cdot], (\Gamma(\xi^{\txtc},{W})[-1,\cdot])' )
\in D^{2\alpha}_{{W}}([0,1];\cX)$ solves the RDE
\begin{align}
\label{a1}
	\begin{cases}
	\txtd V = (A V + F_{R} (V))~\txtd t + G_{R}(V) ~\txtd \bm{W}_{t}\\
	V(0)=\Gamma(\xi^{\txtc},{W})[-1,0]\in \cX.
	\end{cases}
\end{align}
This follows from the definition of $J_{R,d}(W\,\mathbb{U}, 
\xi^{\txtc})$ regarding that $\Gamma(\xi^{\txtc},W)$ is the unique 
fixed-point of $J_{R,d}(W,\mathbb{U},\xi^{\txtc})$. To obtain the initial 
condition $\Gamma(\xi^{\txtc},W)[-1,0]$ in~\eqref{a1} one sets $i=0$ 
and $t=0$ in~\eqref{j}.
\end{remark} 

Keeping the previous considerations in mind we can state the next 
fundamental result, which characterizes the local center manifold 
of~\eqref{sde1}. In the following we denote by $B_{\cX^{\txtc}}(0,\frak{r}({W}))$ 
a ball in $\cX^{\txtc}$, which is centered in $0$ and has a random 
radius $\frak{r}({W})$.

\begin{lemma}
\label{localcman} There exists a tempered from below random variable 
$\rho({W})$ such that the local center manifold of~\eqref{sde1} is given 
by the graph of a Lipschitz function, namely
\begin{equation}
\cM^{\txtc}_{loc} ({W})=\{ \xi + h^{\txtc}(\xi,{W}) 
: \xi\in B_{\cX^{\txtc}}(0,\rho({W})) \},
\end{equation}
where we define
\begin{align*}
&	h^{\txtc}(\xi,{W}):=P^{\txts}\Gamma(\xi,{W})[-1,1]|_{B_{\cX^{\txtc}}(0,\rho({W}))},
\end{align*}
and consequently
\begin{align*}
	h^{\txtc}(\xi,{W}) &= \sum\limits_{k=-\infty}^{0}S^{\txts}(-k) 
	\int\limits_{0}^{1} S^{\txts}(1-r) P^{\txts }F(\Gamma(\xi,{W})[k-1,r]) ~\txtd r\\
	& +\sum\limits_{k=-\infty}^{0}S^{\txts}(-k) \int\limits_{0}^{1} 
	S^{\txts}(1-r)P^{\txts } G(\Gamma(\xi,{W})[k-1,r]) ~\txtd \Theta_{k-1}\bm{{W}}_{r}.
\end{align*}
\end{lemma}

\begin{proof}
First of all, since $F(0)=G(0)=0$ we have that $h^{\txtc}(0,{W})=0$, 
consequently $0\in \cM^{\txtc}_{loc}({W})$. Regarding this, the tangency 
condition will be clear in Section~\ref{smoothness} when we investigate 
the smoothness of $\cM^{\txtc}_{loc}({W})$. We now show that 
$\cM^{\txtc}_{loc}({W})$ is a local center manifold with discrete time 
for the random dynamical system $\varphi$ associated to~\eqref{sde1}. 
Namely, for initial conditions $\xi\in B_{\cX^{\txtc}}(0,\rho({W}))$, 
where $\rho({W})$ will be appropriately chosen, we are going to show that 
\benn
\Gamma(\xi,{W})[-1,1]\in \cM^{\txtc}_{loc}({W})\quad \text{and}\quad 
\Gamma(\xi,{W})[i-1,1]\in \cM^{\txtc}_{loc}(\Theta_{i}{W})
\eenn
for $i\in \mathbb{Z}^{-}$. Furthermore, the corresponding trajectories 
starting in $\Gamma(\xi,{W})[-1,1]$ and $\Gamma(\xi,{W})[i-1,1]$ remain 
within a ball with a tempered radius. We set $\widetilde{y}^{i}_{\cdot}(\xi)
:=\Gamma(\xi,{W})[i-1,\cdot]$. We index $\widetilde{y}$ by $i$ and not $i-1$ 
because we want to derive expressions for $\cM^{c}_{loc}({W})$ respectively 
$\cM^{c}_{loc}(\Theta_{i}{W})$ instead of $\cM^{c}_{loc}(\Theta_{i-1}{W})$. When we compute the 
$\textbf{D}^{2\alpha}_{{W}}$-norm we use for simplicity the notation $\widetilde{y}$. 
However, when we analyze the shift with respect to $W$ of 
$\Gamma(\xi,{W})[\cdot,\cdot]$, i.e., $\Gamma(\xi,\Theta_{i}{W})[\cdot,\cdot]$ 
we explicitly write all the arguments.  \\
In order to justify the claimed assertions we firstly set 
\begin{align}\label{rho}
\rho({W}):=\frac{R(\Theta_{-1}{W})}{2L_{\Gamma}\txte^{-\eta}},
\end{align}	
where $R$ was introduced in~\eqref{r} and $L_{\Gamma}$ denotes the Lipschitz constant of the mapping 
$\xi^{\txtc} \mapsto \Gamma(\xi^{\txtc},{W})\in BC^{\eta}(D^{2\alpha}_{{W}})$ 
as obtained in Lemma \ref{lipfp}. Therefore, 
for $\xi\in \cX$ we have
\begin{align*}
\| \Gamma(P^{\txtc}\xi,{W}) \|_{BC^{\eta}(D^{2\alpha}_{{W}})} 
\leq L_{\Gamma} |P^{\txtc}\xi|,
\end{align*}
so letting  $\xi\in B_{\cX^{\txtc}}(0,\rho({W}))$ entails
\begin{align}\label{kugel1}
\| \Gamma(\xi,{W})\|_{BC^{\eta}(D^{2\alpha}_{{W}})} \leq L_{\Gamma} \rho({W}).
\end{align}
Using the definition of the $\|\cdot\|_{BC^{\eta}(D^{2\alpha}_{{W}})}$-norm, 
the previous inequality rewrites as
\begin{align}
\label{pi}
\sup\limits_{i\in\mathbb{Z}_{-}} \txte^{-\eta(i-1)} \|\widetilde{y}^{i}(\xi), 
(\widetilde{y}^{i}(\xi))' \|_{\textbf{D}^{2\alpha}_{W}} \leq L_{\Gamma} 
\frac{R(\Theta_{-1}W)}{2L_{\Gamma}\txte^{-\eta}}.
\end{align}
Setting $i=0$ in~\eqref{pi} we infer that for $|\xi|\leq\rho(W)$, 
the norm of the trajectory $\widetilde{y}^{-1}_{\cdot}(\xi)=
\Gamma(\xi,W)[-1,\cdot]$ can be estimated by
\begin{align*}
\|\widetilde{y}^{0}(\xi), (\widetilde{y}^{0}(\xi))' \|_{\textbf{D}^{2\alpha}_{{W}}} 
\leq \frac{R(\Theta_{-1}{W})}{2}.
\end{align*}
The next step is to derive that $\Gamma(\xi,W)[i-1,1]=\widetilde{y}^{i}_{1}(\xi) 
\in \cM^{\txtc}_{loc}(\Theta_{i}{W})$ and to show that for the corresponding 
trajectory, the relation 
\begin{align*}
\| \widetilde{y}^{i}(\xi), (\widetilde{y}^{i}(\xi))'
\|_{\textbf{D}^{2\alpha}_{{W}}}\leq \frac{R(\Theta_{i-1}{W})}{2}
\end{align*}
holds true. To this aim, we firstly employ~\eqref{temperedmanradius}. Note 
that this is also valid in the discrete-time setting, according 
to~\cite[Sec. 4.1.1]{Arnold}. This further yields that there exists a positive 
random variable $\hat{\rho}({W})$ and a constant (which we choose $L_{\Gamma}$) such that 
\begin{align}
\label{tempered1}
\rho(\Theta_{i}{W}) \geq\hat{\rho}({W}) L_{\Gamma} \txte^{\eta(i-1)}.
\end{align}
Now, taking $\xi\in B_{\cX^{\txtc}}(0,\hat\rho({W}))$ we have according 
to~\eqref{kugel1} that
\begin{align}
\label{step1}
\sup\limits_{i\in\mathbb{Z}_{-}} \txte^{-\eta(i-1)} \|\widetilde{y}^{i}(\xi), 
(\widetilde{y}^{i}(\xi))' \|_{\textbf{D}^{2\alpha}_{W}} \leq L_{\Gamma} \hat{\rho}({W}).
\end{align}
Combining this with~\eqref{tempered1} entails
\begin{align}
\label{telem}
\rho(\Theta_{i}{W}) \geq \hat{\rho}({W}) L_{\Gamma} \txte^{\eta(i-1)} 
\geq |P^{\txtc}\widetilde{y}^{i}_{1} (\xi)|=|P^{\txtc}\Gamma(\xi,W)[i-1,1]|.
\end{align}
Summarizing we obtained for $\xi\in B_{\cX^{\txtc}}(0,\hat{\rho}(W))$ that
\begin{align*}
& \Gamma(\xi,W)[-1,1] \in \cM^c_{loc}(W)\\
&  \Gamma(\xi,W)[i-1,1]\in \cM^c_{loc}(\Theta_{i}W).
\end{align*}
Secondly, in order to derive the invariance property of $\cM^{\txtc}_{loc}(W)$, 
analogously to~\cite[Lem.~4.7]{GarridoLuSchmalfuss} or~\cite[Lem.~5.5]{LuSchmalfuss} 
one can establish a connection between the fixed-points of $J_{d}({W},\cdot,\cdot)$ 
and $J_{d}(\Theta_{i}{W},\cdot,\cdot)$ for $i\in\mathbb{Z}^{-}$. For the sake of 
completeness, the main steps of this proof are indicated in Appendix~\ref{a}. 
Observing that $\Gamma(\xi,{W})[i-1,0]=\Gamma(\xi,{W})[i-2,1]$ one infers 
using~\eqref{fpomega}
\begin{align}
\label{gammai}
\Gamma(\xi,{W})[i-1,\cdot]=\Gamma(P^{\txtc}\Gamma(\xi,{W})[i-1,1],\Theta_{i}{W})[-1,\cdot],
\end{align}
which tells us that $\Gamma(\xi,{W})[i-1,\cdot]$ can be obtained from 
$\Gamma(\cdot,\Theta_{i}{W})[-1,\cdot]$ on the $\Theta_{i}{W}$-fiber starting 
with the initial condition $P^{\txtc}(\Gamma(\xi,{W})[i-1,1])$. For more information 
and a detailed computation compare~\eqref{fpappendix} and consult 
Appendix~\ref{a}. Regarding this, we derive using~\eqref{step1} and~\eqref{telem}
\begin{align*}
\| \widetilde{y}^{i}_{\cdot}(\xi), ( \widetilde{y}^{i}_{\cdot}(\xi))'
\|_{\textbf{D}^{2\alpha}_{{W}}} \leq L _{\Gamma} \txte^{-\eta}\rho(\Theta_{i}W), 
\end{align*}
and consequently obtain
\begin{align}
\label{onecutoff}
\| \widetilde{y}^{i}_{\cdot}(\xi), ( \widetilde{y}^{i}_{\cdot}(\xi))'
\|_{\textbf{D}^{2\alpha}_{{W}}}\leq\frac{R(\Theta_{i-1}{W})}{2}.
\end{align}
Recalling Remark~\ref{obvious}, the estimate~\eqref{onecutoff} allows us to drop the cut-off parameter $R$ and
leads to 
\begin{align*}
&(\Gamma(\xi,{W})[i-1,t], (\Gamma(\xi,{W})[i-1,t])') 
= (\widetilde{y}^{i}_{t}(\xi), (\widetilde{y}^{i}_{t}(\xi))')\\
&= \left(S(t) \Gamma(\xi,{W})[i-1,0] +\int\limits_{0}^{t} 
S(t-r) F(\widetilde{y}^{i}_{r}) ~\txtd r
 + \int\limits_{0}^{t} S(t-r) G(\widetilde{y}^{i}_{r}) ~\txtd 
\Theta_{i-1}\bm{{W}}_{r}, G(\widetilde{y}^{i}_{t}) \right)\\
 & = \left( S(t)\Gamma(\xi,{W})[i-2,1] +\int\limits_{0}^{t} 
S(t-r) F(\widetilde{y}^{i}_{r}) ~\txtd r
 + \int\limits_{0}^{t} S(t-r) G(\widetilde{y}^{i}_{r}) 
~\txtd \Theta_{i-1}\bm{{W}}_{r}, G(\widetilde{y}^{i}_{t}) \right).
\end{align*}
This means that
\begin{align}
\label{x}
\widetilde{y}^{i}_{t}=\Gamma(\xi,{W})[i-1,t]=
\varphi(t,\Theta_{i-1}{W},\Gamma(\xi,{W})[i-2,1] ), 
\end{align}
so setting $t=1$ in~\eqref{x} one obtains
\begin{align*}
\varphi(1,\Theta_{i-1}{W}, \Gamma(\xi,{W})[i-2,1])=
\Gamma(\xi,{W})[i-1,1]\in \cM^{\txtc}_{loc}(\Theta_{i}{W}).
\end{align*}
Now, the cocycle property established in Lemma~\ref{cocycle} implies that
\begin{align*}
\cM^{\txtc}_{loc}({W}) \ni \Gamma(\xi,{W})[-1,1]= 
\varphi(1,\Theta_{-1}{W}, \Gamma(\xi,{W})[-1,0] ) 
= \varphi(-i+1,\Theta_{i-1}{W}, \Gamma(\xi,{W})[i-1,0] ).
\end{align*}
Letting $j:=-i+1$ in the previous relation yields
\begin{align*}
\varphi(j,\Theta_{-j}{W}, \Gamma(\xi,{W})[-j-1,1]) 
\in \cM^{\txtc}_{loc}({W}).
\end{align*}
Replacing $\Theta_{-j}{W}$ by ${W}$, we finally conclude that
\begin{align*}
 \varphi(j,{W}, \Gamma(\xi,\Theta_{j}{W})[-j-1,1]) 
\in \cM^{\txtc}_{loc}(\Theta_{j}{W}).
\end{align*}
One can extend these results to the continuous-time setting, 
namely one follows the steps presented in the previous proof 
replacing $i-1$ by $i-1 +t$, where $i\in\mathbb{Z}^{-}$ and
$t\in(0,1)$. This can easily be achieved regarding that
\begin{align*}
\varphi(-i+1,\Theta_{i-1}{W},\Gamma(\xi,{W})[i-1,0])=
\varphi(-i+1-t, \Theta_{i-1+t}{W},\Gamma(\xi,{W})[i-1,t]),
\end{align*}
according to the cocycle property. Consequently, for a sufficiently 
small initial condition $\xi$, i.e.~$\xi\in B_{\cX^{\txtc}}(0,\hat{\rho}({W}))$, 
one can show that $\Gamma(\xi,{W})[i-1,t]\in 
\cM^{\txtc}_{loc}(\Theta_{t+i-1}{W})$. Indeed, as argued before one 
constructs as in~\eqref{telem} a random tempered radius $\hat{r}({W})$ 
such that $|P^{\txtc}\Gamma(\xi,{W})[i-1,t]|\leq \hat{r}(\Theta_{t+i-1}{W})$; 
see also Appendix~\ref{a} and~\cite{GarridoLuSchmalfuss,LianLu}.
\end{proof}

Putting all these insights together we infer:

\begin{theorem}
\label{lemma:cm} 
Let Assumptions~\eqref{ass:linearpart}, {\bf(F)} and {\bf(G)} be satisfied. Then the RDE~\eqref{sde1} admits a local center manifold $\cM^{\txtc}_{loc}({W})$ given by
\begin{equation*}
		\cM^{\txtc}_{loc} ({W})=\{ \xi + h^{\txtc}(\xi,{W}) 
		: \xi\in B_{X^{\txtc}}(0,\hat\rho({W})) \},
\end{equation*}
where
\begin{align*}
	h^{\txtc}(\xi,{W}): =\int\limits_{-\infty}^{0} S^{\txts}(-r) P^{\txts }
	F(U_{r}(\xi)) ~\txtd r + \int\limits_{-\infty}^{0} S^{\txts}(-r) 
	P^{\txts }G(U_{r}(\xi))~\txtd \bm{{W}}_{r}.
\end{align*}
\end{theorem}
Note that if one takes $\alpha\in(1/2,1)$, i.e.,~this corresponds to a 
fractional Bronwian motion with Hurst parameter $H\in(1/2,1)$, the above 
computations simplify, recall Remark~\ref{fbm}. Random center manifolds for It\^o/Stratonovich 
for certain classes of SDEs have also been investigated 
in~\cite{Arnold,Boxler1,Boxler2,ChenRobertsDuan,DuDuan}, yet these rely on the transformation of the corresponding SDE in a random ODE as 
shortly indicated at the beginning of Section~\ref{sectfp}. Our
results are obtained by different techniques and in the broader
class of RDEs. 

\section{Smoothness of center manifolds}
\label{smoothness}

We first point out the main arguments, which guarantee the smoothness of 
random invariant manifolds. These have been employed in~\cite[Sec.~4]{DuanLuSchmalfuss} 
and~\cite[Sec.~5]{LuSchmalfuss} for random stable/unstable manifolds and 
in~\cite{ChenRobertsDuan} for center manifolds. From the rough path point of 
view it is essential to investigate the differentiability of the It\^{o}-Lyons 
map, i.e., the map that associates a controlled rough path to the solution 
of the RDE driven by this path; see also~\cite{CoutinLejay} 
or~\cite[Sec.~8.9]{FritzHairer}. For our goals, it suffices to show the 
continuous-differentiability of the mapping $\xi\mapsto U_{\cdot}(\xi)$, 
where $\xi\in \cX$ and $U_{\cdot}(\xi)$ is the solution of~\eqref{sde1} 
as discussed in Section~\ref{rde}. We only point out the main arguments and assumptions that one needs in order to investigate the smoothness of the corresponding manifolds and drop the lengthy computations that do not introduce new ideas.

\begin{remark}
Note that $h^{\txtc}(\cdot,{W})$ is Lipschitz due to Lemma~\ref{lipfp}. In 
this section we establish that additional smoothness assumptions on 
$F$ and $G$ (i.e.~$F:\mathbb{R}^{n}\to\mathbb{R}^{n}$ is $C^{m}$ and 
$G:\mathbb{R}^{n}\to\mathbb{R}^{n\times d}$ is $C^{m+3}_{\txtb}$ for $m\geq 1$) 
lead to better regularity of $h^{\txtc}({\cdot}, W)$.
\end{remark}
We now indicate the main ideas in the classical proof of smoothness of 
invariant manifolds for S(P)DEs. In order to show that $\cM^{\txtc}_{loc}({W})$ 
obtained in Theorem~\ref{lemma:cm} is $C^{1}$ one needs to verify that 
$h^{\txtc}(\xi,{W})$ is continuously differentiable in $\xi\in \cX^{\txtc}$. 
Therefore, one has to establish the differentiability of the solution 
of the Lyapunov-Perron fixed-point problem in $\xi$. For notational simplicity, 
we firstly describe the main ideas without using the controlled rough path 
notation. \\
We consider the continuous-time Lyapunov-Perron transform associated 
to~\eqref{sde1} and do a formal computation to illustrate the main idea. 
For $\xi\in\cX^{\txtc}$ we have
\begin{align*}
U_{t}(\xi) &= S^{\txtc}(t) \xi + \int\limits_{0}^{t} 
S^{\txtc}(t-r)P^{\txtc }F(U_{r}(\xi))~\txtd r + \int\limits_{0}^{t} 
S^{\txtc}(t-r)P^{\txtc }G(U_{r}(\xi))~\txtd \bm{{W}}_{r}\\& 
+ \int\limits_{-\infty}^{t} S^{\txts}(t-r) P^{\txts }F(U_{r}(\xi)) ~\txtd r 
+ \int\limits_{-\infty}^{t} S^{\txts}(t-r) P^{\txts }G(U_{r}(\xi))~\txtd \bm{{W}}_{r}.
\end{align*}
Since we analyze the differentiability of $\xi\mapsto U_{t}(\xi)$ we have 
to investigate the difference $U_{t}(\xi)-U_{t}(\xi_{0})$, where 
$\xi_{0}\in\cX^{\txtc}$. We consider
\begin{align}
\label{repr}
U_{t}(\xi) - U_{t}(\xi_{0}) - \mathcal{T}(U_{t}(\xi) - U_{t}(\xi_{0}) ) 
=S^{\txtc}(t)(\xi-\xi_{0}) + I,
\end{align}
where
\begin{align}
\label{t}
\mathcal{T}(Z) :&=\int\limits_{0}^{t} S^{\txtc}(t-r) P^{\txtc } \txtD F(U_{r}(\xi_{0})) Z
~\txtd r + \int\limits_{0}^{t} S^{\txtc} (t-r)P^{\txts } \txtD G (U_{r}(\xi_{0})) 
Z~\txtd \bm{{W}}_{r}\\&
+ \int\limits_{-\infty}^{t} S^{\txts}(t-r)P^{\txts } \txtD F (U_{r}(\xi_{0})) Z
~\txtd r + \int\limits_{-\infty}^{t} S^{\txts}(t-r) P^{\txts }\txtD G (U_{r}(\xi_{0})) 
Z~\txtd \bm{{W}}_{r} \nonumber
\end{align}
and
\begin{align}\label{i}
I:& = \int\limits_{0}^{t} S^{\txtc}(t-r) P^{\txtc }(F(U_{r} (\xi)) 
- F(U_{r}(\xi_{0})) - \txtD F (U_{r}(\xi_{0}))( U_{r}(\xi) 
- U_{r}(\xi_{0}) ) )) ~\txtd r\\
& +  \int\limits_{0}^{t} S^{\txtc}(t-r)P^{\txtc }( G(U_{r} (\xi)) 
- G(U_{r}(\xi_{0})) - \txtD G (U_{r}(\xi_{0}))( U_{r}(\xi) 
- U_{r}(\xi_{0}) ) )) ~\txtd \bm{{W}}_{r}\nonumber\\
& +\int\limits_{-\infty}^{t} S^{\txts}(t-r)P^{\txts }( F(U_{r} (\xi)) 
- F(U_{r}(\xi_{0})) - \txtD F (U_{r}(\xi_{0}))( U_{r}(\xi) 
- U_{r}(\xi_{0}) )) ) ~\txtd r\nonumber\\
& +  \int\limits_{-\infty}^{t} S^{\txts}(t-r)P^{\txts }( G(U_{r} (\xi)) 
- G(U_{r}(\xi_{0})) - \txtD G (U_{r}(\xi_{0}))( U_{r}(\xi) 
- U_{r}(\xi_{0}) ) )) ~\txtd \bm{{W}}_{r}. \nonumber
\end{align}
Now, one derives conditions which ensure that $\|\mathcal{T}\|<1$, 
in order for $(\mbox{Id}-\mathcal{T})$ to be invertible together with $|I|=
o(|\xi-\xi_{0}|)$ as $\xi\to \xi_{0}$. Then, due to~\eqref{repr} 
one can conclude that
\begin{align}\label{ziel}
U_{t}(\xi) - U_{t}(\xi_{0}) = (\mbox{Id}- \mathcal{T})^{-1} S^{\txtc}(t)(\xi-\xi_{0}) 
+ o(|\xi-\xi_{0}|), \mbox{ as } \xi\to \xi_{0},
\end{align}
which implies that $U_{t}(\xi)$ is differentiable in $\xi$. Its derivative 
is constituted by
\begin{align*}
\txtD_{\xi}U_{t}(\xi) & = S^{\txtc}(t) + \int\limits_{0}^{t} 
S^{\txtc} (t-r) P^{\txtc }\txtD F(U_{r}(\xi))\txtD_{\xi}U_{r}(\xi) 
~ \txtd r + \int\limits_{0}^{t} S^{\txtc}(t-r)P^{\txtc } \txtD G(U_{r}(\xi)) 
\txtD_{\xi} U_{r}(\xi) ~\txtd\bm{{W}}_{r} \\
& + \int\limits_{-\infty}^{t} S^{\txts}(t-r) P^{\txts }\txtD F(U_{r}(\xi)) 
\txtD_{\xi} U_{r}(\xi)~\txtd r + \int\limits_{-\infty}^{t} 
S^{\txts} (t-r) P^{\txts }\txtD G (U_{r}(\xi)) \txtD_{\xi} U_{r}(\xi) 
~\txtd \bm{{W}}_{r}.
\end{align*}
The fact that such a formula is valid for the controlled rough integral 
introduced in Theorem \ref{gubinelliintegral} follows according to the Omega Lemma, see~\cite[Cor.~2]{CoutinLejay}~and~\cite[Thm.~8.10]{FritzHairer}.\\
To prove the continuity of the mapping $\xi\mapsto \txtD_{\xi}U_{t}(\xi)$ 
one computes for $\xi$ and $\xi_{0}\in \cX^{\txtc}$ 
\begin{align}\label{diff1}
\txtD_{\xi}U_{t}(\xi) - \txtD_{\xi}U_{t}(\xi_{0})& = 
\int\limits_{0}^{t} S^{\txtc}(t-r)P^{\txtc } ( \txtD F (U_{r}(\xi)  ) 
\txtD_{\xi}U_{r}(\xi) - \txtD F(U_{r}(\xi_{0})) \txtD_{\xi}U_{r}(\xi_{0}) ) ~\txtd r\\
& + \int\limits_{-\infty}^{t} S^{\txts} (t-r) P^{\txts } (\txtD G (U_{r}(\xi)  ) 
\txtD_{\xi}U_{r}(\xi) - \txtD G(U_{r}(\xi_{0})) \txtD_{\xi}U_{r}(\xi_{0})) 
~\txtd \bm{{W}}_{r}\nonumber\\
& = \int\limits_{0}^{t} S^{\txtc}(t-r)P^{\txtc } ( \txtD F(U_{r}(\xi) ) 
( \txtD_{\xi}U_{r}(\xi) -\txtD_{\xi}(U_{r}(\xi_{0})  )  ) ~\txtd r\nonumber\\
& + \int\limits_{-\infty}^{t} S^{\txts} (t-r) P^{\txts } ( \txtD G(U_{r}(\xi) ) 
( \txtD_{\xi}U_{r}(\xi) -\txtD_{\xi}(U_{r}(\xi_{0})  )  )  
~\txtd \bm{{W}}_{r} + \overline{I}, \nonumber
\end{align} 
where
\begin{align*}
\overline{I}&=\int\limits_{0}^{t} S^{\txtc} (t-r) P^{\txtc }(\txtD 
F(U_{r}(\xi)) -\txtD F(U_{r} (\xi_{0})) \txtD_{\xi}U(\xi_{0}))  ~\txtd r\\
& + \int\limits_{-\infty}^{t} S^{\txts} (t-r)P^{\txts } (\txtD G(U_{r}(\xi)) 
-\txtD G(U_{r} (\xi_{0})) \txtD_{\xi}U(\xi_{0}) ) ~\txtd \bm{W}_{r}.
\end{align*}

For further details and properties of flows associated to rough 
differential equations, see~\cite{B, CoutinLejay} and the references 
specified therein. Regarding all these preliminary considerations, one can show the following result. The proof relies on similar computations to the proof of Theorem~\ref{contraction} and~\cite[Sec.~4]{DuanLuSchmalfuss}.

\begin{theorem}
\label{glattheit} Assume that $F$ is $C^{m}$ and $G$ is $C^{m+3}_{\txtb}$ 
for $m\geq 1$. If $-\beta< m\eta<\gamma$ and
\begin{align*}
	 K  \left( \frac{ \txte^{\beta+j\eta}(C_{S}M_{s}\txte^{-\eta j}
	+1)}{1-\txte^{-(\beta+\eta j)}} +  \frac{\txte^{\gamma-\eta j} 
	(C_{S}M_{c}\txte^{-\eta j}+1)}{1-\txte^{-(\gamma-\eta j)}} \right) <1,
	~\mbox{ for } 1\leq j\leq m,
	\end{align*}
	then $\cM^{\txtc}_{loc}({W})$ is a local $C^{m}$-center manifold.
\end{theorem}

\appendix
\section{Properties of the discrete Lyapunov-Perron map}
\label{a}

We start by pointing out following general technique, which is required 
in order to prove the invariance of a manifold $\widetilde{\cM}^{\txtc}$ 
for a continuous-time random dynamical system $\widetilde{\varphi}$, i.e.
\benn
\widetilde{\varphi}(T,{W},\widetilde{\cM}^{\txtc}({W}))\subset 
\widetilde{\cM}^{\txtc}(\Theta_t{W}). 
\eenn
Again, $\widetilde{\cM}^{\txtc}({W})$ is given by the graph of a 
function $\widetilde{h}(\cdot,{W}):\cX^{\txtc}\to \cX^{\txts}$, i.e., 
$\widetilde{\cM}^{\txtc}({W}):=\{\widetilde{\xi}^{\txtc} 
+ \widetilde{h} (\widetilde{\xi}^{\txtc},{W}) \mbox{ : } 
\widetilde{\xi}^{\txtc} \in \cX^{\txtc} \}$, where 
$\widetilde{h}(\widetilde{\xi}^{\txtc},{W}):=P^{\txts}
\widetilde{\Gamma}(\widetilde{\xi}^{\txtc},{W})[0]$.
Here we denote with $\widetilde{\Gamma}(\xi^{\txtc},{W})[\cdot]$ the 
fixed-point of the corresponding Lyapunov-Perron map, where $\cdot$ 
stands only for the time-variable. For further details, 
see~\cite[Lem.~5.5]{LuSchmalfuss} or~\cite[Lem.~4.7]{GarridoLuSchmalfuss}, 
where similar computations are performed. 

Typically, one shows that for $\widetilde{\xi}^{\txtc}\in \cX^{\txtc}$ 
small enough it holds that $P^{\txtc}\widetilde{\Gamma}(\xi^{\txtc},{W})[0]
\in \widetilde{M}^{\txtc}({W})$. To justify the invariance property one 
has to infer that $\widetilde{\varphi}(T,{W},\widetilde{\Gamma}(\xi^{\txtc},
{W})[0])\in \widetilde{M}^{\txtc}(\Theta_t{W})$. To this aim, regarding the 
structure of the manifold, one needs to analyze $\widetilde{h}$ on the 
$\Theta_t{W}$-fiber replacing $\widetilde{\xi}^{\txtc}$ with 
$P^{\txtc}\widetilde{\varphi}(T,{W},\widetilde{\Gamma}(\xi^{\txtc},{W})[0]) $, 
more precisely $\widetilde{h}(P^{\txtc}\widetilde{\varphi}(T,{W},
\widetilde{\Gamma}(\widetilde{\xi}^{\txtc},{W})[0]), \Theta_t{W} )$.

Therefore one needs to derive an expression for the fixed-point of the 
Lyapunov-Perron transform on the $\Theta_t{W}$-fiber replacing 
$\widetilde{\xi}^{\txtc}$ by $P^{\txtc}\widetilde{\varphi}(T,{W},
\widetilde{\xi}^{\txtc}) $, i.e., $\widetilde{\Gamma}(P^{\txtc}
\widetilde{\varphi}(T,{W},\widetilde{\xi}^{\txtc}),\Theta_t{W})[\cdot]$, 
compare~\eqref{gammai}. Hence, it is enough to prove that
\begin{align}
\label{gammat}
\widetilde{\Xi}_{T}(\sigma,{W})=
\begin{cases}
\widetilde{\Gamma}(\widetilde{\xi}^{\txtc},{W})[\sigma+T] &: \sigma<T\\
\widetilde{\varphi}(\sigma+T,{W},\widetilde{\Gamma}(
\widetilde{\xi}^{\txtc},{W})[0]) &: \sigma\in[-T,0]
\end{cases}
\end{align}
is the fixed-point of $\widetilde{\Gamma}( P^{\txtc}\widetilde{\varphi}(
T,{W},\widetilde{\xi}^{\txtc}) ,\Theta_t{W})[\sigma]$,
see~\cite[Lem.~5.5]{LuSchmalfuss}. This implies
\begin{align}
\label{fpappendix}
P^{\txts}\widetilde{\Xi}_{T}(0,{W}) = \widetilde{h}(P^{\txtc}
\widetilde{\varphi}(T,{W},\widetilde{\Gamma}(\widetilde{\xi}^{\txtc},
{W})[0]),  \Theta_t{W}),
\end{align}
which yields that $\widetilde{\varphi}(T,{W},\widetilde{\Gamma}(
\widetilde{\xi}^{\txtc},{W})[0])\in \widetilde{M}^{\txtc}(\Theta_t{W})$.
For completeness we now derive the analogue of~\eqref{gammat} for 
discrete-time random dynamical systems, compare~\eqref{gammai}. We 
indicate here the computation, which allows us to express the fixed-point 
of $J_{d}$ as introduced in Section~\ref{sectfp}, when shifting the fiber 
of the noise. More precisely, we investigate a connection between the 
fixed-points of $J_{R,d}({W},\cdot)$ and $J_{R,d}(\Theta_{i}{W},\cdot)$ 
for $i\in\mathbb{Z}_{-}$. Let $\varphi_R(\cdot,W,\cdot)$ denote the solution of~\eqref{a1} on $[0,1]$.

\begin{lemma}
\label{invfp}
Let $\xi^{\txtc}_{-1}:= P^{\txtc} \Gamma(\xi^{\txtc},{W})[-1,0]$. 
The fixed-point $\Gamma(\xi^{\txtc},{W})$ of $J_{R,d}({W},\cdot,\xi^{\txtc})$ 
can be expressed by
\begin{align}\label{fpomega}
		\Xi(\xi^{\txtc},{W})[i,\cdot]:=\begin{cases}
		\Gamma(\xi^{\txtc}_{-1},\Theta_{-1}{W})[i+1,\cdot], 
		& \mbox{if } i=-2,-3,\ldots\\
			\varphi_{R}(\cdot,\Theta_{-1}{W}, 
			\Gamma(\xi^{\txtc}_{1},\Theta_{-1}{W} )[-1,1] ), &\mbox{if } i=-1.
		\end{cases}
\end{align}
\end{lemma}
	
\begin{proof}
To prove the statement we are going to show that $\Xi(\xi^{\txtc},{W})$ is a 
fixed-point of $J_{R,d}({W},\cdot,\xi^{\txtc})$. Due to the uniqueness of the 
fixed-point we may then infer that 
\benn
\Xi(\xi^{\txtc},{W})=\Gamma(\xi^{\txtc},{W}).
\eenn	
Since $\Gamma(\xi^{\txtc}_{-1},\Theta_{-1}{W} )$ is the fixed-point 
of $J_{R,d}(\Theta_{-1}{W},\cdot,\xi^{\txtc}_{-1})$ we explore~\eqref{j}. For notational simplicity we supress the Gubinelli derivative in the following computation, i.e~we only use the expression for $J^{1}_{R,d}(\cdot,\cdot,\cdot)$. We also recall that for a fixed $k$, $\Gamma(\xi,W)[k,\cdot]\in D^{2\alpha}_W([0,1];\cX)$. Consequently, $F_R(\Gamma(\xi,W)[k,\cdot])$ indicates that $F_R$ is applied to the path $\Gamma(\xi,W)[k,\cdot]$. \\

On the stable part we compute
\begin{flalign*}
 	\Gamma^{\txts}(\xi^{\txtc}_{-1}, \Theta_{-1}{W} )[-1,1]=
	\sum\limits_{k=-\infty}^{-1} S^{\txts}(-k) &\Bigg (
	\int\limits_{0}^{1} S^{\txts}(1-r) P^{\txts }F_{R}(\Gamma(\xi^{\txtc}_{1},
	\Theta_{-1}{W} )[k-1,\cdot])(r) ~\txtd r \\
 	& + \int\limits_{0}^{1} S^{\txts}(1-r)P^{\txts } G_{R}(\Gamma(\xi^{\txtc}_{1},
	\Theta_{-1}{W} )[k-1,\cdot])(r) ~\txtd \Theta_{k-1}\Theta_{-1} \bm{{W}}_{r} \Bigg)
 \end{flalign*}
 \vspace*{-8 mm}
 \begin{align*}
 	& +\int\limits_{0}^{1} S^{\txts}(1-r)P^{\txts } F_{R}(\Gamma(\xi^{\txtc}_{1},
	\Theta_{-1}{W} )[-1,\cdot])(r) ~\txtd r\\
 	& + \int\limits_{0}^{1} S^{\txts} (1-r)P^{\txts } G_{R}(\Gamma(\xi^{\txtc}_{1},
	\Theta_{-1}{W} )[-1,\cdot])(r) ~\txtd \Theta_{-1}\Theta_{-1}{W}
 	\end{align*}
 	\vspace*{-5 mm}
 	\begin{align*}
 \hspace*{30 mm}	 = \sum\limits_{k=-\infty}^{0} S^{\txts}(-k) &
\Bigg (\int\limits_{0}^{1} S^{\txts}(1-r)P^{\txts } F_{R}(\Gamma(\xi^{\txtc}_{1},
\Theta_{-1}{W} )[k-1,\cdot])(r) ~\txtd r \\
 	& + \int\limits_{0}^{1} S^{\txts}(1-r)P^{\txts } G_{R}(\Gamma(\xi^{\txtc}_{1},
	\Theta_{-1}{W} )[k-1,\cdot])(r) ~\txtd \Theta_{k-1}\Theta_{-1} \bm{{W}}_{r} \Bigg).
\end{align*}
Using~\eqref{fpomega} for $i=-1$ we have on the stable component
\begin{flalign*}
&\hspace*{-75 mm} \Xi^{\txts}(\xi^{\txtc},{W})[-1,t]=
\varphi^{\txts}_{R}(t,\Theta_{-1},{W}, \Gamma(\xi^{\txtc}_{-1},\Theta_{-1}{W} )[-1,1])
\end{flalign*}
\begin{flalign*}
 = \sum\limits_{k=-\infty}^{0} S^{\txts}(t-k) &\Bigg( \int\limits_{0}^{1} P^{\txts }
F_{R}( \Gamma(\xi^{\txtc}_{-1},\Theta_{-1}{W} )[k-1,\cdot] )(r)~\txtd r \\
 &+ \int\limits_{0}^{t}  P^{\txts }G_{R}( \Gamma(\xi^{\txtc}_{-1},\Theta_{-1}{W} )[k-1,\cdot] )(r) 
~\txtd \Theta_{k-1}\Theta_{-1}\bm{{W}}_{r} \Bigg)
 \end{flalign*}
 \begin{flalign*}
 & + \int\limits_{0}^{t}S^{\txts}(t-r)P^{\txts }  F_{R}( \Xi(\xi^{\txtc}_{-1},
\Theta_{-1}{W} )[-1,\cdot] )(r)~\txtd r + \int\limits_{0}^{t} S^{\txts}(t-r) P^{\txts }
G_{R}( \Xi(\xi^{\txtc}_{-1},\Theta_{-1}{W} )[-1,\cdot] )(r) ~\txtd \Theta_{-1}\bm{{W}}_{r}\\
 &  = \sum\limits_{k=-\infty}^{-1} S^{\txts}(t-k-1)  \Bigg( 
\int\limits_{0}^{1} P^{\txts }F_{R}( \Gamma(\xi^{\txtc}_{-1},\Theta_{-1}{W} )[k,\cdot] )(r)
~\txtd r + \int\limits_{0}^{t}P^{\txts }  G_{R}( \Gamma(\xi^{\txtc}_{-1},\Theta_{-1}{W} )
[k,\cdot] )(r) ~\txtd \Theta_{k-1}\bm{{W}}_{r} \Bigg)\\ 
& + \int\limits_{0}^{t}S^{\txts}(t-r) P^{\txts } F_{R}( \Xi(\xi^{\txtc}_{-1},
\Theta_{-1}{W} )[-1,\cdot] )(r)~\txtd r + \int\limits_{0}^{t} S^{\txts}(t-r) P^{\txts }G_{R}( 
\Xi(\xi^{\txtc}_{-1},\Theta_{-1}{W} )[-1,\cdot] )(r) ~\txtd \Theta_{-1}\bm{{W}}_{r}
\end{flalign*}
\vspace*{-4 mm}
\begin{flalign*}
\hspace*{-60 mm} =\sum\limits_{k=-\infty}^{-1} S^{\txts}(t-k-1) 
&\Bigg(\int\limits_{0}^{1}S^{\txts}(1-r)P^{\txts }  F_{R}( \Xi(\xi^{\txtc},{W} )
[k-1,\cdot] )(r)~\txtd r\\
 & + \int\limits_{0}^{1} S^{\txts}(1-r) P^{\txts }G_{R}( \Xi(\xi,{W} )[k-1,\cdot] ) (r)
~\txtd \Theta_{k-1}\bm{{W}}_{r}\Bigg)
\end{flalign*}
 
\begin{flalign*}
&+ \int\limits_{0}^{t}S^{\txts}(t-r)P^{\txts }  F_{R}( \Xi(\xi^{\txtc}_{-1},
\Theta_{-1}{W} )[-1,\cdot] )(r)~\txtd r + \int\limits_{0}^{t} S^{\txts}(t-r) P^{\txts }
G_{R}( \Xi(\xi^{\txtc}_{-1},\Theta_{-1}{W} )[-1,\cdot] )(r) ~\txtd \Theta_{-1}
\bm{{W}}_{r},\\
\end{flalign*}
where in the last step we use again~\eqref{fpomega}. Furthermore, regarding 
that $\Gamma(\xi^{\txtc}_{-1},\Theta_{-1}{W} )$ is the fixed-point of 
$J_{R,d}(\Theta_{-1}{W},\cdot, \xi^{\txtc}_{-1})$ we now compute analogously 
to the first step from~\eqref{j} for $i=-2,-3,\ldots$
\begin{flalign*}
 \Gamma^{\txts} (\xi^{\txtc}_{-1},\Theta_{-1}{W})[i+1,t]=
\sum\limits_{k=-\infty}^{i+1} S^{\txts}(t+i+1-k)&
 \Bigg (\int\limits_{0}^{1} S^{\txts}(1-r)P^{\txts } F_{R}(
\Gamma(\xi^{\txtc}_{1},\Theta_{-1}{W} )[k-1,\cdot])(r) ~\txtd r \\
& \hspace*{-3 mm}+ \int\limits_{0}^{1} S^{\txts}(1-r)P^{\txts } G_{R}(
\Gamma(\xi^{\txtc}_{1},\Theta_{-1}{W} )[k-1,\cdot])(r) ~\txtd 
\Theta_{k-1}\Theta_{-1} \bm{{W}}_{r} \Bigg)
\end{flalign*}
\vspace*{-6 mm}
\begin{flalign*}
& +\int\limits_{0}^{t}S^{\txts}P(t-r)P^{\txts }  F_{R}(\Gamma(\xi^{\txtc}_{1},
\Theta_{-1}{W} )[i+1,\cdot])(r) ~\txtd r\\
& + \int\limits_{0}^{t}S^{\txts }(t-r)P^{\txts }  F_{R}(\Gamma(\xi^{\txtc}_{1},
\Theta_{-1}{W} )[i+1,\cdot])(r) ~\txtd\Theta_{i+1}\Theta_{-1}\bm{{W}}_ {r}
 \end{flalign*}
 \vspace*{-6 mm}
 \begin{flalign*}
 \hspace*{30 mm}=\sum\limits_{k=-\infty}^{i} S^{\txts}(t+i-k) &
 \Bigg (\int\limits_{0}^{1} S^{\txts}(1-r)P^{\txts } F_{R}(\Gamma(
\xi^{\txtc}_{1},\Theta_{-1}{W} )[k,\cdot])(r) ~\txtd r \\
 & + \int\limits_{0}^{1} S^{\txts}(1-r) P^{\txts }G_{R}(\Gamma(
\xi^{\txtc}_{1},\Theta_{-1}{W} )[k,\cdot])(r) ~\txtd \Theta_{k-1} \bm{{W}}_{r} \Bigg)
 \end{flalign*}
 \vspace*{-6 mm}
 \begin{flalign*}
 & +\int\limits_{0}^{t}S^{\txts}(t-r)P^{\txts }  F_{R}(\Gamma(\xi^{\txtc}_{1},
\Theta_{-1}{W} )[i+1,\cdot])(r) ~\txtd r\\
 & + \int\limits_{0}^{t}S^{\txts}(t-r)  P^{\txts }G_{R}(\Gamma(\xi^{\txtc}_{1},
\Theta_{-1}{W} )[i+1,\cdot])(r) ~\txtd\Theta_{i}\bm{{W}}_ {r}.
 \end{flalign*}
 On the other hand~\eqref{fpomega} gives us on the stable part
 \begin{align*}
 \Xi^{\txts}(\xi^{\txtc},{W})[i,t] = \sum\limits_{k=-\infty}^{i} 
S^{\txts} (t+i-k) &
 \Bigg (\int\limits_{0}^{1} S^{\txts}(1-r)P^{\txts } F_{R}(\Xi(\xi^{\txtc},
{W} )[k-1,\cdot])(r) ~\txtd r \\
 & + \int\limits_{0}^{1} S^{\txts}(1-r)P^{\txts } G_{R}(\Xi(\xi^{\txtc},{W} )
[k-1,\cdot])(r) ~\txtd \Theta_{k-1} \bm{{W}}_{r} \Bigg)
 \end{align*}
 \vspace*{-0.4 cm}
 \begin{align*}
 & +\int\limits_{0}^{t} S^{\txts}(t-r) P^{\txts }F_{R}(\Xi(\xi^{\txtc},{W} )
[i,\cdot])(r) ~\txtd r \\
 & + \int\limits_{0}^{t} S^{\txts}(t-r)P^{\txts } G_{R}(\Xi(\xi^{\txtc},{W} )
[i,\cdot])(r) ~\txtd \Theta_{i} \bm{{W}}_{r}.
 \end{align*}
In order to prove the assertion one carries out a similar computation for the 
center component. One infers that

\begin{flalign*}
\Gamma(\xi^{\txtc},{W})[i,t] : = S^{\txtc}(t+i) \xi^{\txtc} -\sum\limits_{k=0}^{i+2} S^{\txtc} (t+i-k) &\Big(\int\limits_{0}^{1} 
S^{\txtc} (1-r) P^{\txtc }F_{R}(\Gamma(\xi^{\txtc},{W})[k-1,\cdot])(r) ~\txtd r \\
&+ \int\limits_{0}^{1} S^{\txtc}(1-r)P^{\txtc }G_{R}(\Gamma(\xi^{\txtc},{W})[k-1,\cdot])(r) 
~\txtd \Theta_{k-1} \bm{{W}}_{r}  \Big)\nonumber
\end{flalign*}
\begin{flalign*}
& \hspace*{10 mm}- \int\limits_{t}^{1} S^{\txtc} (t-r)P^{\txtc }F_{R}(\Gamma(\xi^{\txtc},{W})[i,\cdot])(r) 
~\txtd r - \int\limits_{t}^{1} S^{\txtc} (t-r)P^{\txtc } G_{R}(\Gamma(\xi^{\txtc},{W})
[i,\cdot])(r) ~\txtd \Theta_{i} \bm{{W}}_{r}
\end{flalign*}
\begin{flalign*}
 + \sum\limits_{k=-\infty}^{i} S^{\txts} (t+i-k) &\Big(\int\limits_{0}^{1} 
S^{\txts} (1-r) P^{\txts }F_{R}(\Gamma(\xi^{\txtc},{W})[k-1,\cdot])(r) ~\txtd r \\
 &+ \int\limits_{0}^{1} S^{\txts}(1-r)P^{\txts }G_{R}(\Gamma(\xi^{\txtc},{W})[k-1,\cdot])(r) 
~\txtd \Theta_{k-1} \bm{{W}}_{r}  \Big)\nonumber
 \end{flalign*}
 \vspace*{-0.4 cm}
 \begin{flalign*}
&\hspace*{10 mm} + \int\limits_{0}^{t} S^{\txts} (t-r)P^{\txts }F_{R}(\Gamma(\xi^{\txtc},{W})[i,\cdot]) (r)
~\txtd r + \int\limits_{0}^{t} S^{\txts} (t-r) P^{\txts }G_{R}(\Gamma(\xi^{\txtc},{W})
[i,\cdot])(r) ~\txtd \Theta_{i} \bm{{W}}_{r}.
\end{flalign*}
This shows that $\Xi(\xi^{\txtc},{W})$ is the fixed-point of $J_{R,d}({W},\cdot,
\xi^{\txtc})$. Due to uniqueness we have that $\Xi(\xi^{\txtc},{W})=
\Gamma(\xi^{\txtc},{W})$. This proves the statement. 
\end{proof}

\section{Exponential trichotomy}
\label{b} 

We shortly indicate how the results obtained in Section~\ref{lpm} can 
be applied if there additionally exists an unstable subspace. In that 
case one assumes that there exist constants $\rho_{1}>\rho_{2}\geq 0\geq 
-\rho_{2}>-\rho_{3}$ such that
\begin{align*}
& |S(t)P^{\txtc} x | \leq M_{c} \txte^{\rho_{2} |t|} |P^{\txtc}x|, 
~~ \mbox{  for } t\in\mathbb{R} \mbox{ and } x\in X;\\
& |S(t)P^{u} x| \leq  M_{u} \txte^{\rho_{1} t} |P^{u}x|, 
~~~\mbox{ for } t\leq 0 \mbox{ and } x\in X;\\
& |S(t)P^{\txts}x| \leq M_{s} \txte^{-\rho_{3} t} |P^{\txts}x|, 
~~\mbox{ for } t\geq 0 \mbox{ and } x\in X,
\end{align*}
holds true, compare~\cite[Sec.~7.1.2]{SellYou}.
Then, the continuous-time Lyapunov-Perron map given by
\begin{align*}
J({W}, U,\xi)[\tau] &: = S^{\txtc}(\tau) \xi^{\txtc} + 
\int\limits_{0}^{\tau} S^{\txtc} (\tau-r)P^{\txtc }\overline{F}(U)(r) ~\txtd r + 
\int\limits_{0}^{\tau} S^{\txtc} (\tau-r)P^{\txtc } \overline{G}(U)(r) ~\txtd \bm{{W}}_r\\
& +\int\limits_{-\infty}^{\tau} S^{\txts} (\tau-r) P^{\txts }\overline{F}(U)(r) ~\txtd r 
+ \int\limits_{-\infty}^{\tau} S^{\txts} (\tau-r) P^{\txts }\overline{G}(U)(r) ~\txtd \bm{{W}}_{r}\\
& - \int\limits_{\tau}^{\infty} S^{u} (\tau-r)P^{\txtu } \overline{F}(U)(r) ~\txtd r 
- \int\limits_{\tau}^{\infty} S^{u} (\tau-r) P^{\txtu }\overline{G}(r) ~\txtd \bm{{W}}_{r},
\end{align*}
has to be discretized for $\tau\in\mathbb{Z}$. In this case, we introduce 
for $\eta>0$ satisfying $\rho_{2}<\eta<\min\{\rho_{1},\rho_{3}\}$ the space
\begin{align*}
||\mathbb{U}||_{BC^{\eta} (D^{2\alpha}_{W})}:=\sup\limits_{i\in\mathbb{Z}} 
\txte^{-\eta |i| } || U^{i-1}, (U^{i-1})'||_{D^{2\alpha}_{W}}.
\end{align*}
Proceeding exactly as in Subsection~\ref{dlp} leads to the following 
definition of a discrete Lyapunov-Perron transform
$J_{d}({W},\mathbb{U},\xi):=(J^{1}_{d}({W},\mathbb{Y},\xi), 
J^{2}_{d}({W},\mathbb{U},\xi))$ for a sequence $\mathbb{U}\in 
BC^{\eta}(D^{2\alpha}_{W})$, $t\in[0,1]$, $W\in\Omega_{W}$ and $i\in\mathbb{Z}$:
\begin{align*}
&J^{1}_{d}({W}, \mathbb{U},\xi)[i-1,t] : = S^{\txtc}(t+i-1) 
\xi^{\txtc} \\
&  -\sum\limits_{k=0}^{i+1} S^{\txtc} (t+i-1-k) \left(
\int\limits_{0}^{1} S^{\txtc} (1-r)P^{\txtc } \overline{F}(U^{k-1})(r) ~\txtd r + 
\int\limits_{0}^{1} S^{\txtc}(1-r)P^{\txtc }\overline{G}(U^{k-1})(r) ~\txtd \Theta_{k-1} 
\bm{{W}}_{r}  \right)\nonumber\\
& - \int\limits_{t}^{1} S^{\txtc} (t-r)P^{\txtc }\overline{F}(U^{i-1})(r) ~\txtd r 
- \int\limits_{t}^{1} S^{\txtc} (t-r)P^{\txtc }\overline{G}(U^{i-1})(r) ~\txtd 
\Theta_{i-1} \bm{{W}}_{r}\nonumber\\
& + \sum\limits_{k=-\infty}^{i-1} S^{\txts} (t+i-1-k) 
\left(\int\limits_{0}^{1} S^{\txts} (1-r)P^{\txts } \overline{F}(U^{k-1})(r) 
~\txtd r + \int\limits_{0}^{1} S^{\txts}(1-r)P^{\txts }\overline{G}(U^{k-1})(r) 
~\txtd \Theta_{k-1} \bm{{W}}_{r}  \right)\nonumber\\
& + \int\limits_{0}^{t} S^{\txts} (t-r)P^{\txts }\overline{F}(U^{i-1})(r) ~\txtd r 
+ \int\limits_{0}^{t} S^{\txts} (t-r) P^{\txts }\overline{G}(U^{i-1})(r) ~\txtd 
\Theta_{i-1} \bm{{W}}_{r}\\
& + \int\limits_{0}^{t} S^{u} (t-r) P^{\txtu }\overline{F}(U^{i-1})(r)~\txtd r 
+ \int\limits_{0}^{t} S^{u}(t-r)P^{\txtu } \overline{G}(U^{i-1})(r)~\txtd 
\Theta_{i-1}\bm{{W}}_{r}\\
& +\sum\limits_{k=i-1}^{\infty} S^{u}(t+i-1-k)\left( 
\int\limits_{0}^{1}S^{u}(1-r) P^{\txtu }\overline{F}(U^{k-1})(r)~\txtd r 
+ \int\limits_{0}^{1} S^{u}(1-r)P^{\txtu }\overline{G}(U^{k-1})(r) ~\txtd\Theta_{k-1}\bm
{W}_{r} \right).
\end{align*}
Of course, $J^{2}_{d}({W},\mathbb{U},\xi):=(J^{1}_{d}({W},\mathbb{Y},\xi))'$.
By the same arguments employed in Theorem~\ref{contraction} one can infer that 
the gap condition is now given by

\begin{align*}
K  \left( \frac{ \txte^{\eta-\rho_{2}}(C_{S}M_{c}\txte^{-\eta}+1)}{1-
\txte^{-(\eta-\rho_{2})}} +  \frac{\txte^{\rho_{1}-\eta} (C_{S}M_{u}
\txte^{-\eta}+1)}{1-\txte^{-(\rho_{1}-\eta)}} + \frac{\txte^{\rho_{3}-
\eta} (C_{S}M_{s}\txte^{-\eta}+1)}{1-\txte^{-(\rho_{3}-\eta)}} \right) <\frac{1}{4}.
\end{align*}
Hence, using our techniques one can also prove the existence of center-unstable, 
respectively center-stable manifolds for rough differential equations following 
the steps as demonstrated above.


\end{document}